\documentclass[a4paper]{article}

\usepackage[english]{babel}
\usepackage[utf8x]{inputenc}
\usepackage[T1]{fontenc}

\usepackage[a4paper,top=2cm,bottom=2.5cm,left=2.3cm,right=2.3cm,marginparwidth=1.75cm,margin=1.0in]{geometry}

\usepackage[colorinlistoftodos]{todonotes}
\usepackage[colorlinks=true, allcolors=blue]{hyperref}
\usepackage{xy}
\usepackage[affil-it]{authblk}
\usepackage{amssymb}
\usepackage{cancel}
\usepackage{epstopdf}
\usepackage{verbatim}
\usepackage{amsmath,amscd}
\usepackage{graphicx}
\usepackage{subcaption}
\usepackage{mathtools}
\usepackage{amsthm}
\usepackage{mathrsfs}
\usepackage{tikz}
\usetikzlibrary{cd}
\usepackage{pgf}
\usepackage{mathtools}
\usepackage[noend]{algpseudocode}
\usepackage{algorithm}

\usepackage{mathrsfs}
\usetikzlibrary{arrows}
\usepackage[toc,page]{appendix}
\usepackage{caption}
\usepackage[colorinlistoftodos]{todonotes}

\makeatletter
\newcommand{\subjclass}[2][2010]{%
  \let\@oldtitle\@title%
  \gdef\@title{\@oldtitle\footnotetext{#1 {Mathematics subject classification:} #2}}%
}

\newcommand\blfootnote[1]{%
  \begingroup
  \renewcommand\thefootnote{}\footnote{#1}%
  \addtocounter{footnote}{-1}%
  \endgroup
}

\newtheorem{theorem}{Theorem}[section]

\makeatletter
\newtheorem*{rep@theorem}{\rep@title}
\newcommand{\newreptheorem}[2]{%
\newenvironment{rep#1}[1]{%
 \def\rep@title{#2 \ref{##1}}%
 \begin{rep@theorem}}%
 {\end{rep@theorem}}}
\makeatother

 \newreptheorem{theorem}{Theorem}

\newtheorem{lemma}[theorem]{Lemma}
\newtheorem{pro}[theorem]{Proposition}

\theoremstyle{definition}
\newtheorem{mydef}[theorem]{Definition}
\newtheorem{rem}[theorem]{Remark}
\newtheorem{exa}[theorem]{Example}
\newtheorem{cor}[theorem]{Corollary}

\title{{\bf{A generalization of the Newton-Puiseux algorithm for coverings of semistable models}}} 
\author{\large{Paul Alexander Helminck}}
\affil{Durham University\\ 
\vspace{0.3cm}
Department of Mathematics}%


\date{}
\begin{document}
\maketitle
\definecolor{qqqqff}{rgb}{0,0,1}
\begin{abstract}

In this paper we give an algorithm that calculates the skeleton of a tame covering of curves over a complete discretely valued field. 
The algorithm relies on the {{tame simultaneous semistable reduction theorem}}, for which we give a short proof. To use this theorem in practice, we show that we can find extensions of chains of prime ideals in normalizations using compatible power series.
This allows us to reconstruct the skeleton of the covering. In studying the connections between power series and extensions of prime ideals, we obtain generalizations of classical theorems from number theory such as the Kummer-Dedekind theorem and Dedekind's theorem for cycles in Galois groups.

\end{abstract}

\blfootnote{The author was funded by the UKRI grant "{\it{Computational tropical geometry and its applications}}" with reference number MR/S034463/1.}

\section{\large{Introduction}}

Let $f(x,y)\in\mathbf{Q}[x,y]$ be a non-constant polynomial. The classical Newton-Puiseux algorithm calculates local parametrizations of the branches of the curve defined by $f(x,y)=0$ using Newton polygons and residual approximations. The goal of this paper is to give a generalization of this algorithm for coverings of semistable models. The input is a polynomial $f(x,y)$ defined over %
a complete discretely valued field $K$  %
together with a tame morphism from the curve $X$ defined by $f(x,y)=0$ to %
 the projective line $Y=\mathbf{P}^{1}_{K}$. The reader can think of $K=\mathbf{Q}((t))$ and the map $(x,y)\mapsto{x}$ here. The branch locus $D$ of this morphism
  gives rise to a natural semistable model %
$\mathcal{Y}$ of $\mathbf{P}^{1}_{K}$ whose dual intersection graph $\Sigma(\mathcal{Y})$ is the tropical tree associated to $D$. In Theorem \ref{MainThm1}, we show that  
the morphism $X\to{Y}$ extends to a morphism of semistable models $\mathcal{X}\to\mathcal{Y}$ after a finite tame extension of $K$. This induces a morphism of dual intersection graphs $\Sigma(\mathcal{X})\to\Sigma(\mathcal{Y})$ which we call the skeleton or tropicalization of $\mathcal{X}\to\mathcal{Y}$. The output of our algorithm is then exactly this skeleton. 
To find it, 
we first calculate power series expansions of the roots of $f$ %
at the edges and vertices of $\Sigma(\mathcal{Y})$  %
using generalizations %
of the %
Newton-Puiseux algorithm. We then show that we can reconstruct $\Sigma(\mathcal{X})$ by patching these local expansions over adjacent edges and vertices.
This in particular implies that we do not 
need to calculate the full normalization of $\mathcal{Y}$ in $K(X)$ in order  
to find $\Sigma(\mathcal{X})$.

The main theoretical tool behind the algorithm is a generalization of the tame simultaneous semistable reduction theorem given in \cite[Theorem 2.3]{liu1}. %
\begin{theorem}\label{MainThm1}{\bf{[Tame simultaneous semistable reduction theorem]}}
Let $\phi:(X,D_{X})\rightarrow{(Y,D_{Y})}$ be a finite separable morphism of smooth proper geometrically connected marked curves %
and let $\mathcal{Y}/R$ be a semistable model for $(Y,D_{Y})$. %
Suppose that the induced morphism of normalizations $\mathcal{X}\to\mathcal{Y}$ is tame in codimension one. There is then a finite Kummer extension $R\subset{R'}$ such that the normalized base change of $\mathcal{X}$ with respect to $R'$ %
is semistable. %
If $\mathcal{Y}$ is strongly semistable, then the normalized base change of $\mathcal{X}$ with respect to $R'$ %
is also strongly semistable.
\end{theorem}

Theorem \ref{MainThm1} generalizes \cite[Theorem 2.3]{liu1} in two ways. First, the covering is not assumed to be Galois. Moreover, if we write $G$ for the Galois group of the covering, then our tameness assumption is less strict than $p\nmid{|G|}$. %
A closer analogue of Theorem \ref{MainThm1} can be found in \cite[Theorem 3.1]{DecompFund1}, which proves the same result but in the context of Berkovich spaces. To see the resemblance, we invite the reader to compare the notion of residual tameness used in \cite{DecompFund1} to the tameness condition imposed here. %

As a first step towards proving Theorem \ref{MainThm1}, we establish a semistable generalization of the classical {{Newton-Puiseux}} theorem, which says that the field of Puiseux series $\mathcal{P}$ over $\mathbf{C}$ is algebraically closed. Our generalization works with the ring $A=R[[u,v]]/(uv-\pi^{n})$, which is the completed local ring of an ordinary double point on a semistable model. By adjoining the $m$-th roots of $u$ and $v$ to $K(A)$ for all $m$ coprime to the residue characteristic of $R$, we then obtain the tame Kummer field $K(W^{\mathrm{tame}})$. 
This is the analogue of the field of Puiseux series for $K(A)$. The corresponding Newton-Puiseux theorem is then as follows.  

\begin{theorem}\label{NPTheorem}{\bf{[Semistable Newton-Puiseux theorem]}}
Let $K(W^{\mathrm{tame}})$ be the tame Kummer field over $K(A)$. 
Then $K(W^{\mathrm{tame}})$ is the composite of all extensions of $K(A)$ that are at most tamely ramified over the special fiber and \'{e}tale over all generic primes. %

\end{theorem}

We give a proof here using well-known results from commutative algebra. This simplifies some of the underlying principles of the proof given in \cite[Page 316, Corollaire 5.3]{SGA1}. We also prove an analogous theorem for $A=R[[u]]$, which can be viewed as the completed local ring of a smooth point on a model of a curve. By combining these two results we then quickly obtain a proof of Theorem \ref{MainThm1}.  

In the second part of this paper, we put Theorems \ref{MainThm1} and \ref{NPTheorem} into practice in the form of an algorithm. %
We assume for simplicity that we are given a plane curve determined by a bivariate polynomial $f(x,y)$ and that the covering of curves in Theorem \ref{MainThm1} is given by the normalization of the birational map $(x,y)\mapsto{x}$.  %
We write %
$K(x)\subset{K(x)[y]/(f)}=K(X)$ for the extension of function fields associated to this covering. %
If we take a strongly semistable model $\mathcal{Y}$ of $(\mathbf{P}^{1}_{K},D)$ where $D$ is the branch locus of $\phi$, then by Theorem \ref{MainThm1} the morphism $\phi$ lifts to a morphism of semistable models $\mathcal{X}\to\mathcal{Y}$ after a finite Kummer extension of $K$.  
For explicit local versions of the models $\mathcal{Y}$, we refer the reader to Section \ref{P1Models}. The morphism $\mathcal{X}\to\mathcal{Y}$ gives rise to a morphism $\Sigma(\mathcal{X})\to\Sigma(\mathcal{Y})$ of dual intersection graphs and our goal is to
reconstruct $\Sigma(\mathcal{X})$ from $\Sigma(\mathcal{Y})$. To that end, let $(v_{1},e,v_{2})$ be a triple consisting of an open edge $e\in{E(\Sigma(\mathcal{Y}))}$ together with its two endpoints $v_{1}$ and $v_{2}$. In terms of schemes, these correspond to the points $\eta_{v_{i}}$ and $\eta_{e}$ in $\mathcal{Y}$.  The main procedures used in the algorithm are then as follows. %
\begin{center}
{\bf{[Main procedures]}}
\end{center}

\begin{itemize}
\item Calculate the $\eta_{v_{i}}$-adic power series expansions of the roots of $f$. %
\item Calculate the $\eta_{e}$-adic power series expansions of the $\eta_{v_{i}}$-adic coefficients to obtain the $\eta_{e}$-adic power series of the roots of $f$. 
\item Connect the $\eta_{e}$-adic power series for the two pairs $(v_{1},e)$ and $(e,v_{2})$.  
\end{itemize}

We calculate these expansions using various generalizations of the original Newton-Puiseux method. %
By repeating these procedures for all edges $e\in{E(\Sigma(\mathcal{Y}))}$, %
we are %
then able to reconstruct $\Sigma(\mathcal{X})$.

\subsection{An overview}

We now give a short overview of the contents of this paper. In Section \ref{ExpansionsExtensions}, 
we review power series expansions for regular local rings. In Sections 
\ref{FiniteNormalExtensions} to \ref{sec:ChainPrime}, we study various connections between roots of polynomials and extensions of prime ideals. More precisely, 
let $A$ be a Noetherian normal domain with fraction field $K(A)$ and let $K(A)\subset{K(B)}$ be a finite separable extension of fields given by an irreducible polynomial $f\in{K(A)[x]}$. We write 
$\alpha_{i}$ for the roots of $f$ in an algebraic closure of $K(A)$. Let $B$ be the normalization of $A$ in $K(B)$. We are interested in the possible extensions of a prime ideal $\mathfrak{m}_{A}$ in $A$ to prime ideals in $B$. 
In Section \ref{GaloisSection}, we find that these are classified by the orbits of the $\alpha_{i}$ under the action of the absolute decomposition group $D_{\mathfrak{m}_{A}}$. This is one of the main facts used in the algorithm. In studying these connections, we are also naturally led to higher-dimensional generalizations
of classical theorems from number theory such as the Kummer-Dedekind theorem and Dedekind's theorem on Galois groups, see Sections \ref{FiniteNormalExtensions} and \ref{GaloisSection}. %

In Section \ref{SectionMainTheorem}, we prove Theorems \ref{MainThm1} and \ref{NPTheorem}. We start by giving a description of the spectra of the complete local rings corresponding to ordinary double points and smooth points on a semistable model, see Section \ref{ModelsCurves}. In Section \ref{KummerExtensions} we study tame Kummer extensions of these points and prove Theorem \ref{NPTheorem}. We then use these results in Section \ref{ProofSemistable} to prove Theorem \ref{MainThm1}.

In Section \ref{Algorithm}, we give the algorithms that allow us to calculate the covering $\Sigma(\mathcal{X})\to\Sigma(\mathcal{Y})$ of dual intersection graphs associated to a covering of semistable models $\mathcal{X}\to\mathcal{Y}$. As mentioned earlier, these algorithms calculate $\eta_{v}$-adic and $\eta_{e}$-adic power series expansions of $f$. %
For the $\eta_{v}$, we use a generalization of the original discrete Newton-Puiseux method for tame polynomials over a discrete valuation ring.   %
This is given in Section \ref{NewtonDiscreteSection}. Here we need an algorithm that can factorize polynomials over finite extensions of the residue field of the discrete valuation ring. For us, the most complicated residue fields will be finite extensions of ${\mathbf{F}}_{p}(x)$ or ${\mathbf{Q}}(x)$. %
 In Section \ref{MixedNP1}, we give the mixed Newton-Puiseux algorithm. This algorithm can roughly be seen as two consecutive applications of the generalized discrete algorithm. %
 It starts by calculating the  %
$\eta_{v}$-adic power series of the roots of $f$ for a vertex $v$ that is adjacent to $e$. %
It then applies the generalized discrete algorithm to the $\eta_{v}$-adic coefficients to obtain the %
$\eta_{e}$-adic power series of the roots. In order for this to work, we need the lifts of the coefficients from the residue field of $\eta_{v}$ to split $\eta_{e}$-adically. %
In residue characteristic zero there is a canonical way to lift these $\eta_{v}$-adic  coefficients using a splitting of the reduction map $R\to{k}$. In residue characteristic $p>0$, we use a different approach that also gives the desired lifts. %
Using a comparison theorem for Galois actions (see Theorem \ref{Dedekind3}), we then show that we can combine these $\eta_{e}$-adic and $\eta_{v}$-adic power series to deduce the global structure of the covering $\Sigma(\mathcal{X})\to\Sigma(\mathcal{Y})$. 

\subsection{Organization of the paper}

This paper is subdivided into two main parts:
\begin{itemize}
\item A {\it{theoretic}} part on power series, the Kummer-Dedekind theorem, Dedekind's theorem on cycles in Galois groups, Kummer extensions and the tame simultaneous semistable reduction theorem. This theoretic part is subdivided into two sections, namely Sections \ref{Preliminaries} and \ref{SectionMainTheorem}. 
\item An {\it{algorithmic}} part on generalizations of the Newton-Puiseux algorithm for discrete valuation rings and ordinary double points. This part consists of only one section, namely Section \ref{Algorithm}. 
\end{itemize}

\begin{figure}[h]
\centering
\includegraphics[width=10cm]{{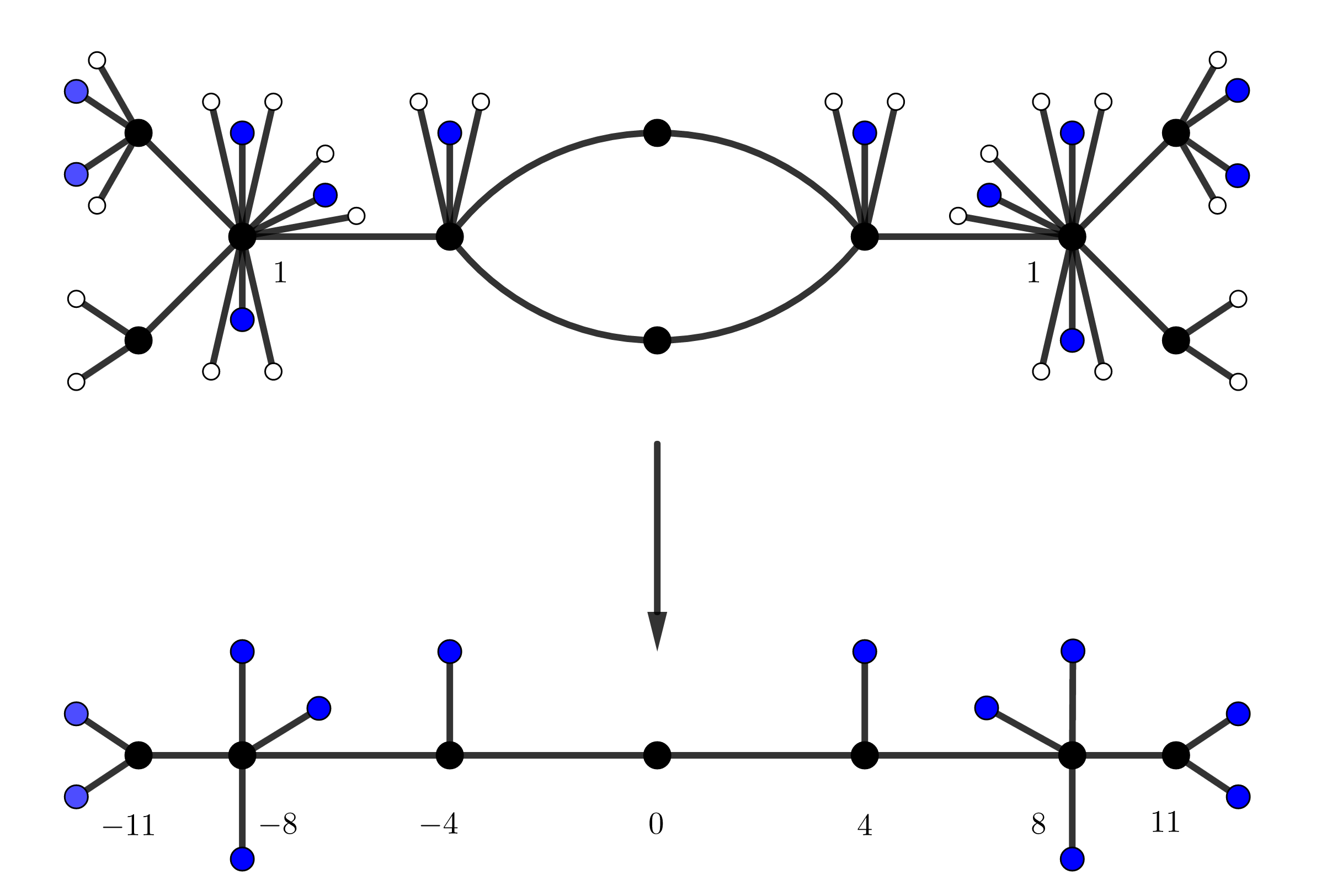}}
\caption{\label{Eindresultaat} {{The tropicalization of the $4:1$ covering in Example  %
\ref{MainExample}. The blue vertices in the top graph are the ramification points and the blue vertices in the bottom graph are the branch points. The $1$'s in the top graph represent the genera of the corresponding vertices. For the integers next to the vertices in the tree, see Example \ref{SeparatingMainExample}.}}} %
\end{figure}

For the algorithmic part, we will focus on the following example communicated to us by Prof. Bernd Sturmfels. This example illustrates many of the subtleties that go into the generalized Newton-Puiseux algorithms we present in Section \ref{Algorithm}, and as a bonus it also illustrates our generalized Dedekind theorem on Galois groups for finite extensions.   

\newpage
\begin{exa}\label{MainExample}{\bf{[Main Example]}}
Let $K=\mathbf{Q}((t))$ be the field of formal Laurent series over ${\mathbf{Q}}$ and consider the projective curve $X\subset{\mathbf{P}^{2}_{K}}$ defined by the homogeneous degree four polynomial
\begin{align*}
f &= t^{22}X^4+t^{14}X^3Y+t^{8}X^3Z+t^{10}X^{2}Y^{2}+t^{16}X^{2}YZ \\
   & +(1+t^6)X^{2}Z^{2}+t^{4}XY^{3}+(t^{12}-2)XY^{2}Z+t^{16}XYZ^{2} \\
    &+t^8XZ^3+(t^{2}+1)Y^4+t^{4}Y^3Z+t^{10}Y^2Z^2+t^{14}YZ^3+t^{22}Z^4.
\end{align*}
By the Jacobi criterion we find that $f$ defines a smooth plane quartic and thus $g(X)=3$. We note that this plane quartic specializes to a double conic, so we are in the situation described in \cite[Page 133]{HarrisMorrison}.

We are now interested in the %
dual intersection graph of a semistable model of $X$ together with the lengths of the edges and the genera of the vertices. %
The answer turns out to be as in Figure \ref{Eindresultaat}. %
To find this graph, we use the degree four map of smooth curves %
$\phi:X\rightarrow{\mathbf{P}^{1}_{K}}$ induced by the rational map %
\begin{equation}
[X:Y:Z]\mapsto{[X:Z]}.
\end{equation}
We start by calculating a separating tree for the branch locus $D$ of $\phi$. This determines a semistable model $\mathcal{Y}$ for $(\mathbf{P}^{1}_{K},D)$ and using Theorem \ref{MainThm1}, we obtain a lift of $\phi$ to a morphism of semistable models %
$\mathcal{X}\rightarrow{\mathcal{Y}}$. %
By applying the Newton-Puiseux algorithms in Section \ref{Algorithm}, we then obtain the intersection graph of $\mathcal{X}$.

\end{exa}

We also apply the algorithm to an example provided by Prof. Hannah Markwig, see \cite{Hahn2019} for more examples. In this example we find a curve of genus $3$ with a degree four covering $X\rightarrow{\mathbf{P}^{1}_{K}}$ such that the local covering data for $\Sigma(\mathcal{X})\to\Sigma(\mathcal{Y})$ does not determine the global graph-theoretical structure of $\Sigma(\mathcal{X})$.  %

\begin{exa}
\label{HannahExampleIntro}
Consider the plane quartic $X/\mathbf{Q}((t))$ given by the degree four homogeneous polynomial 
\begin{align*}
f:=&t^{24}X^4-X^2Y^2+t^8XY^3+t^{18}Y^4-2X^2YZ+t^3XY^2+t^{12}Y^3Z-(t^4-1)X^2Z^2+\\
&XYZ^2+t^8Y^2Z^2+t^6XZ^3+t^{11}YZ^3+t^{18}Z^3.
\end{align*}
As in Example \ref{MainExample}, this defines a smooth plane quartic, so its genus is three. We again consider the covering given by the rational map $[X:Y:Z]\mapsto[X:Z]$. Applying the techniques in this paper, we find that the dual intersection graph is as in Figure %
\ref{EndResultHannah2}. %
The details are in Example \ref{HannahExample1}. 
\begin{figure}[h]
\begin{center}
\includegraphics[height=6cm]{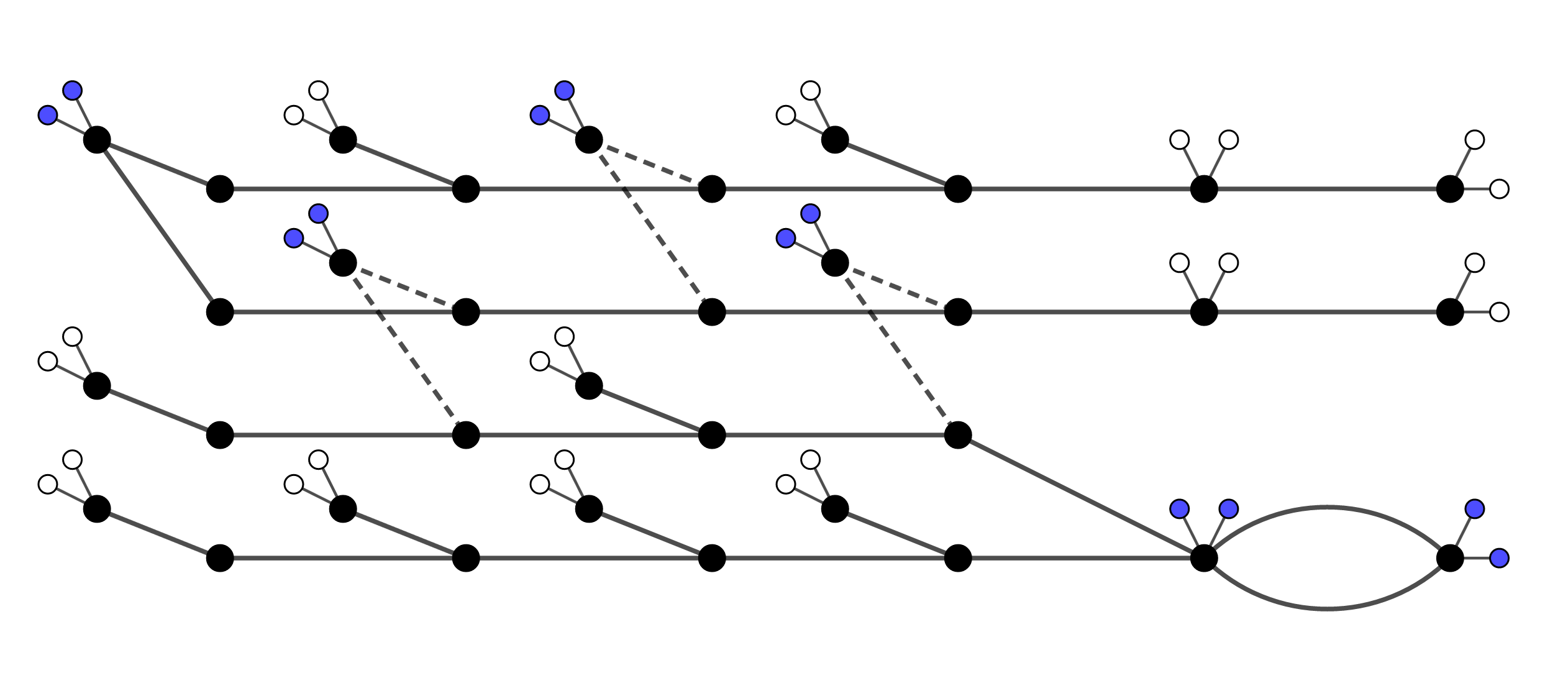}
\caption{\label{EndResultHannah2} The final graph for the curve in Example \ref{HannahExampleIntro}. This admits a degree four covering to a tree (see Figure %
\ref{SeparatingTreeHannah1}) and the blue vertices correspond to the ramification points for this covering. }%
\end{center}
\end{figure}

\end{exa}

\subsection{Connections to the existing literature}

We now point out connections between this paper and the existing literature. The part on Kummer extensions is a simplification of some of the results in \cite{SGA1} and \cite{Grothendieck1971} in terms of fields and normalizations. %
In applying this field-theoretic approach, we also find generalizations of results that are well-known in the one-dimensional case, but seemingly unknown in higher dimensions. For instance, we are able to generalize the {\it{Kummer-Dedekind}} theorem from algebraic number theory, which finds the factorization of a prime number in a number field for primes that do not divide the discriminant of the number field. Another well-known theorem we were able to generalize is Dedekind's theorem on deducing elements of Galois groups of monic irreducible $f\in\mathbf{Z}[x]$ using factorizations over $\mathbf{F}_{p}[x]$, again for primes $p$ that do not divide the discriminant. The original versions can be found in \cite[Theorem 3.1 and Corollary 8.11]{Ste3}, \cite[Chapter 1, Proposition 8.3]{neu} and \cite[Theorem 13.4.5]{coxgalois}. The generalized versions can be found in Theorems \ref{KummerDedekind}, \ref{Dedekind1} and \ref{Dedekind2a}.

Using the theorems on Kummer extensions, we then deduce our main theorem: the tame simultaneous semistable reduction theorem, see Theorem \ref{MainThm1}. Various (tame) simultaneous semistable reduction theorems can be found in the literature, but ours is most closely related to \cite[Theorem 2.3]{liu1} in terms of semistable models, and \cite[Theorem 3.1]{DecompFund1} in terms of Berkovich spaces. In \cite{Liu2006} and \cite{ABBR1} simultaneous semistable reduction theorems are proved that work without any tameness assumptions. This unfortunately doesn't give a direct way of determining semistable models using coverings, since the covering of semistable models $\mathcal{X}\rightarrow{\mathcal{Y}}$ for $X\rightarrow{Y}$ is induced from a semistable model for $X$. A great deal of progress has been made in the meantime on tropicalizations of wild coverings, %
see \cite{Cohen2016} and \cite{Brezner2019} for results in this direction.

The discrete Newton-Puiseux algorithm we give here is a direct generalization of the one given in \cite{Duval}. %
The factorization problems encountered there will also be of importance in this paper, as we will see when we try to determine the vertices of the skeleton of a curve, see Remark \ref{ExtensionsMinimalPolynomialsAlgorithm}. In \cite{BeringerJung2003} and \cite{MCDONALD1995} a multi-variate generalization of the Newton-Puiseux algorithm is given. This method takes a polynomial $f\in{\mathbf{C}[x_{1},...,x_{r}][y]}$ of degree $d$ and returns a set of $d$ power series over some generalized fractional power series ring $\mathbf{C}((C_{\mathbf{Q}}))$, where $C_{\mathbf{Q}}$ is a strongly convex rational polyhedral cone in $\mathbf{R}^{n}$, see \cite{MCDONALD1995} for the exact definition. It is at present not clear to the author if there are any direct connections between this multi-variate algorithm and the one in our paper, since we work with more general regular local rings which are not necessarily defined over a field. The formulation in terms of fields of Theorem \ref{NPTheorem} does however point at some interesting possible generalizations of these multi-variate algorithms. %

Lastly, we mention that there are many algorithms in the literature that calculate semistable models for curves in specific cases, for instance when the curve has a Galois morphism to the projective line or when the curve is embedded in a torus in a suitable way. 
We refer the interested reader to the following papers %
for examples of these algorithms: \cite{TDokchitser2018}, \cite{bouw_wewers_2017}, \cite{CuetoMarkwig2016}, \cite{tropabelian}, \cite{supertrop}, \cite{InvariantsSuper} and \cite{troptamethree}.

\begin{rem}
{\bf{[Algorithms in this paper]}}
Some of the steps in the algorithms we give in this paper contain some intentional ambiguities and omissions in order to convey the idea of the algorithm in a more conceptual way. For instance, in Algorithm \ref{DiscreteNPAlgorithm} we will say that the power series have to be updated with the data that was just calculated. To make this into a line of code, we would have to introduce markers and containers to make sure that the data is transported to the right location. We think that this would detract from the underlying idea, so we instead opt for this style, leaving some room for programmers to transform the ideas into code in their own way.
\end{rem}

\tableofcontents

\subsection{Preliminaries}\label{IntroPrelim}

Throughout this paper, $R$ will be a complete discrete valuation ring with field of fractions $K$, maximal ideal $\mathfrak{m}$, algebraically closed residue field $k$\footnote{This assumption is made for theoretical reasons. It has one nontrivial practical consequence which will be discussed in Remark \ref{RemarkGeometricallyIrreducible}.} and uniformizer $\pi$. The (normalized) valuation will be denoted by $v(\cdot{}):K^{*}\rightarrow{\mathbf{Z}}$. A point $z_{1}$ in a scheme $Z$ is said to be a generization of $z_{2}$ if $z_{2}\in\overline{\{z_{1}\}}$. We will write this as $z_{1}\subset{z_{2}}$. The point $z_{2}$ is said to be a specialization of $z_{1}$.   %
A curve $X/K$ is %
an algebraic variety over $K$ (see \cite[Chapter 3, Definition 3.47]{liu2}) whose irreducible components are all of dimension one. Unless mentioned otherwise, we will also assume that curves are proper, smooth and geometrically irreducible.  We then have the following notion of a model $\mathcal{X}/R$ for a curve $X/K$.

\begin{mydef}{\bf{[Models of curves]}}\label{Models2}
Let $X/K$ be a curve over a complete discretely valued field $K$ with valuation ring $R$. A model $\mathcal{X}$ for $X$ is an integral scheme $\mathcal{X}$ with a flat and proper morphism $\mathcal{X}\rightarrow{\mathrm{Spec}(R)}$ %
and an isomorphism $\mathcal{X}_{\eta}\rightarrow{X}$. Here $\mathcal{X}_{\eta}$ is the generic fiber of $\mathcal{X}$.   %
\end{mydef}

\begin{mydef}{\bf{[Smooth and ordinary double points]}}\label{PointsCompletions1}
Let $\mathcal{X}$ be a model for a curve $X$. A smooth closed point on the special fiber is a closed point $\eta\in\mathcal{X}$ %
such that 
\begin{equation}
\widehat{\mathcal{O}}_{\mathcal{X},\eta}\simeq{R[[u]]}.
\end{equation}

An ordinary double point is a closed point $\eta\in\mathcal{X}$ %
such that 
\begin{equation}
\widehat{\mathcal{O}}_{\mathcal{X},\eta}\simeq{R[[u,v]]/(uv-\pi^{n})}
\end{equation} 
for some integer $n$. The integer $n$ is called the {\it{thickness}} or {\it{length}} of the ordinary double point. It is independent of the isomorphism chosen by \cite[Chapter 10, Corollary 3.22(c)]{liu2}.   %
\end{mydef}

\begin{rem}
We note that the definition of a smooth point given here coincides with the usual one for points in the special fiber by  
\cite[Proposition 17.5.3]{EGA4}. %
\end{rem}

\begin{mydef}
\label{Semistable1}
{\bf{[Semistable and permanent models]}}
Let $\mathcal{X}\rightarrow{\mathrm{Spec}(R)}$ be a model for a proper curve $X/K$. We say that $\mathcal{X}$ is semistable if the special fiber $\mathcal{X}_{s}$ is reduced and every closed point $\eta\in{\mathcal{X}}$ is either a smooth point or an ordinary double point. It is strongly semistable if additionally every irreducible component of the special fiber is smooth over $k$. %
If the special fiber of a model $\mathcal{X}$ is reduced, then we say that $\mathcal{X}$ is a {\it{permanent}} model.
\end{mydef}

\begin{mydef}
{\bf{[Normalized base change]}} Let $\mathcal{X}\to\mathrm{Spec}(R)$ be a model for a curve $X/K$ and let $R\subset{R'}$ be a finite map of discrete valuation rings. The normalized base change of $\mathcal{X}$ is the normalization of $\mathcal{X}$ in $K'(X)$. Here $K'(X)$ is the function field of $X\otimes_{K}{K'}$. \end{mydef}

Let $\mathcal{X}/R$ be a model for a curve $X/K$. We have a natural reduction map 
\begin{equation}\label{ReductionMap}
\mathrm{red}_{\mathcal{X}}:X(K)\to{\mathcal{X}(k)}.
\end{equation}
That is, starting with a point $P_{K}\in{X(K)=\mathcal{X}(K)}$, we find using the valuative criterion of properness that there is a unique morphism ${P_{R}}:\mathrm{Spec}(R)\to\mathcal{X}$ such that the composition of $P_{R}$ with $\mathrm{Spec}(K)\to\mathrm{Spec}(R)$ gives $P_{K}:\mathrm{Spec}(K)\to\mathcal{X}$. By composing $P_{R}$ with the natural map $\mathrm{Spec}(k)\to\mathrm{Spec}(R)$, we then obtain the desired point $\mathrm{red}_{\mathcal{X}}(P_{K})$ in $\mathcal{X}(k)$.

\begin{mydef}\label{SemStaMarkedCurves}
{\bf{[Semistable models for marked curves]}} Let $D\subset{X(K)}$ be a finite set of closed points on a curve. We call the pair $(X,D)$ a marked curve. A semistable model for $(X,D)$ is a semistable model $\mathcal{X}$ for $X$ such that the restriction of the reduction map $\mathrm{red}_{\mathcal{X}}(\cdot{})$ to $D$ gives an injective map
\begin{equation}
\mathrm{red}_{\mathcal{X}|D}:D\to\mathcal{X}^{\mathrm{sm}}_{s}(k)
\end{equation} to the $k$-points of the smooth locus $\mathcal{X}^{\mathrm{sm}}_{s}\subset{\mathcal{X}_{s}}$ of the special fiber. We similarly define a strongly semistable model for a marked curve.   
\end{mydef}

\begin{mydef}
{\bf{[Dual intersection graphs and skeleta]}}
Let $\mathcal{X}/R$ be a strongly semistable model for a curve $X/K$. Consider the set $V(\mathcal{X})$ of irreducible components in the special fiber $\mathcal{X}_{s}$ and the set $E(\mathcal{X})$ of intersection points of these components. The dual intersection graph $\Sigma(\mathcal{X})$ of $\mathcal{X}$ is the undirected graph with vertex set $V(\mathcal{X})$ and edge set $E(\mathcal{X})$, where the endpoints of the edges are the vertices corresponding to the irreducible components the intersection point lies on. We will also call this the {\it{skeleton}} of $\mathcal{X}$. 

 For any vertex $v\in{V}(\mathcal{X})$ corresponding to a component $\Gamma_{v}$, we write $g(\Gamma_{v})$ for the genus of this curve $\Gamma$ over $k$. For any edge $e\in{E(\mathcal{X})}$ corresponding to an ordinary double point $x_{e}\in\mathcal{X}$, we write $\delta_{e}$ for the thickness of $x_{e}$. We now turn the dual intersection graph into a weighted metric graph $(\Sigma(\mathcal{X}),\ell(\cdot{}),w(\cdot{}))$ by defining the functions
\begin{align*}
\ell:E(\mathcal{X})&\rightarrow{\mathbf{Z}_{\geq{0}}},\\
e&\mapsto{\delta_{e}},\\
w:V(\mathcal{X})&\rightarrow{\mathbf{Z}_{\geq{0}}},\\
v&\mapsto{g(\Gamma_{v})}.
\end{align*}

For a marked curve $(X,D)$ and a strongly semistable model $\mathcal{X}$ for $(X,D)$, we add extra leaves at the vertices that the points in $D$ reduce to. 
\end{mydef}

\begin{mydef}\label{TropicalizationMorphism}
{\bf{[Tropicalization of a finite morphism]}}
Let $\mathcal{X}$ and $\mathcal{Y}$ be strongly semistable models for curves $X$ and $Y$ respectively and let $\phi:X\to{Y}$ be a finite morphism. Suppose that there is a finite $R$-morphism $\phi_{R}:\mathcal{X}\to\mathcal{Y}$ lifting $\phi$ such that %
every non-smooth point in $\mathcal{X}_{s}$ is sent to a non-smooth point of $\mathcal{Y}_{s}$ by $\phi_{R}$. We then have an induced morphism of intersection graphs %
\begin{equation}
\Sigma(\mathcal{X})\to\Sigma(\mathcal{Y}),
\end{equation}
which we call the {\it{tropicalization}} or {\it{skeleton}} of $\phi_{R}$.
\end{mydef}
\begin{rem}
The condition on the non-smooth points ensures that we have maps $E(\Sigma(\mathcal{X}))\to{E(\Sigma(\mathcal{Y}))}$ next to the canonical maps $V(\Sigma(\mathcal{X}))\to{V(\Sigma(\mathcal{Y}))}$. We will see in the proof of Theorem \ref{MainThm1} %
that the morphisms in this paper all satisfy this extra condition. For an example of a finite morphism of semistable models that does not satisfy this condition, consider the %
semistable model $\mathcal{X}$ given locally by 
\begin{equation}
y^2=x(x-\pi)(x+1)(x+1-\pi).
\end{equation}
The morphism $(x,y)\mapsto{x}$ then induces a finite morphism of semistable models $\mathcal{X}\to\mathbf{P}^{1}_{R}$ that sends the non-smooth points $(x,y,\pi)$ and $(x+1,y,\pi)$ (written as prime ideals) to smooth points. The reason for this is that %
there are branch points that reduce to the same point on the special fiber, for instance $(x)$ and $(x-\pi)$.  
\end{rem}

\subsubsection{Models of $\mathbf{P}^{1}_{K}$}\label{P1Models}

We now introduce the models for $\mathbf{P}^{1}_{K}$ that we will be using throughout this paper. %
We also introduce notation for various affine subsets in these models. This notation resembles the Berkovich point of view of giving (formal) semistable models through semistable vertex sets, see \cite[Theorem 5.8]{ABBR1}. %
We also invite the reader to review the picture of an infinite $\mathbf{R}$-tree associated to the Berkovich analytification $\mathbf{P}^{1,\mathrm{an}}_{K}$, see \cite[Section 2.3]{baker2010potential}. %

We start by introducing notation for certain discrete valuations of the function field $K(\mathbf{P}^{1}_{K})=K(x)$. Let $x_{0}\in{K}$ and $k\in\mathbf{Z}_{\geq{0}}$. Consider the sets 
\begin{eqnarray*}
\mathbf{B}^{+}_{k}(x_{0})=\{x\in{K}:v(x-x_{0})\geqslant{}k\},\\
\mathbf{B}^{-}_{k}(x_{0})=\{x\in{K}:v(x-x_{0})\leqslant{}k\}.
\end{eqnarray*}
We refer to these as the positively oriented and negatively oriented closed disks of radius $k$ centered around $x_{0}$. %
If the orientation is clear, we also just call them closed disks. For these disks, we define injective ring homomorphisms %
$\phi_{\pm}:R[u_{1}]\to{K(x)}$ by %
\begin{align*}
\phi_{+}&:u_{1}\mapsto{\dfrac{x-x_{0}}{\pi^{k}}},\\
\phi_{-}&:u_{1}\mapsto{\dfrac{\pi^{k}}{x-x_{0}}}.
\end{align*}
We denote the image of such an embedding in $K(x)$ by $\mathbf{A}_{\mathbf{B}^{\pm}_{k}(x_{0})}$. %
We will sometimes identify $R[u_{1}]$ with its image $\mathbf{A}_{\mathbf{B}^{\pm}_{k}(x_{0})}\subset{K(x)}$. Since the ideal $(\pi)$ is prime in $\mathbf{A}_{\mathbf{B}^{\pm}_{k}(u_{0})}$ and regular of codimension one, we find that it defines a discrete valuation of $K(x)$.

\begin{mydef}
{\bf{[Discrete valuation assigned to a closed disk]}}\label{DefinitionValuationClosedBall}
Consider the subring $\mathbf{A}_{\mathbf{B}_{k}(x_{0})}$ of $K(x)$. %
We define the valuation of $K(x)$ associated to $\mathbf{B}^{\pm}_{k}(x_{0})$ to be the discrete valuation corresponding to $(\pi)$ in the algebra $\mathbf{A}_{\mathbf{B}^{\pm}_{k}(x_{0})}$. 
\end{mydef}

\begin{rem}
In terms of Berkovich spaces, the discrete valuation we introduced here corresponds to the type-$2$ point associated to the closed disk $\mathbf{B}^{\pm}_{k}(x_{0})$. This is independent of the orientation chosen.  
\end{rem}

\begin{exa}
Even though $\pi$ is a uniformizer for all of these discrete valuations of $K(x)$, they are different in general. As an example, consider the valuations $v_{\mathfrak{p}_{1}}$ and $v_{\mathfrak{p}_{2}}$ corresponding to $\mathbf{B}^{+}_{0}(0)$ and $\mathbf{B}^{+}_{1}(0)$. We have $v_{\mathfrak{p}_{1}}(x)=0$, but $v_{\mathfrak{p}_{2}}(x)=1$, since $x=\pi\cdot{}\dfrac{x}{\pi}$. 
\end{exa}

We now introduce notation for closed and open annuli. Consider the set 
\begin{equation}
\mathbf{S}_{a,b}(x_{0})=\{x\in{K}:a\leqslant{}v(x-x_{0})\leqslant{}b\},
\end{equation}
where $a$ and $b$ are any two integers with $a<b$ and $x_{0}\in{K}$. We refer to this as the closed annulus of inner radius $a$ and outer radius $b$ around $x_{0}$. Similarly, we have the open annulus of inner radius $a$ and outer radius $b$
\begin{equation}
\mathbf{S}^{+}_{a,b}(x_{0})=\{x\in{K}:a<v(x-x_{0})<b\}.
\end{equation} 

We now attach an algebra to these two sets. Define $c:=b-a$. Consider the algebra 
$R[u_{1},v_{1}]/(u_{1}v_{1}-\pi^{c})$.
We embed this ring into $K(x)$ by sending
\begin{align*}
u_{1}&\mapsto{\dfrac{x-x_{0}}{\pi^{a}}},\\
v_{1}&\mapsto{\dfrac{\pi^{b}}{x-x_{0}}}.
\end{align*}
We write $\mathbf{A}_{\mathbf{S}_{a,b}(x_{0})}$ for the image of $R[u_{1},v_{1}]/(u_{1}v_{1}-\pi^{c})$ under this embedding. As before, we view $R[u_{1},v_{1}]/(u_{1}v_{1}-\pi^{c})$ as a subring of $K(x)$. 
\begin{mydef}
{\bf{[Maximal ideal corresponding to an open annulus]}}\label{MaximalIdealAnnulus}
Let $\mathbf{A}_{\mathbf{S}_{a,b}(x_{0})}$ be the algebra associated to the annuli $\mathbf{S}_{a,b}(x_{0})$ and $\mathbf{S}^{+}_{a,b}(x_{0})$. Let $\mathfrak{m}=(u_{1},v_{1},\pi)$. We call $\mathfrak{m}$ the maximal ideal associated to the open annulus $\mathbf{S}^{+}_{a,b}(x_{0})$. 
\end{mydef} 
\begin{rem}
The algebra $\mathbf{A}_{\mathbf{S}_{a,b}(x_{0})}$ has two codimension one primes in the special fiber: $\mathfrak{p}_{1}=(u_{1},\pi)$ and $\mathfrak{p}_{2}=(v_{1},\pi)$. Their associated valuations are the same as the ones induced by $\mathbf{B}_{b}(x_{0})$ and $\mathbf{B}_{a}(x_{0})$ respectively (note the change here), see Definition \ref{DefinitionValuationClosedBall}. We will also write these as the valuations corresponding to the sets
\begin{align}\label{Endpoint1}
\partial_{1}\mathbf{S}_{a,b}(x_{0})&=\{x\in{K}:v(x-x_{0})=b\},\\
\partial_{2}\mathbf{S}_{a,b}(x_{0})&=\{x\in{K}:v(x-x_{0})=a\}.\label{Endpoint2}
\end{align} 

We refer to these sets as the endpoints of $\mathbf{S}_{a,b}(x_{0})$. 

\end{rem}

\begin{mydef}
\label{SeparatingModel}{\bf{[Separating models and trees]}}
Let $\mathcal{Y}$ be a strongly semistable model for $\mathbf{P}^{1}_{K}$ and let $D=\{P_{1},...,P_{r}\}$ be a finite set of $K$-rational points of $\mathbf{P}^{1}_{K}$. We say that $\mathcal{Y}$ is separating for $D$ if $\mathcal{Y}$ is a strongly semistable model for the marked curve $(\mathbf{P}^{1}_{K},S)$ in the sense of Definition \ref{SemStaMarkedCurves}. %
If $\mathcal{Y}$ is separating for $D$, then we call the dual intersection graph $\Sigma(\mathcal{Y})$ of $\mathcal{Y}$ a {\it{separating tree}} for $D$. We highlight the reductions of the points $P_{i}\in{D}$ by adding leaves to $T$. 
\end{mydef} 

These separating models are easily created using the Berkovich point of view. %
We will use the language of semistable vertex sets, see %
\cite{ABBR1}. Let $D\subset{\mathbf{P}^{1}_{K}(K)}$ be a finite set of $K$-rational points and assume that the Euler characteristic $\chi(\mathbf{P}^{1}_{K},D)=2-2g-|D|=2-|D|$ is strictly negative. There is then a canonical minimal semistable vertex set $V$ of $\mathbf{P}^{1}_{K}$. This gives a finite set of closed annuli and disks inside $K$ by the following procedure. We first assume that $|V|>1$. For an open annulus in the decomposition afforded by the definition of a semistable vertex set, we take the natural closed annulus associated to it. If $v\in{V}$ is a {\it{leaf vertex}}, then there is exactly one open annulus represented by $a<v(x-x_{0})<b$ adjacent to $v$. Here we assume without loss of generality that $v$ corresponds to $a$. We define the closed disk associated to $v$ to be $\mathbf{B}^{-}_{a}(x_{0})$. The strongly semistable model $\mathcal{Y}=\mathcal{Y}_{V}$ associated to $V$ by \cite[Theorem 5.8]{ABBR1} is then obtained by gluing these closed disks and annuli along the obvious transition maps. If $|V|=1$, then $v$ corresponds to $\mathbf{B}^{\pm}_{k}(x_{0})$ for some $k$ and $x_{0}$. We then just glue the two rings $\mathbf{A}_{\mathbf{B}^{\pm}_{k}(x_{0})}$ along the obvious transition map. For $k=0$ and $x_{0}=0$ this construction gives $\mathbf{P}^{1}_{R}$.

\begin{exa}\label{SeparatingMainExample}
{\bf{[Main Example]}} 
Consider the curve $X$ in Example \ref{MainExample}. The covering of smooth curves $\psi:X\rightarrow{\mathbf{P}^{1}_{K}}$ induced by the rational map $[X:Y:Z]\mapsto{[X:Z]}$ has a finite {{branch locus}} which we denote by $D$. We determine a separating model $\mathcal{Y}$ for $D$. To do this, we first calculate the discriminant of the covering. Over the affine chart $Z\neq{0}$ with $x=X/Z$, this gives a polynomial $\Delta(x)$ of degree $12$. %
A Newton polygon calculation shows that the roots $\alpha_{i}$ of the discriminant $\Delta$ are distributed as in the following table. 

\begin{center}
 \begin{tabular}{|c | c|}
 \hline
$\#\{\alpha_{i}\}$ & $v({\alpha_{i}})$ \\
 \hline
$2$ & $11$ \\
 $3$ & $8$ \\
 $1$ & $4$\\
 $1$ & $-4$\\
 $3$ & $-8$\\
 $2$ & $-11$\\
 \hline
 \end{tabular}
 \end{center}

 Further calculations then show that the $t$-adic coefficients of these roots only coincide up to these orders. For instance, write $\alpha_{1}$ and $\alpha_{2}$ for the roots with $v(\alpha_{i})=11$. We then calculate $1/t^{54}\cdot{}\Delta(t^{11}x)$, which reduces to the separable polynomial $-27x^2-4$ modulo $(t)$. In other words, the power series expansions of these roots (see Section \ref{ExpansionsExtensions}) start with
 \begin{align*}
 \alpha_{1}=\delta_{1}t^{11}+\mathcal{O}(t^{12}),\\
 \alpha_{2}=\delta_{2}t^{11}+\mathcal{O}(t^{12}),
 \end{align*}
 where the $\delta_{i}$ are the two roots of $-27x^2-4$ in $\overline{\mathbf{Q}}$. Doing similar calculations for the other roots then yields the separating tree as in Figure \ref{SeparatingTree}. This tree yields six closed annuli and their algebras from left to right are %
\begin{equation}
\mathbf{A}_{\mathbf{S}_{-11,-8}(0)},\,\,\mathbf{A}_{\mathbf{S}_{-8,-4}(0)},\,\,\mathbf{A}_{\mathbf{S}_{-4,0}(0)},\,\,\mathbf{A}_{\mathbf{S}_{0,4}(0)},\,\,\mathbf{A}_{\mathbf{S}_{4,8}(0)},\,\,\mathbf{A}_{\mathbf{S}_{8,11}(0)}.
\end{equation}
\begin{figure}[h!]
\centering
\includegraphics[height=3cm, width=10cm]{{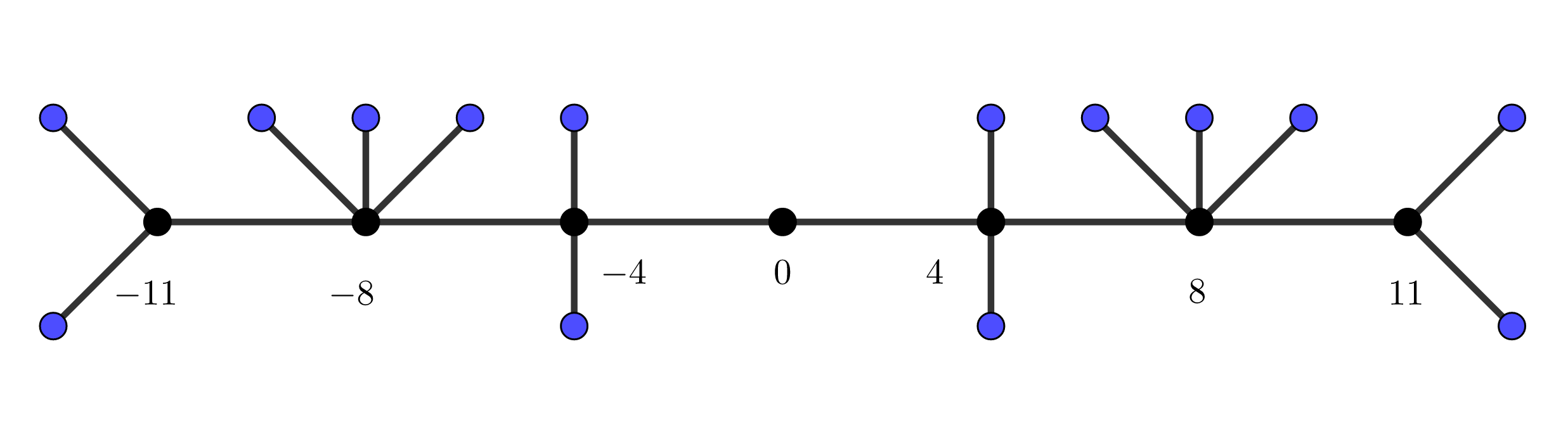}}
\caption{\label{SeparatingTree} {\small{The separating tree for the covering in Example \ref{MainExample}. The integers under the vertices correspond to the closed disks %
$\mathbf{B}_{k}(0)$, as in Definition \ref{DefinitionValuationClosedBall}.    }}} 
\end{figure}

In Section \ref{SectionMainTheorem}, we show that the normalization $\mathcal{X}$ of the model $\mathcal{Y}$ in $K(X)$ is again a strongly semistable model after a finite extension of $K=\overline{\mathbf{Q}}((t))$. This then induces a morphism of intersection graphs $\Sigma(\mathcal{X})\rightarrow{\Sigma(\mathcal{Y})}$ and we show in Section \ref{Algorithm} how we can calculate $\Sigma(\mathcal{X})$ using compatible power series expansions. %

\end{exa}

\section{Finite extensions and power series}
\label{Preliminaries}

In this section we discuss several generalities concerning finite extensions and power series expansions.  %
We start with the uniqueness of %
adic expansions with respect to a monomial basis in a regular local ring. %
This point of view %
forms a natural context for the Newton-Puiseux algorithms in Sections \ref{NewtonDiscreteSection} and \ref{MixedNP1}. %
 After this, we consider finite normal extensions of normal domains and their Galois counterparts. We prove an analogue of \cite[Theorem 3.8]{Ste2}, which says that in a finite extension of number fields one can find the primes lying above a prime $\mathfrak{p}$ by factoring a certain polynomial over the completion or the Henselization at $\mathfrak{p}$. This in turn is equivalent to finding the orbits of the roots under the absolute decomposition group, a fact that lies at the basis of the algorithms we present here. %
 We then generalize a classical theorem by Dedekind on %
cycles in Galois groups, see \cite[Theorem 13.4.5]{coxgalois} and \cite[Corollary 8.11]{Ste3} for instance. Stated somewhat imprecisely, the generalized theorem tells us that we can find shadows of Galois groups by doing computations over the completion or the Henselization of a local ring. %

\subsection{Expansions for regular local rings}\label{ExpansionsExtensions}

Let $(A,\mathfrak{m})$ be a regular local Noetherian ring with residue field $k:=A/\mathfrak{m}$. %
Suppose that we have a set of elements $x_{1},...,x_{r}\in\mathfrak{m}$ whose images in the $k$-vector space %
$\mathfrak{m}/\mathfrak{m}^{2}$ are a basis of that vector space (we will also say that the $x_{i}$ themselves form a basis of this vector space). Here $r=\mathrm{dim}(A)$ by regularity. This basis can then be used to give an isomorphism between the graded ring $\mathrm{gr}_{\mathfrak{m}}(A)$ and the polynomial ring over $k$ in $r$ variables, see \cite[\href{https://stacks.math.columbia.edu/tag/00NO}{Lemma 00NO}]{stacks-project} and \cite[Theorem 13.4]{Kemper2011}. We will interpret this result using the concept of monomials. %
\begin{mydef}\label{MonomialElements}{\bf{[Monomials]}}
Let $(A,\mathfrak{m})$ be a regular local Noetherian ring and let $\{x_{1},...,x_{r}\}$ be a set of generators for $\mathfrak{m}$. %
An element $m\in{A}$ of the form
\begin{equation}
m=x_{1}^{i_{1}}\cdots{x_{r}^{i_{r}}}
\end{equation}
for $i_{1},...,i_{r}\in\mathbf{N}$ 
is a {\it{monomial}} in the $x_{i}$ of total degree $\mathrm{deg}(m):=i_{1}+...+i_{r}$. We will also think of these monomials as being represented by vectors in $\mathbf{N}^{r}$. That is, every monomial $x_{1}^{i_{1}}\cdots{x_{r}^{i_{r}}}$ corresponds to the vector $(i_{1},...,i_{r})$. %
\end{mydef}
The set of all monomials in $A$ is denoted by $M$. For every $j\in\mathbf{N}$, we can then consider the set of monomials of total degree $j$: $M_{j}=\{m\in{M}:\mathrm{deg}(m)=j\}$. These monomials are elements of the ideal $\mathfrak{m}^{j}$. In fact, they form a basis: %
\begin{lemma}\label{BasisVectorSpaces}
Consider the set $M_{j}$ of all monomials in the $x_{i}$ of total degree $j$ and let $\overline{M}_{j}$ be the image of this set in the $k$-vector space $\mathfrak{m}^{j}/\mathfrak{m}^{j+1}$. Then $\overline{M}_{j}$ is a basis for this vector space.  %
\end{lemma}
\begin{proof}
This follows directly from %
\cite[Theorem 13.4]{Kemper2011}. %
\end{proof}%

For the upcoming proposition, let us recall that a set of representatives $S$ for the residue field $k:=A/\mathfrak{m}$ is a set of elements in $A$ such that the restriction of the quotient map $A\rightarrow{A/\mathfrak{m}}$ to $S$ yields a bijection.  %
In other words, for every element of the residue field, we have chosen exactly one representative in $A$. We will assume that $0\in{S}$, so that it represents $\overline{0}$. Recall furthermore that for any regular local Noetherian ring $(A,\mathfrak{m})$, we have the natural $\mathfrak{m}$-adic topology at our disposal, which means that we can write down $\mathfrak{m}$-adic power series as follows. Fix a set of generating monomials $\{x_{1},...,x_{r}\}$ for $A$. %
We write %
\begin{equation}\label{RegularPowerSeriesExample}
z=\sum_{\mathbf{i}\in\mathbf{N}^{r}}c_{\mathbf{i}}\cdot{}x_{1}^{i_{1}}\cdots{x_{r}^{i_{r}}}
\end{equation}
for $z\in{A}$, $\mathbf{i}=(i_{1},...,i_{r})\in\mathbf{N}^{r}$ and $c_{\mathbf{i}}\in{S}$ if the sequence %
$(z_{n})$ defined by taking all terms of total degree less than $n$ converges to $z$. For every $z\in{A}$, we can find a unique power expansion of this form by the following proposition. %

\begin{pro}\label{RegularExpansions}
Let $(A,\mathfrak{m})$ be a regular Noetherian local ring. %
 Let $x_{1},...,x_{r}\in{\mathfrak{m}}$ be elements that map to a basis of the $k$-vector space $\mathfrak{m}/\mathfrak{m}^{2}$ and let $S\subset{A}$ be a set of representatives of the residue field $A/\mathfrak{m}$. Then every element $z\in{A}$ can be written in a {\it{unique}} way as
\begin{equation}
z=\sum_{\mathbf{i}\in\mathbf{N}^{r}}c_{\mathbf{i}}\cdot{}x_{1}^{i_{1}}\cdots{x_{r}^{i_{r}}},
\end{equation}
where $\mathbf{i}=(i_{1},...,i_{r})$ and $c_{\mathbf{i}}\in{S}$. %
\end{pro}
\begin{proof}
The proof follows from Lemma \ref{BasisVectorSpaces} by set-theoretically splitting the exact sequences 
\begin{equation}
0\rightarrow{\mathfrak{m}^{i+1}}\to\mathfrak{m}^{i}\to\mathfrak{m}^{i}/\mathfrak{m}^{i+1}\rightarrow{0}.
\end{equation}
 Here we set $\mathfrak{m}^{0}=A$. We leave the details to the reader. 

\end{proof}

\begin{rem}
Note that this proposition does not require $A$ to be complete. %
\end{rem}

\begin{rem}\label{LocalRingsPowerSeries}
Proposition \ref{RegularExpansions} can be generalized to local Noetherian rings $(A,\mathfrak{m})$ as follows. As before, we fix a set of representatives $S$. Every $k$-vector space %
$\mathfrak{m}^{k}/\mathfrak{m}^{k+1}$ is then finite-dimensional and we can find a basis for each of these vector spaces. We then use lifts of these generators to again obtain %
unique power expansions. %
For regular local Noetherian rings, the generators of the maximal ideal automatically give a set of generators of $\mathfrak{m}^{k}/\mathfrak{m}^{k+1}$ by Lemma \ref{BasisVectorSpaces}, so the theorem becomes more convenient in this case. %
\end{rem}

\begin{cor}\label{CorollaryRegularRingODP}
Let $A:=R[[u,v]]/(uv-\pi)$, where $R$ is a complete discrete valuation ring and $S\subset{R}$ is a set of representatives of the residue field $k$ of $R$. Then every element $z\in{A}$ can be written uniquely as 
\begin{equation}
z=\sum_{(i,j)\in\mathbf{N}^2}c_{i,j}u^{i}v^{j}, 
\end{equation}
where $c_{i,j}\in{S}$. %
\end{cor}
\begin{proof}
The maximal ideal of $A$ is given by $(u,v,\pi)=(u,v)$. The rest now follows from %
Proposition \ref{RegularExpansions}. %

\end{proof}

By Proposition \ref{RegularExpansions}, we see that after choosing a basis of $\mathfrak{m}/\mathfrak{m}^{2}$ we have a natural injection of the regular local ring $A$ into the set $\mathcal{S}_{A}$ of functions
\begin{equation}
\mathbf{N}^{r}\rightarrow{S},
\end{equation}
where $S$ is a set of representatives of the residue field and $r=\mathrm{dim}(A)$. This map is not necessarily surjective, but  we can consider the {\it{completion}} $\hat{A}$ of $A$ with respect to its $\mathfrak{m}$-adic topology. This is again a regular local ring of the same dimension and for this ring (and any basis of the cotangent space $\mathfrak{m}/\mathfrak{m}^{2}$ of $\hat{A}$) the map $\hat{A}\rightarrow{\mathcal{S}_{\hat{A}}}$ is automatically bijective. %

 \begin{mydef}{\bf{[Separating ideals]}}\label{SeparatingMonomialsRegular}
Let $T$ be a finite set of pairwise distinct elements in $A$.   %
 An ideal $I$ is said to be {\it{separating}} for $T$ if the restriction of the quotient map $A\to{A/I}$ to $T$ is injective. %
 \end{mydef}

 \begin{rem}{\bf{[$\mathcal{O}$-notation]}}
 The notation
 $$x=y+\mathcal{O}(I)$$
will be used to denote $x-y\in{I}$. %
 \end{rem}

\subsection{Finite extensions and the Kummer-Dedekind theorem}\label{FiniteNormalExtensions}

In this section we study extensions of prime ideals in normal local Noetherian domains in terms of completions. We prove a generalization of the Kummer-Dedekind theorem, which says that we can find extensions of prime ideals by factoring polynomials over the completion. This fact lies at the basis of our algorithms for calculating skeleta, as it implies that we can work with normalizations using power series, see Sections \ref{GaloisSection} and \ref{HenselizationsSection}.

\begin{pro}\label{CompletionDecomposition}
Let $(A,\mathfrak{m})$ be a Noetherian local ring and let $B$ be a finite $A$-algebra. Then every maximal ideal of $B$ lies over $\mathfrak{m}$ and there are only finitely many of these. Denoting them by $\mathfrak{m}_{i}$, we have that
\begin{equation}
   B\otimes_{A}{\hat{A}}\simeq{\prod{(B_{\mathfrak{m}_{i}})^{\hat{}}}},
\end{equation}
where every completion of a local ring is with respect to the natural maximal ideal of that local ring. %
\end{pro}
\begin{proof}
See \cite[\href{https://stacks.math.columbia.edu/tag/07N9}{Lemma 07N9}]{stacks-project}.

\end{proof}

\begin{lemma}\label{LemmaDimension}
Let $(A,\mathfrak{m})$ be a local Noetherian domain with field of fractions $K$ and let $L$ be a finite extension of $K$ of degree $n$. Let $B$ be a finite $A$-subalgebra of $L$ such that the quotient field of $B$ is equal to $L$. %
Then the base change $K(\hat{A})\otimes_{A}{B}$ has dimension $n$ as a $K(\hat{A})$-vector space. %
\end{lemma}
\begin{proof}

Since the fraction field of $B$ is $L$, we can find a set of $n$ elements $x_{i}$ in $B$ that generate $B\otimes_{A}{K}=L$. We write $M$ for the $A$-module generated by the $x_{i}$. We then have an exact sequence %
\begin{equation}
 0\rightarrow{}T\rightarrow{}A^{n}\rightarrow{B}\rightarrow{B/M}\rightarrow{0}
\end{equation}
of $A$-modules. Here $T=0$ since the $x_{i}$ form a linearly independent set over $K(A)$. %
Furthermore, the module $N:=B/M$ %
is torsion in the sense that %
$N\otimes_{A}{K(A)}=0$. %
Since $N$ is finitely generated, we can find an $s\in{A\backslash\{0\}}$ such that $sN=0$. %

The completion $\hat{A}$ is (faithfully) flat over $A$. Using \cite[Corollary 3.14, page 21]{liu2}, we thus obtain an exact sequence %
\begin{equation}
     0\rightarrow{}\hat{A}^{n}\rightarrow{\hat{B}}\rightarrow{\hat{N}}\rightarrow{0}%
\end{equation}
of $\hat{A}$-modules. %
Since localizations are flat, we also obtain an exact sequence
\begin{equation}
     0\rightarrow{}K(\hat{A})^{n}\rightarrow{B\otimes_{A}{K(\hat{A})}}\rightarrow{{N}\otimes_{A}{K(\hat{A})}}\rightarrow{0}
\end{equation}
of $K(\hat{A})$-vector spaces\footnote{Throughout the arguments given here, we freely use the equalities $\hat{M}\otimes_{\hat{A}}K(\hat{A})=(M\otimes_{A}\hat{A})\otimes_{\hat{A}}K(\hat{A})=M\otimes_{A}K(\hat{A})$, see \cite[Chapter 1, Corollary 1.11 and Corollary 3.14]{liu2}.}. We claim that the vector space on the right is $0$. In fact, the $s$ we found earlier maps to a nonzero element in $\hat{A}$ by Krull's theorem, and it annihilates $\hat{N}=N\otimes_{A}{\hat{A}}$, %
so indeed ${N}\otimes_{A}{K(\hat{A})}=0$. %
This shows that $B\otimes_{A}{K(\hat{A})}$ is isomorphic to $K(\hat{A})^{n}$, as desired. 
\end{proof}

\begin{rem}\label{AssumptionsGalois}
For the remainder of Sections \ref{FiniteNormalExtensions} and \ref{GaloisSection}, we will assume that the completion of any local Noetherian domain is a domain. If $A$ is normal and {\it{excellent}} (see \cite[Section 8.2.3]{liu2}), then by \cite[Chapter 8, Proposition 2.41]{liu2} the completion $\hat{A}$ is again normal and thus a domain. In this paper all coordinate rings $A$ are excellent, since they are of finite type over a complete discrete valuation ring $R$. We thus see that this assumption holds for the rings under consideration here. 
In Section \ref{HenselizationsSection} we will point out how to remove this assumption by working with the %
 Henselization of $A$ instead of its completion.  

\end{rem}

\begin{theorem}\label{PropIrrDivisors}\label{KummerDedekind}
{\bf{[Kummer-Dedekind]}}
Let $A$ be a Noetherian local normal domain with fraction field $K$ and consider a finite separable extension $L$ of $K$ with normalization $B$. We write $\mathfrak{m}_{i}$ for the maximal ideals in $B$ lying over $\mathfrak{m}$ and %
$B_{i}$ for the completions of $B$ at the $\mathfrak{m}_{i}$. %
Let $\alpha\in{B}$ be a generator of $L/K$ and let $f$ be its minimal polynomial over $A$. Write %
$f=\prod{f_{i}}$ for its factorization into monic irreducibles in $K(\hat{A})[x]$. There is then a bijection
\begin{equation}
   \{\mathfrak{m}_{i}\supset{\mathfrak{m}}\}\rightarrow{\{f_{i}\}}
\end{equation}
between the set of extensions of $\mathfrak{m}$ to $B$ and the set of irreducible factors of $f$ inside $K(\hat{A})[x]$. This bijection is determined by the $K(\hat{A})$-isomorphism %
\begin{equation}
    K(B_{i})\simeq{}K(\hat{A})[x]/(f_{i}). 
\end{equation}

\end{theorem}
\begin{proof}
Consider the subring $C:=A[x]/(f)\subset{B}$. Since $\alpha$ generates the extension $L/K$, we have that $L=C\otimes_{A}{K(A)}=B\otimes_{A}{K(A)}$. Let us write $\hat{B}$ and $\hat{C}$ for the completions of $B$ and $C$ with respect to $\mathfrak{m}$. If we tensor the injection $C\to{B}$ with the flat $A$-module $\hat{A}$, then we obtain an injection %
\begin{equation}
  \hat{C}=C\otimes_{A}{\hat{A}}\rightarrow{\hat{B}=B\otimes_{A}{\hat{A}}}.
\end{equation}
Tensoring with $K(\hat{A})$ over $A$ (and using again that any localization is flat), we then obtain the injection $K(\hat{A})\otimes_{A}{\hat{C}}\rightarrow{K(\hat{A})\otimes_{A}{\hat{B}}}$.  %
By Lemma \ref{LemmaDimension}, the dimensions of these $K(\hat{A})$-modules are equal, so we must have  $K(\hat{A})[x]/(f)=K(\hat{A})\otimes_{A}{B}$. We can give more explicit descriptions of these rings using Proposition \ref{CompletionDecomposition} and the Chinese Remainder Theorem. The Chinese Remainder Theorem gives us
\begin{equation}
    K(\hat{A})[x]/(f)\simeq{}\prod{K(\hat{A})[x]/(f_{i}^{r_{i}})},
\end{equation}
where the $f_{i}$ are the monic irreducible factors of $f$. On the other hand, by Proposition \ref{CompletionDecomposition} we know that $\hat{B}\simeq{\prod{B_{i}}}$, so that
\begin{equation}
    \hat{B}\otimes_{\hat{A}}{K(\hat{A})}\simeq{\prod{(K(\hat{A})\otimes_{\hat{A}}{B_{i}})}}.
\end{equation}
By assumption, the $B_{i}$ are domains. Furthermore, they are finite over $\hat{A}$, since the completion $\hat{B}$ is finite over $\hat{A}$. An easy check using integral equations then shows that %
$K(B_{i})=K(\hat{A})\otimes_{\hat{A}}{B_{i}}$. %
 We now compare idempotent factors in $\prod{}K(\hat{A})[x]/(f_{i}^{r_{i}})$ and $\prod{(K(\hat{A})\otimes_{\hat{A}}{B_{i}})}\simeq\prod{K(B_{i})}$ and find that $r_{i}=1$ for every $i$ and $K(\hat{A})[x]/(f_{i})\simeq{K(B_{i})}$. To show that this is a $K(\hat{A})$-isomorphism, we note that it arises from the commutative diagram %
  \begin{center}
\begin{tikzcd}
\hat{A}\arrow{r} \arrow{dr}&\hat{C}\arrow{d}\\
& \hat{B}
\end{tikzcd}
\end{center}
Tensoring this with $K(\hat{A})$ then gives the desired commutativity. %

\end{proof}

\begin{rem}
If the primes $\mathfrak{m}_{i}$ in Theorem \ref{KummerDedekind} are of height one, then the corresponding completed rings are discrete valuation rings. %
Using this, we directly obtain \cite[Theorem 3.8]{Ste2}. 
In this paper we will also be interested in finite extensions of the ring $A=R[u,v]/(uv-\pi^{n})$ for $R$ a complete discrete valuation ring. %
The corresponding algebras are excellent by the considerations in Remark \ref{AssumptionsGalois}, so we see that %
conditions of Theorem \ref{KummerDedekind} are always satisfied here.
\end{rem}

\subsection{Galois extensions and a generalization of Dedekind's theorem}\label{GaloisSection}

In this section, we compare the extensions from Section \ref{FiniteNormalExtensions} to their Galois counterparts. In doing this, we find that we can compute the extensions of a prime ideal using the orbits of the roots of a polynomial under the corresponding absolute decomposition group. A careful study of this action then also gives a generalization of  
Dedekind's theorem on cycles in Galois groups. We give several examples to show how this generalization can be used in practice. %

Let %
$(A,\mathfrak{m})$ be a local Noetherian normal domain with field of fractions $K$ and let $L$ be a finite separable extension of $K$. %
The normalization $B$ of $A$ inside $L$ is then a finite $A$-module by \cite[Chapter 4, Proposition 1.25]{liu2}. %
\begin{mydef}{\bf{[Integrally closed finite extensions]}}\label{IntegrallyClosedExtension}
Let $A$ and $B$ be as above. We say that $B$ is an integrally closed finite extension of $A$. If $L$ is furthermore Galois, then we say that $B$ is an integrally closed finite Galois extension. If $f\in{K[x]}$ is an irreducible polynomial with associated field extensions $K\subset{}L=K[x]/(f)\subset{\overline{L}}$ and normalization $N(A,\overline{L})=\overline{B}$, then we say that $\overline{B}\supset{A}$ is the %
{\it{integrally closed finite Galois extension}} associated to $f$ over $A$. %
\end{mydef} 

Let $(A,\mathfrak{m})$ be a local Noetherian domain as above and consider an integrally closed finite extension $A\subset{B}$ with Galois closure $\overline{B}$ and Galois group $G$. The corresponding fields are denoted by $K\subset{L}\subset{\overline{L}}$. Let $\mathfrak{m}_{i}$ be a prime in $\mathrm{Spec}(B)$ lying above $\mathfrak{m}$ as in Proposition \ref{CompletionDecomposition}, and let $\overline{\mathfrak{m}}_{i}$ be a prime in $\mathrm{Spec}(\overline{B})$ lying over $\mathfrak{m}_{i}$. Consider the decomposition group $D_{\overline{\mathfrak{m}}_{i}}=\{\sigma\in{G}:\sigma(\overline{\mathfrak{m}}_{i})=\overline{\mathfrak{m}}_{i}\}$. This group acts on the completed ring $\overline{B}_{\overline{\mathfrak{m}}_{i}}$ and we have the equality $(\overline{B}_{\overline{\mathfrak{m}}_{i}})^{D_{\overline{\mathfrak{m}}_{i}}}=A_{\mathfrak{m}}$ by \cite[Chapter 4, Exercise 3.18]{liu2}. %
Note that for every $\sigma\in{D_{\overline{\mathfrak{m}}_{i}}}$, there is a commutative diagram %
 \begin{center}
\begin{tikzcd}
\overline{B}\arrow[r, "\sigma"] \arrow{d}& \overline{B} \arrow{d} \\ %
\overline{B}_{\overline{\mathfrak{m}}_{i}} \arrow[r,"\sigma"] & \overline{B}_{\overline{\mathfrak{m}}_{i}},
\end{tikzcd}
\end{center}
which arises from the corresponding diagram with $\overline{B}/\overline{\mathfrak{m}}^{k}_{i}$. %

\vspace{0.1cm}

By our assumption in \ref{AssumptionsGalois}, the completion of $\overline{B}$ at $\overline{\mathfrak{m}}_{i}$ is again a domain. We can thus consider its field of fractions, which we denote by $\overline{L}_{\overline{\mathfrak{m}}_{i}}$. We similarly write $K_{\mathfrak{m}}$ for the field of fractions of the completion of $A$ at $\mathfrak{m}$. The decomposition group $D_{\overline{\mathfrak{m}}}$ also acts on $\overline{L}_{\overline{\mathfrak{m}}_{i}}$ through its action on $\overline{B}_{\overline{\mathfrak{m}}_{i}}$. We now let $f$ be a monic irreducible polynomial in $K[x]$ that gives rise to the extension $K\subset{L}$ and we consider the splitting field $\overline{L}_{f,\mathfrak{m}}$ of $f$ over $K_{\mathfrak{m}}$. %

\begin{lemma}\label{LemmaGaloisGroup}
Let $\overline{L}_{f,\mathfrak{m}}$ and $\overline{L}_{\overline{\mathfrak{m}}_{i}}$ be as above. We then have a $K_{\mathfrak{m}}$-isomorphism $\overline{L}_{f,\mathfrak{m}}\simeq{}\overline{L}_{\overline{\mathfrak{m}}_{i}}$, giving a Galois extension over %
$K_{\mathfrak{m}}$ with Galois group $D_{\overline{\mathfrak{m}}_{i}}$.   %
\end{lemma}
\begin{proof}
By the aforementioned equality $(\overline{B}_{\overline{\mathfrak{m}}_{i}})^{D_{\overline{\mathfrak{m}}_{i}}}=A_{\mathfrak{m}}$ and the fact that $D_{\overline{\mathfrak{m}}_{i}}$ acts faithfully on $\overline{L}_{\overline{\mathfrak{m}}_{i}}$ (as it acts faithfully on $\overline{L}\subset{\overline{L}_{\overline{\mathfrak{m}}_{i}}}$), %
we directly find that $\overline{L}_{\overline{\mathfrak{m}}_{i}}$ is Galois over $K_{\mathfrak{m}}$ with Galois group $D_{\overline{\mathfrak{m}}_{i}}$. Since $f$ splits in $\overline{L}_{\overline{\mathfrak{m}}_{i}}$, there exists a $K_{\mathfrak{m}}$-injection $\overline{L}_{f,\mathfrak{m}}\rightarrow{}\overline{L}_{\overline{\mathfrak{m}}_{i}}$. But this must be an isomorphism: if there exists a %
$\sigma\in{D_{\overline{\mathfrak{m}}_{i}}}$ that acts trivially on the roots of $f$ (so that $\sigma$ belongs to the subgroup corresponding to $\overline{L}_{f,\mathfrak{m}}$), then $\sigma$ acts trivially on $\overline{L}$ and by Krull's theorem it also acts trivially on %
$\overline{L}_{\overline{\mathfrak{m}}_{i}}$, so $\sigma=\mathrm{id}$. We then conclude by Galois theory. %
\end{proof}

\begin{lemma}\label{FactorsOrbits}
Let $A\subset{B}\subset{\overline{B}}$ be the integrally closed finite Galois extension associated to an irreducible monic polynomial $f\in{K[x]}$ with Galois group $G$. %
 Let $D_{\overline{\mathfrak{m}}}$ be the decomposition group of a prime $\overline{\mathfrak{m}}$ above $\mathfrak{m}$. %
Then the orbits of the roots under $D_{\overline{\mathfrak{m}}}$ correspond exactly to monic irreducible factors $f_{i}$ of $f$ over $K_{\mathfrak{m}}$.
\end{lemma}

\begin{proof}
This follows directly from Lemma \ref{LemmaGaloisGroup} and Galois theory. 
\end{proof}

We now interpret the Kummer-Dedekind Theorem given in \ref{KummerDedekind} in terms of Galois groups.

\begin{pro}\label{OrbitsFactors}
There is a natural bijection between the set of orbits $D_{\overline{\mathfrak{m}}}(\alpha_{i})$ of the roots $\alpha_{i}$ under the action of $D_{\overline{\mathfrak{m}}}$ and the set of extensions $\mathfrak{m}_{i}\supset{\mathfrak{m}}$ of $\mathfrak{m}$ to $B$. More explicitly, this map takes a root $\alpha_{i}$ with corresponding embedding $\phi_{i}: L\rightarrow{\overline{L}}$ and %
sets $\mathfrak{m}_{i}:=\phi_{i}^{-1}(\overline{\mathfrak{m}})$.  %
\end{pro}
\begin{proof}
 An extension $\mathfrak{m}_{i}\supset{\mathfrak{m}}$ corresponds to an irreducible factor $f_{\mathfrak{m}_{i}}$ of $f$ over $K_{\mathfrak{m}}$ by Theorem \ref{KummerDedekind} and this gives an orbit by Lemma \ref{FactorsOrbits}. We now show that the converse map in %
 the statement of the proposition is well defined. Suppose that $\alpha_{j}$ is in the same $D_{\overline{\mathfrak{m}}}$-orbit. The elements $\alpha_{i}$ and $\alpha_{j}$ then have the same minimal polynomial over $K_{\mathfrak{m}}$. The embedding $\phi_{i}$ now induces a commutative diagram 
  \begin{center}
\begin{tikzcd}
A\arrow{r} \arrow{d} & A_{\mathfrak{m}} \arrow{d}\\
B\arrow{r} \arrow{d}& {B}_{\phi_{i}^{-1}(\overline{\mathfrak{m}})} \arrow{d} \\
\overline{B} \arrow{r}& \overline{B}_{\overline{\mathfrak{m}}}  \\
\end{tikzcd}
\end{center}
by taking completions. %
The minimal polynomial of $\overline{x}\in{{B}_{\phi_{i}^{-1}(\overline{\mathfrak{m}})}}$ over $K_{\mathfrak{m}}$ is then equal to the minimal polynomial of $\phi_{i}(\overline{x})=\alpha_{i}$ over $K_{\mathfrak{m}}$, which is $f_{\mathfrak{m}_{i}}$. By Theorem \ref{KummerDedekind}, this minimal polynomial uniquely determines the extension so we must have $\phi^{-1}_{j}(\overline{\mathfrak{m}})=\phi^{-1}_{i}(\overline{\mathfrak{m}})$. %
We thus see that the map is well defined. The same argument also shows that the composite of the above maps is the identity. Indeed, if we start with an extension $\mathfrak{m}'\supset{\mathfrak{m}}$ and choose a root $\alpha_{i}$ of $f_{\mathfrak{m}'}$, then $\phi^{-1}_{i}(\overline{\mathfrak{m}})=\mathfrak{m}'$ by the earlier argument, so the composite of the two maps is the identity.          %
 
\end{proof}

\begin{rem}
We will use this proposition to detect extensions of a prime ideal $\mathfrak{m}$ in the upcoming algorithms. Our algorithms luckily will not have to work explicitly with a Galois action in many cases. For us, it suffices to know that the Galois group of an irreducible polynomial acts transitively on the roots. This will reduce the problem to calculating irreducible factorizations over certain fields. 
\end{rem}

We now work on a generalization of another well-known theorem: Dedekind's theorem on cycles in Galois groups. Roughly speaking, this theorem says that we can deduce the existence of cycles in Galois groups for polynomials $f\in\mathbf{Z}[x]$ by factorizing $f$ over finite fields $\mathbf{F}_{p}$. We will give a version using completions and finite approximations, which we then apply to polynomials over $\mathbf{Z}$ to obtain the original version. We will also give an example over $\mathbf{C}[t]$.  %

\begin{theorem}
\label{Dedekind1}{\bf{[Dedekind I]}}
Let $A\subset{B}\subset{\overline{B}}$ be the integrally closed finite Galois extension associated to an irreducible polynomial $f\in{K[x]}$ with Galois group $G$ and roots $V$, see  %
Definition \ref{IntegrallyClosedExtension}. Let $D_{\overline{\mathfrak{m}}}$ be the decomposition group of a prime $\overline{\mathfrak{m}}$ above $\mathfrak{m}$ and consider the factorization of the action $D_{\overline{\mathfrak{m}}}\to\mathrm{Aut}(V)$ into $D_{\overline{\mathfrak{m}}}$-orbits: %
\begin{equation}\label{Factorization1}
i_{1}:D_{\overline{\mathfrak{m}}}\to\prod_{i}\mathrm{Aut}(V_{i}).%
\end{equation}

On the other hand, consider the monic irreducible factors $f_{i}$ of $f$ over $K_{\mathfrak{m}}$ with roots $V_{f_{i}}$. Then $D_{\overline{\mathfrak{m}}}$ acts on the $V_{f_{i}}$, giving an injective homomorphism
\begin{equation}\label{Factorization2}
i_{2}:D_{\overline{\mathfrak{m}}}\to\prod_{i}\mathrm{Aut}(V_{f_{i}}).
\end{equation} 

There are then canonical bijections $V_{i}\simeq{V_{f_{i}}}$ %
such that the induced isomorphism $j:\prod_{i}\mathrm{Aut}(V_{i})\to\prod_{i}\mathrm{Aut}(V_{f_{i}})$ yields $i_{2}=j\circ{i_{1}}$. %

\end{theorem}
\begin{proof}

Let $V_{i}$ be an orbit under $D_{\overline{\mathfrak{m}}}$. Using the embedding $e:\overline{B}\to\overline{B}_{\overline{\mathfrak{m}}}$, we then obtain the set $e(V_{i})\subset{\overline{B}_{\overline{\mathfrak{m}}}}$. Since the action of $D_{\overline{\mathfrak{m}}}$ commutes with the embedding $e$, we find that $e(V_{i})$ is an orbit under $D_{\overline{\mathfrak{m}}}$ of roots of $f$ . By Lemma \ref{FactorsOrbits}, we find that $e(V_{i})$ corresponds to an irreducible factor $f_{i}$ of $f$ over $K_{\overline{\mathfrak{m}}}$, so that $e(V_{i})=V_{f_{i}}$. Using again the fact that the action of $D_{\overline{\mathfrak{m}}}$ commutes with the embedding $e(\cdot{})$, we directly find that the homomorphisms $i_{1}$ and $i_{2}$ commute with $j:\prod_{i}\mathrm{Aut}(V_{i})\to\prod_{i}\mathrm{Aut}(V_{f_{i}})$. %

\end{proof}

\begin{rem}
To put Theorem \ref{Dedekind1} differently, if we can detect the action of $D_{\overline{\mathfrak{m}}}$ on the completion, then we obtain information about cycle patterns in the original Galois group. Indeed, if we know the image $i_{2}(D_{\overline{\mathfrak{m}}})$, then we can push it forward through $j^{-1}$, which gives $j^{-1}(i_{2}(D_{\overline{\mathfrak{m}}}))=i_{1}(D_{\overline{\mathfrak{m}}})$. Using the map $\prod_{i}\mathrm{Aut}(V_{i})\to\mathrm{Aut}(V)$, we then directly obtain cycles in the original Galois group $G$ through the commutative diagram

\begin{equation}
\begin{tikzcd}
D_{\overline{\mathfrak{m}}} \arrow{r} \arrow{d} & \prod_{i}\mathrm{Aut}(V_{i}) \arrow{d} \\
G \arrow{r} & \mathrm{Aut}(V).
\end{tikzcd}
\end{equation}
\end{rem}

We would now like to obtain an analogue of Theorem \ref{Dedekind1} using approximations for the roots. %
To that end, consider a local ring %
$(B,\mathfrak{m})$ %
and suppose that there is an action of a finite group $G$ on $B$:
\begin{equation}
G\rightarrow{\mathrm{Aut}(B)}.
\end{equation} 
This action is automatically continuous with respect to the $\mathfrak{m}$-adic topology, since $B$ is local. We moreover have $\sigma(\mathfrak{m}^{k})=\mathfrak{m}^{k}$ by a simple calculation. %

 \begin{mydef}{\bf{[Invariant ideals]}}
Let $G$ be a finite group acting on a local ring $B$. %
An ideal $I$ is said to be $G$-invariant if $\sigma(I)=I$ for every $\sigma\in{G}$. %
 \end{mydef}
Let $I$ be an invariant ideal in a local ring $B$. %
Every automorphism $\sigma$ now induces a well-defined ring homomorphism 
\begin{equation}
B/I\to{B/I}
\end{equation}
which fits into a commutative diagram 
\begin{equation}\label{CommutativeDiagramAction}
\begin{tikzcd}
B \arrow[r,"\sigma"] \arrow{d} & B  \arrow{d} \\ 
 B/I \arrow[r,"\sigma"]& B/I
\end{tikzcd}.
\end{equation} 
We thus have an induced action of $G$ on $B/I$ that commutes with the quotient map. %
\begin{lemma}\label{FaithfulActionQuotient}
Let $B$ be a local ring with a finite group $G$ acting on it and let $I$ be an invariant ideal.  Let $T\subset{B}$ be a finite $G$-invariant set and let $\overline{T}$ be the induced set in $B/I$.  There is then an induced action of $G$ on $\overline{T}$. Write $H=\{g:gx=x \text{ for all }x\in{T}\}=\bigcap_{x\in{T}}{\mathrm{Stab}(x)}$. There is then also an induced action of $G/H$ on $T$ and $\overline{T}$ and this action is faithful on $T$. %
If $I$ is separating for $T$, %
then the action of $G/H$ on $\overline{T}$ is also faithful. %
\end{lemma}
\begin{proof}
This follows directly from the definitions. %
\end{proof}

\begin{theorem}\label{Dedekind2a}
{\bf{[Dedekind II]}} Let $A\subset{B}\subset{\overline{B}}$ be an integrally closed finite Galois extension associated to an irreducible polynomial $f\in{K[x]}$. Let $A\subset{C}$ %
be an integrally closed finite Galois extension containing $\overline{B}$, with $\mathfrak{m}_{C}$ a prime lying over $\mathfrak{m}$. %
We write $C_{\mathfrak{m}_{C}}$ for either the localization of $C$ at $\mathfrak{m}_{C}$ or the completion of the localization with respect to $\mathfrak{m}_{C}$. Let $I\subset{C_{\mathfrak{m}_{C}}}$ be a $D_{\mathfrak{m}_{C}}$-invariant ideal %
that is separating for the roots of $f$. %
Let $V$ be the set of roots of $f$ in $C$ and %
consider the factorization of the action $D_{\overline{\mathfrak{m}}}\to\mathrm{Aut}(V)$ into orbits %
\begin{equation}
i_{1}:D_{\overline{\mathfrak{m}}}\to\prod_{i}\mathrm{Aut}(V_{i}).
\end{equation} 

Let $\overline{V}$ be the image of $V$ in $C_{\mathfrak{m}_{C}}/I$. 
There is then an action of $D_{\overline{\mathfrak{m}}}$ on $\overline{V}$. Write $\overline{V}_{i}$ for the orbits under this action. This gives an injective homomorphism 
\begin{equation}
i_{2}:D_{\overline{\mathfrak{m}}}\to\prod_{i}\mathrm{Aut}(\overline{V}_{i}).
\end{equation}  
There are then bijections $V_{i}\simeq{\overline{V}_{i}}$ such that the induced isomorphism $j:\prod_{i}\mathrm{Aut}(V_{i})\simeq{\prod_{i}\mathrm{Aut}(\overline{V}_{i})}$ satisfies $i_{2}=j\circ{i_{1}}$. 
\end{theorem}

\begin{proof}

Viewing $D_{\mathfrak{m}_{C}}$ and $D_{\overline{\mathfrak{m}}}$ as Galois groups over the completions as in Lemma \ref{LemmaGaloisGroup}, we obtain a canonical surjective map
\begin{equation}
\phi:D_{\mathfrak{m}_{C}}\to{D_{\overline{\mathfrak{m}}}}.
\end{equation}
The kernel consists of all automorphisms in $D_{\mathfrak{m}_{C}}$ that are the identity on all the roots of $f$. In the notation of Lemma \ref{FaithfulActionQuotient}, we then have $\bigcap_{\alpha\in{V}}\mathrm{Stab}(\alpha)=H=\mathrm{ker}(\phi)$. Since $I$ is separating, we find by Lemma \ref{FaithfulActionQuotient} that the action of $D_{\overline{\mathfrak{m}}}$ on $\overline{V}$ is faithful. We now repeat the proof of Theorem \ref{Dedekind1}: if $V_{i}$ is an orbit in $C$, then $e(V_{i})$ is also an orbit in $C_{\mathfrak{m}_{C}}$ since the Galois action commutes with $e:C\to{C_{\mathfrak{m}_{C}}}$. This action furthermore commutes with the quotient map $C_{\mathfrak{m}_{C}}\to{C_{\mathfrak{m}_{C}}/I}$, see Equation \ref{CommutativeDiagramAction}. As before, this gives the desired commutativity of the morphisms $D_{\overline{\mathfrak{m}}}\to\prod_{i}\mathrm{Aut}(V_{i})$ and $D_{\overline{\mathfrak{m}}}\to\prod_{i}\mathrm{Aut}(\overline{V}_{i})$.    %

\end{proof}

We now indicate how this can be used to obtain the original version of Dedekind's theorem. 
\begin{cor}{\bf{[Dedekind]}} \label{DedekindTheorem} %
Let $f\in\mathbf{Z}[x]$ be a monic irreducible polynomial and consider a prime such that $p\nmid{\Delta(f)}$. Suppose that the reduction $\overline{f}\in\mathbf{F}_{p}[x]$ factors into irreducible factors %
as $\overline{f}=\prod_{i=1}^{k}\overline{f}_{i}$, where the $\overline{f}_{i}$ have degree $d_{i}$. Then the Galois group $G$ of $f$ contains a product of $k$ disjoint cycles of length $d_{i}$.%
\end{cor}
\begin{proof}

We apply Theorem \ref{Dedekind2a} with $\overline{B}=C$ the integral closure of $\mathbf{Z}$ in the splitting field of $f$ and $\overline{\mathfrak{m}}$ a prime lying above $(p)$. Here we take $C_{\overline{\mathfrak{m}}}$ to be the localization of $\overline{B}=C$ at $\overline{\mathfrak{m}}$. The ideal $\overline{\mathfrak{m}}$ is then invariant and separating for the roots. %
Indeed, suppose that it is not separating. Then there are two roots, say $\alpha_{1}$ and $\alpha_{2}$, such that $\alpha_{1}-\alpha_{2}\in\overline{\mathfrak{m}}$. But then $\Delta:=\prod_{i<j}(\alpha_{i}-\alpha_{j})^{2}\in\overline{\mathfrak{m}}\cap\mathbf{Z}=(p)$, a contradiction. 

By Theorem \ref{Dedekind2a} we see that we can calculate the action of $D_{\overline{\mathfrak{m}}}$ on $V$ through $\overline{V}$. Note that the inertia group $I_{\overline{\mathfrak{m}}}$ acts trivially on every element of $\overline{V}$. We have $I_{\overline{\mathfrak{m}}}\subset{D_{\overline{\mathfrak{m}}}}\hookrightarrow{\mathrm{Aut}(\overline{V})}$, so %
$I_{\overline{\mathfrak{m}}}=(1)$.  %
We thus have an induced action of the cyclic group $D_{\overline{\mathfrak{m}}}/I_{\overline{\mathfrak{m}}}=\mathrm{Gal}(\mathbf{F}_{q}/\mathbf{F}_{p})=:G_{k}$ on $\overline{V}=\sqcup{\overline{V}_{i}}$, where $q=\#k(\overline{\mathfrak{m}})$. Every $\overline{V}_{i}$ is a $D_{\overline{\mathfrak{m}}}$-orbit and thus a $G_{k}$%
-orbit. We now consider the action of a generator $\sigma$ of $G_{k}$ %
on a factor $\overline{V}_{i}$ with $\#(\overline{V}_{i})=d_{i}$. We claim that this defines a $d_{i}$-cycle. Indeed, otherwise $G_{k}$ %
would leave a subset of $\overline{V}_{i}$ invariant, which means that $\overline{V}_{i}$ is not a $G_{k}$-%
orbit. We thus see that the image of $\sigma$ in $\prod_{i}\mathrm{Aut}(\overline{V}_{i})$ gives a product of $d_{i}$-cycles. As in Lemma \ref{FactorsOrbits}, every $G_{k}$%
-orbit $\overline{V}_{i}$ corresponds to an irreducible factor $\overline{f}_{i}$ of $\overline{f}$ over $\mathbf{F}_{p}$. We thus have $d_{i}=\mathrm{deg}(\overline{f}_{i})$, which gives the desired statement.   

\end{proof}

We now give a geometric example of Theorem \ref{Dedekind2a}. 

\begin{exa}
Consider the polynomial 
\begin{equation}
f(x)=x^{p}+t(x+1)
\end{equation}
over $\mathbf{C}(t)$, where $p$ is a prime number. 
The completion of $\mathbf{C}[t]$ at $(t)$ is the %
ring of Laurent formal power series $R=\mathbf{C}[[t]]$. Over %
$R'=\mathbf{C}[[t]][t^{1/p}]$, we then find %
\begin{equation}
f_{1}(x):=1/t\cdot{}f(t^{1/p}x)=x^{p}+t^{1/p}x+1.
\end{equation}    %
Its reduction modulo $\mathfrak{m}'=(t^{1/p})$ has $p$ distinct roots over $\mathbf{C}=R'/\mathfrak{m}'$, so we can use Hensel's Lemma to lift these to $R'$. This means that $f_{1}$ and $f$ split over $R'$.  %
We now find that the  %
Newton-Puiseux series of the roots $\alpha_{i}\in{R'}$ start with 
\begin{equation}
\alpha_{i}=\zeta^{i}_{p}t^{1/p}+\mathcal{O}(t^{2/p}).
\end{equation}
Here $\zeta_{p}$ is a primitive $p$-th root of unity in $\mathbf{C}$. Let $L$ be a %
Galois extension of $\mathbf{C}(t)$ that contains the splitting field of $f$ and $\mathbf{C}(t)(t^{1/p})$. We write $C$ for the integral closure of $\mathbf{C}[t]$ in $L$ and $\mathfrak{m}_{C}$ for a prime extending $(t)$. We then directly see that $\mathfrak{m}^{2}_{C}$ is separating for the $\alpha_{i}$. We apply Theorem \ref{Dedekind2a} with $C_{\mathfrak{m}_{C}}$ the completion of $C$ at $\mathfrak{m}_{C}$ %
and see that we can calculate the Galois action on the approximations $\zeta^{i}_{p}t^{1/p}$ inside the completion $C_{\mathfrak{m}_{C}}$. The Galois group of $\mathbf{C}((t))\subset{\mathbf{C}((t))(t^{1/p})}$ (which is a quotient of $D_{\mathfrak{m}_{C}}$) is cyclic of order $p$, being generated by the automorphism $\sigma$ with $\sigma(t^{1/p})=\zeta_{p}t^{1/p}$. This automorphism acts as a $p$-cycle on the approximations for the $\alpha_{i}$. Using Theorem \ref{Dedekind2a} we see that %
the Galois group of $f$ contains a $p$-cycle. 

A quick calculation shows that the discriminant of $f(x)$ contains another linear factor and the ramification above this point is simple of degree $2$. We then similarly deduce that the Galois group contains a $2$-cycle. Since $\mathrm{Gal}(f)$ contains an $p$-cycle and a $2$-cycle, we deduce that the Galois group is $S_{p}$ (here we use that $p$ is a prime number). The Galois group of this polynomial over $\mathbf{Q}(x)$ is then also $S_{p}$. By Hilbert's irreducibility theorem, we then find that for infinitely many values $t_{0}\in\mathbf{Q}$ of $t$, the Galois group of the polynomial $x^{p}+t_{0}x+t_{0}$ over $\mathbf{Q}$ is $S_{p}$ (one value for which this of course does not hold is $t_{0}=0$).

\end{exa}

\subsection{Henselizations}\label{HenselizationsSection}

In this section we discuss some generalities on Henselizations. We also quickly mention how the results in Sections \ref{FiniteNormalExtensions} and \ref{GaloisSection} can be generalized to the context of Henselizations. %

We start with a short categorical introduction on Henselizations. For any commutative local ring $(R,\mathfrak{m})$, we can define its (strict) Henselization as colimits of a certain filtered diagrams of \'{e}tale algebras, see \cite[\href{https://stacks.math.columbia.edu/tag/0BSK}{Section 0BSK}]{stacks-project} for a general introduction from this point of view. To define the Henselization for instance, we consider the category of pairs $(S,\mathfrak{m}_{S})$, where $S$ is an \'{e}tale $R$-algebra and $\mathfrak{m}_{S}$ is a prime ideal mapping to $\mathfrak{m}$ with residue field $k(\mathfrak{m}_{S})=k(\mathfrak{m})$. This category $\mathcal{I}$ is filtered, so we can find a directed set $I_{0}$ with category $\mathcal{I}_{0}$ and a cofinal functor $\mathcal{I}_{0}\to\mathcal{I}$. %
We can then consider the functor $\mathcal{F}_{h}:\mathcal{I}\to(\mathrm{CRings})$ given by $(S,\mathfrak{m}_{S})\mapsto{S}$. This functor has a colimit if and only if the composite functor $\mathcal{I}_{0}\to(\mathrm{CRings})$ has a colimit. But $(\mathrm{CRings})$ is cocomplete, so this colimit exists. We call this object the Henselization of $(R,\mathfrak{m})$ and we denote it by $R^{\mathrm{h}}$. The strict Henselization is defined similarly using triples $(S,\mathfrak{m}_{S},\alpha)$, where $\alpha:k(\mathfrak{m}_{S})\to{k(\mathfrak{m})^{\mathrm{sep}}}$ is a $k(\mathfrak{m})$-embedding of residue fields and $S$ and $\mathfrak{m}_{S}$ are as before. Here the colimit is denoted by $R^{\mathrm{sh}}$. The strict Henselization can also be thought of as the local ring of the corresponding geometric point with respect to the \'{e}tale topology on $R$, see \cite[\href{https://stacks.math.columbia.edu/tag/04HX}{Lemma 04HX}]{stacks-project}. The Henselization is then similarly the local ring with respect to the Nisnevich topology.

We now assume that $R$ is a normal domain. This holds if and only if $R^{\mathrm{h}}$ or $R^{\mathrm{sh}}$ is a normal domain. To work with these rings in practice, we can use the following lemma.
\begin{lemma}\label{InjectionHenselization}
Let $(S,\mathfrak{m}_{S})$ be a pair as in the definition of the Henselization of $R$ and suppose that $S$ is a domain. Then the natural map $S\rightarrow{R^{\mathrm{h}}}$ is an injection. The same holds for the strict Henselization. 
\end{lemma}

\begin{proof}
Let $f\in{S}$. We can represent its image in the Henselization as an equivalence class $(S,\mathfrak{m}_{S})$. Suppose that $f$ maps to zero in $R^{\mathrm{h}}$. We can then find a commutative diagram 
\begin{equation}
\begin{tikzcd}
(R,\mathfrak{m})\arrow{r} \arrow{dr} & (S,\mathfrak{m}_{S}) \arrow{d} \\
{} & (B,\mathfrak{m}_{B})
\end{tikzcd}
\end{equation}
such that $f$ maps to zero in $B$.  We can localize $B$ to make it local so that it becomes a normal domain by %
\cite[\href{https://stacks.math.columbia.edu/tag/025P}{Tag 025P}]{stacks-project}. %
Since $B$ is \'{e}tale over $R$, we have that the map $S\to{B}$ is \'{e}tale by \cite[\href{https://stacks.math.columbia.edu/tag/00U7}{Lemma 00U7}]{stacks-project} and thus flat. We conclude that the map of localizations $S_{\mathfrak{m}_{S}}\to{B_{\mathfrak{m}_{B}}}$ is faithfully flat and thus injective. We then have a commutative diagram
\begin{equation}
\begin{tikzcd}
S_{\mathfrak{m}_{S}} \arrow{r} & B_{\mathfrak{m}_{B}}\\
S \arrow{u} \arrow{r} & B \arrow{u}
\end{tikzcd}
\end{equation}
Since $S$ is a domain, we have that the map $S\rightarrow{S_{\mathfrak{m}_{S}}}$ is injective and thus $S\to{B_{\mathfrak{m}_{B}}}$ is also injective. We conclude that $f=0$ and therefore $S\to{R^{\mathrm{h}}}$ is injective.

\end{proof} 

\begin{rem}
In our algorithms these \'{e}tale algebras $S$ turn up naturally from lifts of factorizations over  residue fields. %
By choosing a prime ideal $\mathfrak{m}_{S}$ lying over $\mathfrak{m}$ together with some embedding $\alpha$, we then obtain an embedding $S\to{R^{\mathrm{sh}}}$. %
\end{rem}

For local normal domains $(R,\mathfrak{m})$, there is another well-known definition of a (strict) Henselization. We start with a separable closure $K\subset{K^{{\mathrm{s}}}}$ of the fraction field of $R$ and consider the integral closure $R^{\mathrm{s}}$ of $R$ inside $K^{\mathrm{s}}$. By the going-up theorem, we can find a prime $\mathfrak{m}^{\mathrm{\mathrm{s}}}$ lying above $\mathfrak{m}$. We then consider the decomposition and inertia subgroups of $\mathfrak{m}^{\mathrm{s}}$ in the absolute Galois group $G_{K}$:
\begin{align*}
D_{\mathfrak{m}}=D_{\mathfrak{m}^{\mathrm{s}}/\mathfrak{m}}:=&\{\sigma\in{G}_{K}:\sigma(\mathfrak{m}^{\mathrm{s}})=\mathfrak{m}^{\mathrm{s}}\},\\
I_{\mathfrak{m}}=I_{\mathfrak{m}^{\mathrm{s}}/\mathfrak{m}}:=&\{\sigma\in{D_{\mathfrak{m}}:\sigma(x)-x\in{\mathfrak{m}^{\mathrm{s}}} \,\text{ for all }x\in{{R}^{\mathrm{s}}}}\}.
\end{align*}
Since the local ring $R^{\mathrm{s}}_{\mathfrak{m}^{\mathrm{s}}}$ is (strictly) Henselian, we have induced injections %
\begin{equation}
R\rightarrow{R^\mathrm{h}}\rightarrow{R^{\mathrm{sh}}}\rightarrow{R^{\mathrm{s}}_{\mathfrak{m}^{\mathrm{s}}}},
\end{equation}
see \cite[\href{https://stacks.math.columbia.edu/tag/0BSD}{Section 0BSD}]{stacks-project}. Under these embeddings, we have 
\begin{align*}
(R^{\mathrm{s}}_{\mathfrak{m}^{\mathrm{s}}})^{D_{\mathfrak{m}}}&=R^\mathrm{h},\\
(R^{\mathrm{s}}_{\mathfrak{m}^{\mathrm{s}}})^{I_{\mathfrak{m}}}&=R^\mathrm{sh}.
\end{align*}
As we saw above, the (strict) Henselization of a ring can be constructed without choosing any embeddings, so it is a more canonical object in this sense.

\begin{rem}
We now discuss the results of Sections \ref{FiniteNormalExtensions} and \ref{GaloisSection} in the context of Henselizations. We will assume that $R$ is Noetherian so that the normalization morphisms are finite. The results of Section \ref{FiniteNormalExtensions} continue to hold by \cite[Proposition 18.6.8, Page 139]{EGA4} and the fact that $R^{\mathrm{h}}$ is faithfully flat over $R$. We do not however have to assume that the Henselizations of the local rings are domains, since this is automatic. For the results of Section \ref{GaloisSection}, one first proves the analogue of Lemma \ref{LemmaGaloisGroup}. %
As before, the corresponding extension is Galois with Galois group $D_{\mathfrak{m}}=D_{\overline{\mathfrak{m}}/\mathfrak{m}}$.  %
The characterization of primes in terms of $D_{\mathfrak{m}}$-orbits follows similarly, as do the generalized %
Kummer-Dedekind theorems.\footnote{The proofs of some of these statements can be found in older versions of this manuscript.} 
\end{rem}

\subsubsection{Computational aspects of Henselizations}

We now discuss some computational aspects of Henselizations. Let $(A,\mathfrak{m})$ be a local Noetherian ring with Henselization $A^\mathrm{h}$. This ring is again local with maximal ideal $\mathfrak{m}_\mathrm{h}=\mathfrak{m}A^\mathrm{h}$, so we can consider its $\mathfrak{m}_\mathrm{h}$-adic completion. By \cite[\href{https://stacks.math.columbia.edu/tag/06LJ}{Lemma 06LJ}]{stacks-project}, we then have that $A^{\vee}\simeq{(A^{\mathrm{h}})^{\vee}}$. Using Proposition \ref{RegularExpansions}, we can then represent elements of the Henselization as $\mathfrak{m}$-adic power series if $A$ is regular. For instance, if $A^{\vee}\simeq{R[[u,v]]/(uv-\pi)}$ we can express elements of the Henselization as power series in the variables $u$ and $v$. If the residue field of the discrete valuation ring $R$ is algebraically closed, then this also gives the strict Henselization.

We now turn to %
the strict Henselization of a discrete valuation ring $A$ with maximal ideal  %
$\mathfrak{m}$,  residue field $k$ and uniformizer $\pi$. The ring $A^{\mathrm{sh}}$ is again a discrete valuation ring with uniformizer $\pi$ by \cite[\href{https://stacks.math.columbia.edu/tag/0AP3}{Lemma 0AP3}]{stacks-project}. %
We choose a set of representatives $S$ in $R^{\mathrm{sh}}$ of the residue field $k^{\mathrm{sep}}$. By Proposition \ref{RegularExpansions}, we can find unique $c_{i}\in{S}$ for every element $z\in{R^{\mathrm{sh}}}$ such that %
\begin{equation}\label{RepresentationHenselization2}
z=\sum_{i=0}^{\infty}c_{i}\pi^{i}. 
\end{equation} 
If $z$ has a representative in some %
\'{e}tale algebra $B$, then all of the $c_{i}$ can be chosen in $B$. Instead of fixing a set of representatives at the start, we will create larger sets $S_{B}$ using \'{e}tale algebras and then assume that $S_{B}\subset{S}$. 

\begin{rem}\label{CharacteristicZeroRepresentation}
In this paper we are mostly interested in algebras of the form $A:=R[u,v]/(uv-\pi^{n})$, where $R$ is a fixed discrete valuation ring. The codimension one primes %
$(u,\pi)$ and $(v,\pi)$ are regular, so their localizations give discrete valuation rings. %
In these valuation rings %
$\pi$ is a uniformizer and the residue fields are $k(v)$ and $k(u)$ respectively. Elements in the strict Henselization can then be represented as in Equation \ref{RepresentationHenselization2}, where the $c_{i}$ are lifts of elements in ${k(v)}^{\mathrm{sep}}$ and ${k(u)}^{\mathrm{sep}}$.  

We now suppose that $R=\overline{\mathbf{Q}}[[t]]$ and we fix the prime $\mathfrak{p}=(v,t)$ in $A=R[u,v]/(uv-t^{n})$ with localization $A_{\mathfrak{p}}$ and Henselization $A^{\mathrm{h}}_{\mathfrak{p}}$. %
In this case, the residue field at $\mathfrak{p}$ is $\overline{\mathbf{Q}}(u)$ and there is a natural injection of $\overline{\mathbf{Q}}(u)$ into $A_{\mathfrak{p}}\subset{A^{\mathrm{h}}_{\mathfrak{p}}}$. 
By \cite[Theorem 2, Page 33]{Ser1}, the completion of $A^{\mathrm{h}}_{\mathfrak{p}}$ is isomorphic to $\overline{\mathbf{Q}}(u)[[t]]$. Furthermore, if we take the completion of the strict Henselization of $A_{\mathfrak{p}}$, then it is isomorphic to $\overline{\mathbf{Q}(u)}[[t]]$. In other words, we can perform arithmetic in \'{e}tale extensions of $A_{\mathfrak{p}}$ by doing polynomial arithmetic over finite extensions $L\supset{\overline{\mathbf{Q}}(u)}$ of the residue field $\overline{\mathbf{Q}}(u)$. This greatly simplifies computations in these rings. If $R$ is of mixed characteristic then this won't work, so we lift coefficients defined over finite extensions of the residue field to an \'{e}tale algebra over $R$ and do our computations there. This problem is discussed in more detail in Sections \ref{LiftResidueZero} and \ref{LiftCharacteristicp}.  %
\end{rem}

\begin{rem}\label{ConstructionEtaleAlgebras}
We can construct connected \'{e}tale algebras from finite extensions of the residue field $k$ of $A$ using the following procedure. %
Consider an irreducible separable polynomial $\overline{g}\in{k[x]}$, giving rise to %
a finite field extension
\begin{equation}
k\subset{k[x]/(\overline{g})}.
\end{equation}
Let $g\in{A[x]}$ be a lift of $\overline{g}$ of the same degree. 
The extension
\begin{equation}
A\subset{A[x]/(g)}=:A_{1}
\end{equation} 
is then \'{e}tale. %
Note furthermore that there is only one extension of $\mathfrak{m}$ to $A_{1}$, namely $\mathfrak{m}A_{1}$. %
If we now choose an embedding of $k(\mathfrak{m}A_{1})=k[x]/(\overline{g})\rightarrow{k^{\mathrm{sep}}}$, then we obtain an injection $A_{1}\rightarrow{A^{\mathrm{sh}}}$. This choice will come up again when we discuss the algorithms, see Remark \ref{ChoiceEmbedding}. If we have a set of representatives $S$ of $A$, then any set $\{x_{i}\}$ of elements in $A_{1}$ that maps to a basis of $k[x]/(\overline{g})$ over $k$ gives a set of representatives $S_{A,x_{i}}:=\{sx_{i}:s\in{S}\}$. One can for instance choose the standard basis %
$\{1,x,...,x^{d-1}\}$, where $n=\mathrm{deg}(g)$.
\end{rem}

\subsection{Chains of prime ideals}\label{sec:ChainPrime}

In Proposition \ref{OrbitsFactors},  %
we expressed the extensions of a prime ideal $\mathfrak{m}$ in terms of the $D_{\mathfrak{m}}$-orbits of the roots of a polynomial. %
In this section we give a similar Galois-theoretic characterization for extensions of chains of prime ideals. %
Let $\mathfrak{p}\subset{\mathfrak{m}}$ be prime ideals in a normal Noetherian domain $A$ and consider the integral closure $\overline{A}$ of $A$ in the separable closure of the fraction field $K(A)$. We can then lift this chain $\mathfrak{p}\subset{\mathfrak{m}}$ to a chain $\overline{\mathfrak{p}}\subset{\overline{\mathfrak{m}}}$ in $\overline{A}$. We denote the corresponding absolute inertia groups by $I_{\mathfrak{p}}$ and $I_{\mathfrak{m}}$. These are related as follows.  
\begin{lemma}
$I_{{\mathfrak{p}}}\subset{I_{{\mathfrak{m}}}}$.
\end{lemma}
\begin{proof}
Let $\sigma\in{I_{\mathfrak{p}}}$. Consider the injection 
\begin{equation}
\overline{A}/\overline{\mathfrak{p}}\to\overline{A}_{\overline{\mathfrak{p}}}/\overline{\mathfrak{p}}=k(\overline{\mathfrak{p}}). %
\end{equation}
By the bijection between ideals containing $\overline{\mathfrak{p}}$ and ideals of $\overline{A}/\overline{\mathfrak{p}}$, %
we then find that $\sigma$ fixes $\overline{\mathfrak{m}}$, so $\sigma\in{D_{\mathfrak{m}}}$. We then have that $\sigma$ acts trivially on $\overline{A}/\overline{\mathfrak{m}}$. %
Indeed, let $p(x)\in{\overline{A}/\overline{\mathfrak{m}}}$ and consider a lift $x\in{\overline{A}}$. We then have $\sigma(x)-x\in\overline{\mathfrak{p}}\subset{\overline{\mathfrak{m}}}$ so $\sigma$ acts trivially. This then also implies that $\sigma$ acts trivially on $\mathrm{Frac}(\overline{A}/\overline{\mathfrak{m}})=%
k(\overline{\mathfrak{m}})$, as desired. %
\end{proof}

\begin{rem}\label{StrictHenselizationInclusion}
Let $A$ be a Noetherian normal domain with prime ideals $\mathfrak{m}\supset{\mathfrak{p}}$. The lemma then implies that we have maps
\begin{equation}
A\subset{}\overline{A}^{I_{\mathfrak{m}}}\subset{\overline{A}^{I_{\mathfrak{p}}}}\subset{\overline{A}.}
\end{equation}
We thus have a map of strict Henselizations $A^\mathrm{sh}_{\mathfrak{m}}\to{A^\mathrm{sh}_{\mathfrak{p}}}$. The groups %
$D_{\mathfrak{m}}/I_{\mathfrak{m}}=\mathrm{Gal}(k(\mathfrak{m})^{\mathrm{sep}})$ and $D_{\mathfrak{p}}/I_{\mathfrak{p}}=\mathrm{Gal}(k(\mathfrak{p})^{\mathrm{sep}}/k(\mathfrak{p}))$ then act naturally on $A^\mathrm{sh}_{\mathfrak{m}}$. %
In our algorithms $\mathfrak{m}$ will correspond to an edge and $\mathfrak{p}$ will correspond to an adjacent vertex. The computations for any such pair will then take place in a finite Kummer extension of $K^{\mathrm{sh}}_{\mathfrak{m}}$ and this Kummer extension will lie inside $K^{\mathrm{sh}}_{\mathfrak{p}}$.  %

\end{rem}

We would now like to be able to characterize extensions of chains $\mathfrak{p}\subset{\mathfrak{m}}$ using orbits. We will do this under an extra assumption. %

\begin{theorem}\label{Dedekind3}
Let $A\subset{B}\subset{\overline{B}}$ be an integrally closed Galois extension associated to a polynomial $f\in{K[x]}$ 
and write $\overline{\mathfrak{p}}\subset\overline{\mathfrak{m}}\subset\overline{B}$ for a chain of prime ideals extending $\mathfrak{p}\subset{\mathfrak{m}}\subset{A}$. Assume that %
$\overline{\mathfrak{p}}$ is the unique extension of $\mathfrak{p}$ lying under $\overline{\mathfrak{m}}$. %
There is then a bijection between the set of extensions %
$\mathfrak{p}'\subset{\mathfrak{m}'}$ of $\mathfrak{p}\subset{\mathfrak{m}}$ to $B$ and the set of $D_{\mathfrak{m}}$-orbits of the roots of $f$. %
\end{theorem} 

\begin{proof}
An extension $\mathfrak{m}'$ of $\mathfrak{m}$ gives a unique $D_{\mathfrak{m}}$-orbit by Proposition \ref{OrbitsFactors}. The condition in the theorem then directly implies that $D_{\mathfrak{m}}\subset{D_{\mathfrak{p}}}$, so we obtain a $D_{\mathfrak{p}}$-orbit and thus an extension $\mathfrak{p}'$ of $\mathfrak{p}$. We take a root $\alpha$ in the $D_{\mathfrak{m}}$-orbit and consider the corresponding embedding $\phi:L\to\overline{L}$. By construction, we then have $\mathfrak{m}'\supset{\mathfrak{p}'}$. Conversely, suppose that we have a chain $\mathfrak{m}'\supset{\mathfrak{p}'}$. We consider a root $\alpha$ in the corresponding $D_{\mathfrak{m}}$-orbit together with its embedding $\phi$. We claim that $\mathfrak{p}'':=\phi^{-1}(\overline{\mathfrak{p}})=\mathfrak{p}'$. Indeed, otherwise we would have two prime ideals $\mathfrak{p}',\mathfrak{p}''\subset{\mathfrak{m}'}$ lying over $\mathfrak{p}$ and these prime ideals can be lifted to two distinct primes lying under $\overline{\mathfrak{m}}$ by the going-down theorem, a contradiction.

\end{proof}

\begin{rem}
We will see in \ref{RemarkGaloisActionMonodromy} that the assumption in Theorem \ref{Dedekind3} always holds for coverings of semistable models. 
\end{rem}

This classification of extensions of chains can give rise to monodromy as follows. Let $\mathcal{X}$ be a normal integral scheme. We then assign an undirected graph to $\mathcal{X}$ using the concept of specialization. %
The vertex set of this graph is $\mathcal{X}$ and %
we assign an edge between two vertices if one is a specialization (or generization) of the other. In this graph, we can consider topological loops. If we have a finite \'{e}tale covering $\mathcal{Y}\to{\mathcal{X}}$ of normal schemes, this can give rise to monodromy in the sense of algebraic topology: if we traverse a lift of a loop, then we need not end up at the original starting point. If we moreover allow algebraic loops (in the sense of Galois categories), then this can give monodromy even if the underlying graph is a tree. For explicit examples, see  \ref{MonodromyExampleProjectiveLine1} and \ref{HannahExample1}. Our algorithm tackles this problem by calculating compatible $\mathfrak{m}$-adic and $\mathfrak{p}$-adic approximations. We then use Theorem \ref{Dedekind3} to find the corresponding extensions of chains. %

\begin{rem}\label{MonodromyExampleProjectiveLine1}
{\bf{[Monodromy]}}
We now give an example of a covering of the projective line $X\to\mathbf{P}^{1}_{K}$ with a tropicalization that is not determined by local data. %
We will use the material on metrized complexes in \cite{ABBR1} to create these. 
Consider the coverings of metric graphs indicated in Figure \ref{Monodromy5}. We can make these into morphisms of metrized complexes as follows: we assign a degree three covering ramified over three points on the outside vertices. Two of the nontrivial ramification indices here are $2$ and the other is $3$. We furthermore assign %
the standard Kummer covering of degree two to the vertices in the middle. By \cite[Theorem 7.4]{ABBR1}, these coverings can be lifted to coverings of curves. %
These coverings are locally the same: %
there are three edges lying above the middle edge and two over the outside edges. Note that this does not determine the graph-theoretical structure of the skeleton, as the graphs in these two examples are non-isomorphic. We say that this covering has {\it{monodromy}}.  %
The terminology is taken from complex analysis, where one finds similar coverings. %
 For instance, the $2:1$ covering $\mathbf{C}^{*}\to\mathbf{C}^{*}$ given by $z\mapsto{z^2}$ is locally trivial, but not globally. %
 We will find an explicit example of monodromy over the projective line in Section \ref{LocalPresentations1}. 

\begin{figure}[H]
\centering
\includegraphics[height=6cm]{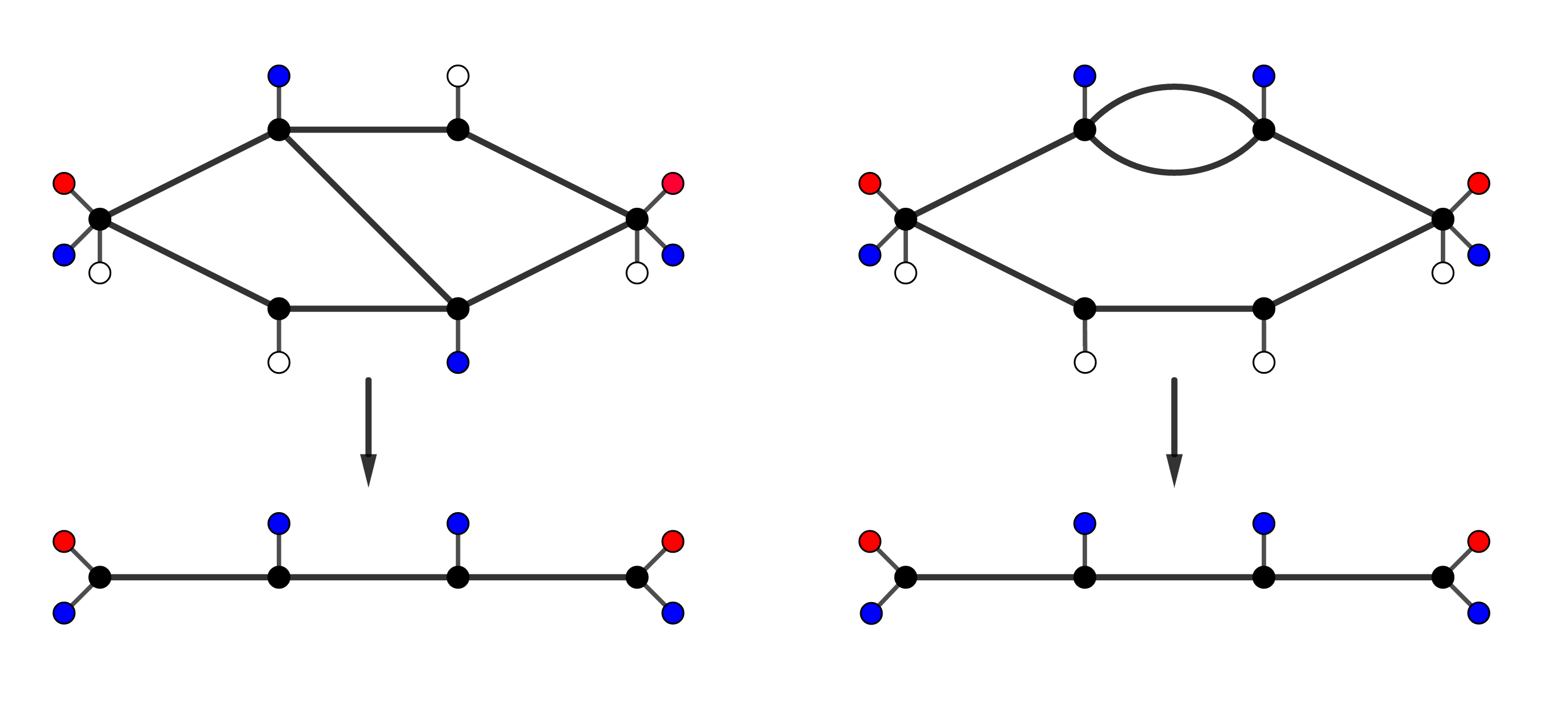}
\caption{\label{Monodromy5} An example of {\it{tropical monodromy}} of metrized complexes: these coverings are locally the same, but not globally. The blue vertices in the graphs on top have ramification index $2$, the red vertices have ramification index $3$ and the white vertices have ramification index $1$. } %
\end{figure}

\end{rem}

\section{The tame simultaneous semistable reduction theorem}
\label{SectionMainTheorem}

In this section, we prove the tame simultaneous semistable reduction theorem.  %
We first determine the spectra of the rings $R[[u]]$ and $R[[u,v]]/(uv-\pi^{n})$. %
We then use this together with Abhyankar's Lemma and purity of the branch locus to prove Theorem \ref{NPTheorem} and an analogous theorem for $R[[u]]$. In Section \ref{ProofSemistable} we combine these results to prove Theorem \ref{MainThm1}. %

\subsection{Smooth and ordinary double points}\label{ModelsCurves} %

We start by determining the spectra of smooth and ordinary double points in %
Lemmas \ref{SmoothPoints} and \ref{OrdinaryDoublePoints}. %

\begin{lemma}\label{SmoothPoints}
Let $A=R[[u]]$, where $R$ is a complete discrete valuation ring with uniformizer $\pi$.  Then the following hold:
\begin{enumerate}
\item $A$ is a complete, Noetherian, local unique factorization domain of dimension two. %
\item We have
 \begin{equation}
\mathrm{Spec}(A)=\{(u,\pi)\}\cup\{(\pi)\}\cup{\{(0)\}}\cup{T},%
\end{equation}
where $T$ is the set of ideals generated by irreducible monic polynomials $f$ in $R[u]$ whose roots $\alpha$ in $\overline{K}$ satisfy $v(\alpha)>0$.
\end{enumerate}
\end{lemma}
\begin{proof}
The first follows from some standard commutative algebra. For the proof of the second we give an outline and leave some of the details to the reader for the sake of brevity. We start with the natural map $\mathrm{Spec}(A)\to\mathrm{Spec}(R)$ arising from $R\to{A}$. Points that map to $(\pi)$ under this map are said to be in the special fiber. These correspond to prime ideals of $A/(\pi)\simeq{k[[u]]}$, so there are exactly two: the maximal ideal $(u,\pi)$ and $(\pi)$. We then consider the generic fiber, consisting of the prime ideals that map to $(0)$. Since $A$ is a unique factorization domain and $\mathrm{dim}(A)=2$, we have that any such prime is of the form $(f)$ for $f$ an irreducible element of $A$. By the Weierstrass preparation theorem (see \cite[Proposition 8, Page 511]{BourbakiCommAlg}), we can assume that $f$ is an irreducible monic polynomial in $R[u]$ of the form %
\begin{equation}
f=u^{n}+b_{n-1}u^{n-1}+...+b_{0},
\end{equation}
with $b_{i}\in(\pi)$. This condition implies that the roots of $f$ in an algebraic closure of $K$ all have strictly positive valuation and one quickly obtains the desired bijection from this.

\end{proof}

\begin{lemma}\label{OrdinaryDoublePoints}
Let $A=R[[u,v]]/(uv-\pi^{n})$ for some $n\in\mathbf{Z}_{>{0}}$, where $R$ is a complete discrete valuation ring with uniformizer $\pi$. Then the following hold:
\begin{enumerate}
\item $A$ is a complete, Noetherian, local normal domain of dimension two.
\item If $n=1$, then $A$ is regular and thus a unique factorization domain.  
\item Write $\mathrm{Spec}(A)_{s}$ for the prime ideals that map to $(\pi)$  and $\mathrm{Spec}(A)_{\eta}$ for the prime ideals that map to $(0)$ under $\mathrm{Spec}(A)\to\mathrm{Spec}(R)$. Then 
\begin{equation}
\mathrm{Spec}(A)_{s}=\{(u,v,\pi),(u,\pi),(v,\pi)\}.
\end{equation}
There furthermore is a bijection 
\begin{equation}
\mathrm{Spec}(A)_{\eta}\backslash{(0)}\to\{\alpha\in\overline{K}:0<v(\alpha)<n\}/\sim.
\end{equation}
Here two elements $\alpha$ and $\beta$ are equivalent if they have the same minimal polynomial over $K$. Equivalently, there is an automorphism $\sigma\in\mathrm{Aut}(\overline{K}/K)$ such that $\sigma(\alpha)=\beta$.

\end{enumerate}
\end{lemma}
\begin{proof}
The first two statements follow from basic commutative algebra; we give an outline of how to prove the third. %
The classification of the primes in the special fiber follows as in the previous lemma. %
For the primes in the generic fiber, we %
 consider the localization %
$S^{-1}A$, where $S$ is the multiplicative set generated by $u$, $v$ and $\pi$. Since $A$ is local and two-dimensional, we easily find that $\mathrm{dim}(S^{-1}A)=1$. %
We write $A_{1}=R[[u]]$ and $A_{2}=R[[v]]$, which we consider as subrings of $A$.  %
The sets $S_{i}=A_{i}\cap{S}$ then give rise to injective localization maps $S^{-1}_{i}A_{i}\to{S^{-1}A}$. We represent a non-zero prime ideal by a homomorphism $\phi:S^{-1}A\to{L}$ for some field $L$. This induces homomorphisms $\phi_{i}:S_{i}^{-1}A_{i}\to{L}$. By Lemma \ref{SmoothPoints}, we find that $\phi(u)=\phi_{1}(u)$ and $\phi(v)=\phi_{2}(v)$ are defined over a finite extension $K'$ of $K$. They furthermore have %
 have strictly positive valuations. %
 Since $\phi(u)\phi(v)=\phi(\pi^{n})$, we then find that $0<v(\phi(u))<n$. Conversely, for any element $\alpha\in\overline{K}$ with $0<v(\alpha)<n$, there is a homomorphism $R[[u,v]]\to{\overline{K}}$ sending $u\mapsto{\alpha}$ and $v\mapsto{\pi^{n}/\alpha}$. This factors through $S^{-1}A$, giving the desired prime ideal. It is easy to see that this gives a bijection. %

\end{proof}

\subsection{Tame Kummer extensions}\label{KummerExtensions}

In this section, we study Kummer extensions of smooth and ordinary double points. %
We show that a tame extension of a smooth point that is only ramified above one generic prime is Kummer. We also show that a tame extension of an ordinary double point that is only ramified above the vertical primes is contained in a Kummer covering. %
We start by summarizing the necessary ingredients for the proofs. %

\begin{mydef}{\bf{[Tame extensions and weakly unramified extensions]}}\label{Tame}
Let $K$ be a discretely valued field with valuation $v$ and let $L$ be a finite extension of $K$. For any valuation $w$ of $L$ extending $v$, let $e(w/v)$ be the ramification index and let $k(v)\subset{k(w)}$ be the extension of residue fields. We say that an extension $w/v$ is {\it{tame}} if the ramification index $e(w/v)$ is coprime to the characteristic and the extension $k(v)\subset{k(w)}$ is separable. We say that $L$ is tame over $K$ if all of the extensions $w/v$ are tame. We say that an extension $w/v$ is {\it{weakly unramified}} if $e(w/v)=1$ and it is {\it{unramified}} if furthermore  $k(v)\subset{k(w)}$ is separable. We say that $L$ is (weakly) unramified over $K$ if every extension $w/v$ is (weakly) unramified.  
\end{mydef}

The tame extensions of a discretely valued field can be characterized using a classical result known as  %
Abhyankar's lemma. 

\begin{lemma}\label{Abhyankar}
{\bf{[Abhyankar's Lemma]}}
Let $v$ be a discrete valuation on a field $K$ and let $L$ and $E$ be two finite extensions of $K$. Writing $M$ for the composite field $L\cdot{E}$, we arrive at the following diagram of fields
\begin{center}
\begin{tikzcd}[every arrow/.append style={dash}]\label{DiagramAbhyankar}
& M \\
L \arrow{ur}{} &  \\
& E \arrow[swap]{uu}{} \\
K \arrow{uu}{} \arrow{ur}{}& %
\end{tikzcd}
\end{center}
 Let $w$ be an extension of $v$ to $M$ and let $w_{E}$ and $w_{L}$ be the restrictions of $w$ to $E$ and $L$ respectively. Suppose that $w_{L}/v$ is tamely ramified and that $e(w_{L}/v)$ divides $e(w_{E}/v)$. Then $w$ is unramified over $w_{E}$. 
\end{lemma}

\begin{proof}
The current formulation is the additive version of \cite[Chapter IV, Exercise 3]{Ste2}. For a proof in terms of discrete valuation rings, see \cite[\href{https://stacks.math.columbia.edu/tag/0BRM}{Lemma 0BRM}]{stacks-project}. 
\end{proof}

\begin{theorem}{\bf{[Purity of the branch locus]}}\label{Purity}
Let $f:X\rightarrow{Y}$ be a morphism of locally Noetherian schemes. Let $x\in{X}$ and set $y=f(x)$. Suppose furthermore that the following hold: $\mathcal{O}_{X,x}$ is normal, $\mathcal{O}_{Y,y}$ is regular, $f$ is quasi-finite at $x$ and $\mathrm{dim}(\mathcal{O}_{X,x})=\mathrm{dim}(\mathcal{O}_{Y,y})\geq{1}$. If $f$ is \'{e}tale at specializations $x'\subset{x}$, with $\mathrm{dim}(\mathcal{O}_{X,x'})=1$, then %
 $f$ is \'{e}tale at $x$. 
\end{theorem}
\begin{proof}
See \cite[\href{https://stacks.math.columbia.edu/tag/0BMB}{Lemma 0BMB}]{stacks-project}. %
\end{proof}

\begin{exa}
We give an example where purity fails. Let $A=R[u,v]/(uv-\pi^{2})$, so that $A$ is not regular at $(u,v,\pi)$. Consider the normalization $B$ of $A$ in the field extension $K(A)\subset{K(A)[z]/(z^2-u)}$. More explicitly, this normalization is generated as an $A$-module by the elements $z_{1}:=\dfrac{z}{1}$ and $z_{2}:=\dfrac{\pi}{z}$, where $z_{2}$ satisfies
\begin{equation}
z_{2}^2=v.
\end{equation} 
There is then only one prime in $\mathrm{Spec}(B)$ over $(u,v,\pi)$, namely  %
$(z_{1},z_{2},\pi)=(z_{1},z_{2})$. An easy check now shows that $B$ is \'{e}tale over every codimension one prime, but not over $(u,v,\pi)$. We conclude that the purity theorem does not hold in this case. By %
\cite[\href{https://stacks.math.columbia.edu/tag/0BJH}{Lemma 0BJH}]{stacks-project} we then also find that $A\rightarrow{B}$ is not flat. The flatness part of the result stated in \cite[Proposition 2]{BCKKGS2020} is thus incorrect. %

\end{exa}

\subsubsection{Kummer extensions of smooth points}\label{KummerSmoothSection}

We now introduce Kummer extensions of smooth points. We show that any extension that is only ramified over $(u)$ is automatically Kummer. We start with the ring of formal power series $A=R[[u]]$ over a complete discrete valuation ring $R$, with field of fractions $K(A)$. 
\begin{rem}
For the remainder of the section, we fix a compatible set of $n$-th roots of $u$ inside an algebraic closure $\overline{K(A)}$. %
We denote these $n$-roots of $u$ by $u^{1/n}$. The compatibility means that for any divisor $k$ of $n$ with $kh=n$, we have that $(u^{1/n})^{k}=u^{1/h}$. 
\end{rem}  %

\begin{mydef}
{\bf{[Kummer extensions I]}}\label{TameKummerSmooth}
Let $A:=R[[u]]$ be the ring of formal power series over a complete discrete valuation ring $R$ with algebraically closed residue field $k$ and field of fractions $K(A)$. The field $K(W)$ of all $u$-Kummer extensions is the composite of all extensions $K(A)(u^{1/n})$, where $n$ varies over all positive integers. We say that a finite subextension of $K(W)$ is a $u$-Kummer extension of $K(A)$. We will omit the $u$ when the parameter is clear from context. If the degree of the Kummer extension $L$ is coprime to the residue characteristic $\mathrm{char}(k)$, then we say that $L$ is {\it{tame}} over $K(A)$. We denote the subfield of all tame Kummer extensions of $K(W)$ by $K(W^{\mathrm{tame}})$.      %

\end{mydef}

\begin{rem}
We note that the field of all Kummer extensions $K(W)$ depends on the choice of a parameter $u$ of $K(A)$. For instance, let $\alpha\in{(\pi)\backslash\{0\}}\subset{R}$ and consider the two extensions $L_{1}=K(A)(u^{1/n})$ and $L_{2}=K(A)(\beta)$, where $\beta^{n}=u-\alpha$. %
By inspecting the ramification in both extensions, we then see that there is then no $K(A)$-automorphism $L_{1}\to{L_{2}}$. 
If we allow more general automorphisms, then these two extensions are isomorphic. Indeed, there is an automorphism of $K(A)$ that sends %
\begin{equation}
u\mapsto{u+\alpha}
\end{equation} 
and from this we obtain an induced automorphism $L_{1}\to{L_{2}}$. %
\end{rem} %

We start with some preliminary remarks about these Kummer extensions. 
The normalization of $A$ inside the Kummer extension $K(A)\subset{K(A)(u^{1/n})}=:L$ %
is isomorphic to
$B=A[y]/(y^{n}-u)$ since it is regular and thus normal. %
Moreover, for $n$ coprime to the residue characteristic we have that the extension $K(A)\subset{L}$ is Galois with Galois group $\mathbf{Z}/n\mathbf{Z}$. Indeed, $K$ is a complete local field and $k$ is algebraically closed, so $K$ contains all primitive $n$-th roots of unity for $n$ coprime to the residue characteristic by Hensel's lemma. We then easily see that the field of all tame Kummer extensions $K(W^{\mathrm{tame}})$ is Galois, with profinite Galois group  %
$\hat{\mathbf{Z}}'$, which is the inverse limit of $\mathbf{Z}/n\mathbf{Z}$ for $n$ coprime to $\mathrm{char}(k)$. %

\begin{lemma}\label{GaloisRegularKummer}
Let $K$ be a field and let $u\in{K}$ be an element such that $f:=y^{n}-u\in{K[y]}$ is irreducible and separable (in particular $n$ is coprime to the characteristic of $K$). We write %
$K(\alpha)=L=K[y]/(f)$ with $\alpha:=y+(f)$. Then any subfield $K\subset{M}\subset{L}$ is of the form $K(\alpha^{d})$ for $d$ a divisor of $n$. 
\end{lemma}

\begin{proof}
The splitting field of $f$ over $K$ is given by $L(\zeta_{n})$ for a primitive $n$-th root of unity and the Galois group can be embedded into the semi-direct product $(\mathbf{Z}/n\mathbf{Z})\rtimes{(\mathbf{Z}/n\mathbf{Z})^{*}}$ using its action on $\alpha$ and $\zeta_{n}$. %
The lemma now follows by explicitly working out the subgroup corresponding to $M$. We leave the details to the reader. %

\end{proof}

\begin{lemma}\label{LemmaSmooth}
Let $L$ be a finite separable extension of degree $n$ of the quotient field $K$ of $A=R[[u]]$ such that the normalization $B$ of $A$ inside $L$ is at most tamely ramified over $(u)$ and \'{e}tale over all other primes of codimension one. %
Then $K(B)=K(A)(u^{1/r})$ 
for some $r\in\mathbf{Z}_{>{0}}$ coprime to the residue characteristic. %

\end{lemma}

\begin{proof}

Write $e_{\mathfrak{p}_{i}/\mathfrak{p}}$ for the ramification indices of all primes $\mathfrak{p}_{i}$ lying over $\mathfrak{p}$ and let $m$ be their least common multiple. Consider the Kummer extension
\begin{equation}
K(A)\subset{K(A)(u^{1/m})=:K(W)}%
\end{equation}
with normalization $W$. 
We take the composite field $K(W')=K(W)\cdot{K(B)}$ and consider the normalization $N(A,K(W'))$ of $A$ inside this field. Since $B$ is normal, it is contained in $N(A,K(W'))$.  %
By Abhyankar's Lemma we then find that $N(A,K(W'))$ is unramified over $W$ in codimension one. Purity of the branch locus then implies that $N(A,K(W'))$ is \'{e}tale over $W$. The \'{e}tale fundamental group of $W$ is trivial since %
$W$ is a strictly Henselian ring, so we find %
$N(A,K(W'))=W$. We thus see that $B$ injects into $W$ and thus the quotient field of $B$ is contained in the tame Kummer extension $K(W)$. Since all normal domains contained in tame Kummer extensions are again Kummer by Lemma \ref{GaloisRegularKummer}, we find that $B$ is as in the statement of the lemma. %

\end{proof}

\subsubsection{Kummer extensions of ordinary double points}\label{KummerOrdinaryDoublePointsSection}

We now move to Kummer extensions of ordinary double points.  %
We characterize the tame Kummer extensions in terms of the ramification on the generic fiber and we show that the ones that are unramified in codimension one correspond to regularized Kummer extensions. %

\begin{rem}\label{CompatibleRootsDouble}
Let $A=R[[u,v]]/(uv-\pi^{n})$, where $R$ is a complete discrete valuation ring. %
As in Section \ref{KummerSmoothSection}, we fix a compatible set of $m$-th roots of $u$ and $v$ inside $\overline{K(A)}$ for all $m>0$. We write $u^{1/m}$ and $v^{1/m}$ for these. We will assume that $u^{1/n}$ and $v^{1/n}$ are chosen so that $u^{1/n}v^{1/n}=\pi$. For higher roots of $\pi$, we then set $\pi^{1/m}:=u^{1/nm}v^{1/nm}$.  
\end{rem}

\begin{mydef}{\bf{[Kummer extensions II]}}
Let $A=R[[u,v]]/(uv-\pi^{n})$ be the completion of an ordinary double point with field of fractions $K(A)$. The field $K(W)$ of all Kummer extensions of $A$ is the composite of all extensions $K(A)(u^{1/k})$ and $K(A)(v^{1/k})$, where $k$ varies over all positive integers. If $K(A)\subset{}L$ is a finite subextension of $K(W)$, then $L$ is a (finite) Kummer extension. If the degree of $L$ over $K(A)$ is coprime to the residue characteristic, then we say that $L$ is {\it{tame}}. The subfield of all tame Kummer extensions inside $K(W)$ is denoted by $K(W^{\mathrm{tame}})$. %
\end{mydef}

\begin{rem}\label{AssumptionCyclotomic}
Using the fact that $K$ contains all $m$-th roots of unity for $m$ coprime to the residue characteristic, we easily see that any tame Kummer extension is Galois, with abelian Galois group. %
We will say more about this Galois action later on.  
\end{rem}

\begin{exa}
By definition, $K(W)$ contains all $n$-th roots of the uniformizer $\pi$. %
In particular, we see that $K(W)$ contains all tame extensions of $K$ since $k$ is algebraically closed. %

\end{exa}

We now start by considering Kummer extensions for regular annuli with $n=1$.

\begin{lemma}\label{LemmaKummer}
Let $A:=R[[u,v]]/(uv-\pi)$ and consider a finite separable extension $K(A)\subset{L}$ with normalization $A\subset{B}$. Suppose that $B$ is \'{e}tale %
over all generic primes of $A$ and that it is at most tamely ramified above the special fiber. Then $K(B)$ is a subextension of $K(W^{\mathrm{tame}})$. %
\end{lemma}
\begin{proof}
Consider the primes $\mathfrak{p}_{u,i}$ and $\mathfrak{p}_{v,i}$ of $B$ lying above $(u)$ and $(v)$ respectively. We write $e_{u_{i}}$ and $e_{v_{i}}$ for the ramification indices arising from the corresponding extensions of discrete valuation rings. 
Let $m=\mathrm{lcm}(e_{u_{i}},e_{v_{i}})$ and consider the normalization $W_{1}$ of $A$ inside the field extension %
\begin{equation}
K(A)\subset{K(A)(u^{1/m},v^{1/m})}=K(W_{1}).%
\end{equation}
We write $w_{1}=u^{1/m}$ and $w_{2}=v^{1/m}$. Note that $W_{1}$ is isomorphic to $R'[[w_{1},w_{2}]]/(w_{1}w_{2}-\pi^{1/m})$.  Here $R'$ is the unique tamely ramified extension of degree $m$ of $R$. %
In particular we now see that $W_{1}$ is regular.  
Let $L=K(W_{1})\cdot{K(B)}$ be the composite field of $K(W_{1})$ and $K(B)$ inside some algebraic closure and consider the normalization of $A$ inside $L$. We denote this ring by $N(A,L)$. %
Let $\mathfrak{q}\in{N(A,L)}$ be any prime lying above $(w_{1})$ or $(w_{2})$. %
The corresponding localization is then a discrete valuation ring. %
By Abhyankar's Lemma, %
we find that $N(A,L)\supset{W_{1}}$ is \'{e}tale at $\mathfrak{q}$ and thus $N(A,L)\supset{W_{1}}$ is \'{e}tale over all codimension one primes of $W_{1}$. Since $W_{1}$ is regular, we can now use the purity theorem to conclude that $N(A,L)$ is \'{e}tale over $W_{1}$. But $W_{1}$ is a complete Noetherian local ring with algebraically closed residue field $k$, so it is strictly Henselian. This means that its \'{e}tale fundamental group is trivial and thus $N(A,L)=W_{1}$. We conclude that %
$B\subset{N(A,L)}=W_{1}$ and thus $K(B)\subset{K(W_{1})}\subset{K(W^{\mathrm{tame}})}$, as desired. %

\end{proof}

\begin{rem}
By Lemma \ref{LemmaKummer} we see that the composite of all fields that are at most tamely ramified above the special fiber and \'{e}tale outside the special fiber is equal to $K(W^{\mathrm{tame}})$. This latter extension is Galois, as we now point out.  For any pair $(k_{1},k_{2})$ of integers coprime to the residue characteristic, consider the extension
\begin{equation}
K(A)\subset{K(W_{k_{1},k_{2}}):=K(A)(u^{1/k_{1}},v^{1/k_{2}})}.
\end{equation}
This is Galois with Galois group $\mathbf{Z}/k_{1}\mathbf{Z}\times{\mathbf{Z}/k_{2}\mathbf{Z}}$, where $(i,j)$ acts on $u^{1/k_{1}}$ and $v^{1/k_{2}}$ by
\begin{align*}
u^{1/k_{1}}&\mapsto{\zeta^{i}_{k_{1}}\cdot{}u^{1/k_{1}}}\\
v^{1/k_{2}}&\mapsto{\zeta^{j}_{k_{2}}\cdot{}v^{1/k_{2}}}.
\end{align*}
Here the $\zeta_{k_{i}}$ a primitive $k_{i}$-th root of unity. %
These extensions $K(W_{k_{1},k_{2}})$ form a natural direct system with limit $K(W^{\mathrm{tame}})$. This then also yields an inverse system of Galois groups and the inverse limit is given by $(\hat{\mathbf{Z}}')^{2}$, where the prime indicates that the inverse limit is taken over all integers coprime to the characteristic of $k$. We invite the reader to compare this to \cite[Page 316, Corollaire 5.3]{SGA1}. 

\end{rem}

Before we prove our generalized version of the Newton-Puiseux theorem, we first review how to go from a non-regular ring to a regular ring. 
Let $A=R[[u,v]]/(uv-\pi^{n})$. Consider the field extension $K(A)\subset{K(A)}(u^{1/n})=:K(A_{\mathrm{reg}})$ of degree $n$. The normalization of $A$ inside this field extension is isomorphic to $R[[u_{1},v_{1}]]/(u_{1}v_{1}-\pi)$.
\begin{mydef}{\bf{[Regularization of an ordinary double point]}}\label{Regularization1}
The normalization $A_{\mathrm{reg}}$ of $A$ inside the field extension $K(A)\subset{K(A_{\mathrm{reg}})}$ is the {\it{regularization}} of $A$.
\end{mydef} %
This extension $K(A)\subset{K(A_{\mathrm{reg}})}$ is non-separable if $n$ is divisible by the characteristic of $K$. Even if $n$ is not divisible by the characteristic of $K$, it need not be Galois, see  %
Remark \ref{AssumptionCyclotomic}.

 \begin{lemma}\label{LemmaKummer33}
 Let $A=R[[u,v]]/(uv-\pi^{n})\subset{B}$ be a finite extension of normal Noetherian domains that is \'{e}tale in codimension one. Then $B\simeq{R[[u,v]]/(uv-\pi^{k})}$ with $k|n$. Furthermore, $n/k$ is coprime to the residue characteristic. 
 \end{lemma}
 \begin{proof}
 Consider the Kummer extension $A':=R[[u_{1},v_{1}]]/(u_{1}v_{1}-\pi)$ of $A$, where $u\mapsto{u_{1}^{n}}$ as before. The induced extension $B'\supset{A'}$ by taking the normalization inside the composite $K(B)\cdot{K(A')}$ is then \'{e}tale over all codimension one primes by \cite[Corollaire 18.10.14, Page 162]{EGA4} and thus trivial since $A'$ is regular and strictly Henselian. %
 We thus see that we have an inclusion $B\subset{A'}$. %
 Since the \'{e}tale locus of a morphism is open, we find that $B\supset{A}$ is \'{e}tale at the generic point, which implies that $K(B)\supset{K(A)}$ is separable. We write $a=v_{p}(n)$ and $b=p^{a}$ if $p=\mathrm{char}(K)\neq{0}$; if $\mathrm{char}(K)=0$, then we set $a=0$ and $b=1$. The separable closure of $K(A)$ in $K(A')$ is now easily seen to be $K(A)(u_{1}^{b},v_{1}^{b})=K(A)(u_{1}^{b})$, so $K(B)\subset{K(A)(u_{1}^{b})}$. By  %
 Lemma \ref{GaloisRegularKummer} %
we then find %
that $K(B)$ is of the desired form. %
 \end{proof}

\begin{reptheorem}{NPTheorem}{\bf{[Generalized Newton-Puiseux theorem]}}
Let $A:=R[[u,v]]/(uv-\pi^{n})$. Consider the composite $L$ of all finite separable field extensions $K(B)$ such that the normalization $A\subset{B}$ is \'{e}tale over all the generic primes of $A$ and at most tamely ramified over the special fiber. 
Then $L=K(W^{\mathrm{tame}})$.  %
\end{reptheorem}
\begin{proof}

Suppose that $B$ is ramified over $\mathfrak{p}_{1}=(u,\pi)$ and $\mathfrak{p}_{2}=(v,\pi)$ of orders $e_{1}$ and $e_{2}$ respectively and let $m=\mathrm{lcm}(e_{1},e_{2})$. Consider the tame Kummer extension given by $K(A)\subset{K(A)}(\pi^{1/m})=:K(A')$. Note that the normalization of $A$ in $K(A')$ is isomorphic to $R'[[u,v]]/(uv-(\pi^{1/m})^{mn})$. By Abhyankar's Lemma we find that $K(A')\subset{K(A')\cdot{K(B)}}=:K(B')$ is \'{e}tale in codimension one. By Lemma \ref{LemmaKummer33} $K(B')$ is contained in a tame Kummer extension of $K(A')$ and thus $K(B')$ is contained in a tame Kummer extension of $K(A)$, as desired.

\end{proof}

\begin{mydef}
{\bf{[Disjointly branched polynomials]}}\label{DisjointlyBranchedPolynomial}
Let $f\in{K(A)[y]}$ be a polynomial, where $A=R[[u,v]]/(uv-\pi^{n})$, and suppose that the splitting field of $f$ is in $K(W^{\mathrm{tame}})$. We then say that %
$f$ is a disjointly branched polynomial.  
\end{mydef}

\begin{rem}
If $\mathrm{char}(k)=0$, then we only have to check the ramification on the generic fiber to see whether $f$ is disjointly branched. Since the generic primes of $A$ correspond to points $x\in{\overline{K}}$ such that $0<v(x)<n$ by Lemma \ref{OrdinaryDoublePoints}, we only have to check whether these extensions are ramified over points of this form. As a first step towards this problem, one can calculate the Newton-polygon of the discriminant of $f$. If this has a zero $\alpha$ with the property $0<v(\alpha)<n$, then one has to check whether it is actually a branch point. This can be done using the discrete Newton-Puiseux method, which will be given in Section \ref{NewtonDiscreteSection}.     %
\end{rem}

\subsection{A proof of the tame simultaneous semistable reduction theorem}\label{ProofSemistable}

In this section, we prove our main theorem, the tame simultaneous semistable reduction theorem. We first define the notion of a tame covering of permanent models. %
   After this, %
   we prove the main theorem using the results from Section \ref{KummerExtensions}. %

\begin{mydef}{\bf{[Tame morphisms]}}\label{TameCoveringsPermanent}
Let $\phi: \mathcal{X}\rightarrow{\mathcal{Y}}$ be a finite $R$-morphism of normal models of curves. We say that $\phi$ is tame if for every pair of points of codimension one $\eta',\eta$ such that $\phi(\eta')=\eta$, we have that $\phi$ is tame in the sense of Definition \ref{Tame}.
\end{mydef}

\begin{lemma}\label{AbhyankarPermanent}
Let $\phi:\mathcal{X}\to{\mathcal{Y}}$ be a tame morphism of normal models and suppose that $\mathcal{Y}$ is permanent. Then there exists a finite Kummer extension $R\subset{R'}$ such that the normalized base change of $\mathcal{X}$ is permanent. %
\end{lemma}
\begin{proof}
Since $\mathcal{Y}$ is permanent, we find that $\pi$ is a uniformizer for every $\mathcal{O}_{\mathcal{X},\eta}$ corresponding to the generic point $\eta$ of an irreducible component of the special fiber. We now take $r$ to be the least common multiple of the ramification indices $e(\eta_{i}/\eta)$, where $\eta$ ranges over the generic points of the irreducible components of the special fiber and $\eta_{i}$ ranges over the points in $\mathcal{X}$ lying over $\eta$. We claim that $R'=R[\pi^{1/r}]$ has the desired properties. First, every $\eta$ as above gives a unique $\eta'$ in the normalized base change $\mathcal{N}(\mathcal{Y},K'(Y))$ and this scheme has reduced special fiber since $\pi^{1/r}$ is a uniformizer at the corresponding generic points. %
Abhyankar's Lemma then implies that $\mathcal{N}(\mathcal{X},K'(X))$ is \'{e}tale over these $\eta'$, so $\pi^{1/r}$ is again a uniformizer. This implies that it has reduced special fiber, as desired.    %
\end{proof}

We thus see that we can assume in the proof of Theorem \ref{MainThm1} that $\mathcal{X}$ is permanent. To show that it is (strongly) semistable, we will need a small lemma. %

\begin{lemma}\label{Separability}
Let $\phi:\mathcal{X}\rightarrow{\mathcal{Y}}$ be a tame covering of permanent models with $\mathcal{Y}$ strongly semistable. Let $\eta_{e}$ be an ordinary double point and let $\eta\subset{\eta_{e}}$ be a generization of codimension one. Suppose that $\phi$ is \'{e}tale over $\eta$. Then for any $\eta'_{e}\in\mathcal{X}$ mapping to $\eta_{e}$, the induced morphism of  completions $\hat{\mathcal{O}}_{\mathcal{Y},\eta_{e}}\to\hat{\mathcal{O}}_{\mathcal{X},\eta'_{e}}$ is \'{e}tale at any prime ideal mapping to $\eta$ under $\mathrm{Spec}(\hat{\mathcal{O}}_{\mathcal{X},\eta'_{e}})\to\mathcal{X}\to\mathcal{Y}$.  %
\end{lemma}
\begin{proof}
We write $U=\mathrm{Spec}(A)$ for an open affine around $\eta_{e}$ and $\mathrm{Spec}(B)=\phi^{-1}(U)$. The points $\eta_{e}$ and $\eta$ are given by the prime ideals $\mathfrak{m}$ and $\mathfrak{p}$ respectively. We write $\hat{A}$ and $\hat{B}$ for the completions of $A$ and $B$ with respect to $\mathfrak{m}$. Note that these are isomorphic to the completions of the localizations $A_{\mathfrak{m}}$ and $B_{\mathfrak{m}}$ with respect to $\mathfrak{m}$. We now have the Cartesian diagram 
\begin{equation}\label{EquationDiagramCompletion}
\begin{tikzcd}
A \arrow{r} \arrow{d} & B \arrow{d} \\
\hat{A} \arrow{r} & \hat{B}
\end{tikzcd}.
\end{equation}
Since $\phi$ is \'{e}tale over $\eta$, it is not too hard to see that we can find an $h$ in $A\backslash\mathfrak{p}$ such that the localization $A_{h}\to{B_{h}}$ is \'{e}tale. We tensor the diagram in Equation \ref{EquationDiagramCompletion} with $A_{h}$ over $A$ and obtain
\begin{equation}
\begin{tikzcd}
A_{h} \arrow{r} \arrow{d} & B_{h} \arrow{d} \\
\hat{A}_{h} \arrow{r} & \hat{B}_{h}
\end{tikzcd}.
\end{equation}
This diagram is again Cartesian since the diagram in Equation \ref{EquationDiagramCompletion} is Cartesian and Cartesian diagrams are stable under base change. Since \'{e}tale maps are stable under base change, we conclude that %
$\hat{A}_{h}\rightarrow{\hat{B}_{h}}$ is \'{e}tale. This easily implies the statement of the lemma.  %

\end{proof}

\begin{reptheorem}{MainThm1}{\bf{[Main Theorem]}}
Let $\phi:(X,D_{X})\rightarrow{(Y,D_{Y})}$ be a finite separable morphism of smooth proper geometrically connected marked curves %
and let $\mathcal{Y}/R$ be a semistable model for $(Y,D)$. %
Suppose that the induced morphism of normalizations $\mathcal{X}\to\mathcal{Y}$ is tame in codimension one. Then there is a finite Kummer extension $R\subset{R'}$ such that the normalized base change of $\mathcal{X}$ %
is semistable. %
If $\mathcal{Y}$ is strongly semistable, then the normalized base change of $\mathcal{X}$ %
is also strongly semistable.

\end{reptheorem}

\begin{proof}
We take the extension as in Lemma \ref{AbhyankarPermanent} and write $\mathcal{X}$ and $\mathcal{Y}$ for their normalized base changes. We now show that any point in the inverse image of a smooth or ordinary double point is again a smooth or an ordinary double point. Consider a smooth point $\eta\in\mathcal{Y}$ and let $\eta'$ be any point in $\mathcal{X}$ lying above $\eta$. Consider the map of completions 
\begin{equation}\label{ExtensionCompletionsStable}
\hat{\mathcal{O}}_{\mathcal{Y},\eta}\rightarrow{\hat{\mathcal{O}}_{\mathcal{X},\eta'}}.%
\end{equation}
We now show that $\hat{\mathcal{O}}_{\mathcal{Y},\eta}\rightarrow{\hat{\mathcal{O}}_{\mathcal{X},\eta'
}}$ is \'{e}tale outside possibly a single codimension one point. %
Let $\hat{\mathfrak{p}}\in\mathrm{Spec}(\hat{\mathcal{O}}_{\mathcal{Y},\eta})$ be a nonzero point of the generic fiber, in the sense that it is not equal to $(0)$ and maps to $(0)$ under the map $\mathrm{Spec}(\hat{\mathcal{O}}_{\mathcal{Y},\eta})\to\mathrm{Spec}(R)$. Then the residue field of $\mathfrak{\hat{p}}$ is a finite extension $K'$ of $K$ by Lemma \ref{SmoothPoints}. Using \cite[Chapter 10, Proposition 1.40]{liu2} we find that this corresponds to a $K'$-rational point of the %
the generic fiber of $\mathcal{Y}$ that reduces to $\eta$ under the reduction map in Equation \ref{ReductionMap}. %
 By our assumption on the branch locus and Lemma \ref{Separability}, we then find that there is at most one generic prime over which the map of completions is non-\'{e}tale. Our tameness assumptions imply that %
 the map of completions is \'{e}tale over the special fiber and %
 using Lemma \ref{LemmaSmooth} we find that ${\hat{\mathcal{O}}_{\mathcal{X},\eta'}}$ is smooth over $R$, as desired.

Suppose now that $\eta$ is an ordinary double point and let $\eta'\in\mathcal{X}$ be any point mapping to $\eta$. We again consider the corresponding map of completions and retain the notation from the smooth case. As before, we obtain from Lemma \ref{Separability} that the map $\hat{\mathcal{O}}_{\mathcal{Y},\eta}\rightarrow{\hat{\mathcal{O}}_{\mathcal{X},\eta'}}$ is \'{e}tale in codimension one. %
Indeed, there is no ramification on the special fiber by the tameness conditions and there is no generic ramification by our assumption on the branch locus. %
We now apply Lemma \ref{LemmaKummer33} to conclude that $\eta'$ is again an ordinary double point. Since every closed point $\eta'\in\mathcal{X}$ is either an ordinary double point or a smooth point and the special fiber is reduced by assumption, we conclude that $\mathcal{X}$ is semistable. %

We now show that $\mathcal{X}$ is strongly semistable if $\mathcal{Y}$ is strongly semistable. Let $\eta'\in\mathcal{X}$ be an ordinary double point. We saw earlier that $\phi(\eta')$ is an ordinary double point. Suppose that there is only one codimension one generization $\eta'_{g}$ of ${\eta'}$ in the special fiber. Then $\phi(\eta'_{g})$ is also the only codimension one generization of $\phi(\eta')$, a contradiction. We conclude that $\mathcal{X}$ is strongly semistable.

\end{proof}

\begin{cor}
(\cite[Theorem 2.3]{liu1})
Let $\phi: X\rightarrow{Y}$ be a Galois cover of curves over $K$ and suppose that $\mathrm{char}(k)$ does not divide the order of the Galois group. Write $D$ for the branch locus of $\phi$ and $\mathcal{Y}$ for a semistable model for $(Y,D)$. Suppose that the morphism  %
of normalizations $\mathcal{X}\to\mathcal{Y}$ has no vertical ramification, so that the special fiber of $\mathcal{X}$ is reduced. Then $\mathcal{X}$ is semistable.  %
\end{cor}
\begin{proof}
The ramification indices and degrees of the field extensions on the codimension one primes have to divide the order of the Galois group, so we see that these are all coprime to the characteristic of $k$. But then $\mathcal{X}\rightarrow{\mathcal{Y}}$ is tame and we obtain the desired statement using Theorem \ref{MainThm1}. %
\end{proof}

\begin{mydef}{\bf{[Disjointly branched polynomials]}}
Let $\mathcal{Y}$ be a strongly semistable model for a curve $Y/K$ with function field $K(Y)$ and let %
$f\in{K(Y)}[y]$ be a geometrically irreducible\footnote{By geometrically irreducible, we mean that the polynomial stays irreducible over $\overline{K}(Y)$. 
} separable polynomial. We then have a finite extension $K(Y)\subset{K(Y)[y]/(f)}$ and we write $X\to{Y}$ for the induced morphism of smooth curves and $D$ for its branch locus.   %
We say that $f$ is disjointly branched with respect to $\mathcal{Y}$ if the following two hold: 
\begin{enumerate}
\item $\mathcal{Y}$ is a strongly semistable model for the marked curve $(Y,D)$.
\item The morphism of normalizations $\mathcal{X}\to\mathcal{Y}$ induced from the field extension $K(Y)\subset{K(X)=K(Y)[y]/(f)}$ is tame. %
\end{enumerate}

\end{mydef}

\begin{rem}\label{TropicalizationDisjointlyBranched}
Let $f$ be a disjointly branched polynomial. %
Theorem \ref{MainThm1} then implies that there is a finite Kummer extension $R\subset{R'}$ such that the normalized base change of $\mathcal{X}$ with respect to $R'$ is strongly semistable.  %
We will assume that $R=R'$, so that $\mathcal{X}$ is strongly semistable by Theorem \ref{MainThm1}. %
In the proof of Theorem \ref{MainThm1} we showed  %
that the inverse image of a smooth point consists of smooth points. By Remark \ref{TropicalizationMorphism} we then find that the morphism of strongly semistable models tropicalizes to a morphism of skeleta  %
$\Sigma(\mathcal{X})\to\Sigma(\mathcal{Y})$. %
\end{rem}

  \begin{rem}%
  \label{TamenessConditions22}
Suppose we are given a separable, geometrically irreducible polynomial $f\in{K(x)[y]}$. %
This gives an extension of function fields
\begin{equation}
K(x)\subset{K(x)[y]/(f)},
\end{equation}
which in turn corresponds to a finite morphism of curves $X\to\mathbf{P}^{1}_{K}$. %
To find out whether this morphism gives rise to a tame morphism, %
we can use the following two sufficient conditions: %
\begin{enumerate}
\item $\mathrm{char}(k)=0$ or 
\item $\mathrm{deg}_{y}(f)<p$, where $p=\mathrm{char}(k)$. 
\end{enumerate} 
To see the second condition, we write $n=\mathrm{deg}_{y}(f)$. Then $n!$ is not divisible by $p$ so any extension of discrete valuation rings for this covering is tame. Indeed, the ramification indices have to divide the order of the Galois group. This order divides $n!$, so these indices are not divisible by $p$. Furthermore, the residue field degrees are not divisible by $p$ by the formula in \cite[Chapter 1, Section 7, Corollary 1]{Ser1}, so we see that the covering we obtain for these polynomials is tame. 
\end{rem}

\begin{rem}\label{RemarkGaloisActionMonodromy}
Let $\eta_{e}\supset{\eta_{v}}$ correspond to a vertex $v\in{V(\Sigma(\mathcal{Y}))}$ and an adjacent edge $e\in{E(\Sigma(\mathcal{Y}))}$. Let $f$ be a %
disjointly branched polynomial and assume as in Remark \ref{TropicalizationDisjointlyBranched} that $R'=R$ so that we have coverings of strongly semistable models. It is then not too hard to see that the Galois closure also induces a tame covering of strongly semistable models. %
We write $\overline{\eta}_{e}\supset{\overline{\eta}}_{v}$ for an extension of the chain $\eta_{e}\supset{\eta_{v}}$ to the Galois closure. %
The point $\overline{\eta}_{v}$ is then the unique point lying over $\eta_{v}$ with $\overline{\eta}_{v}\subset{\overline{\eta}_{e}}$. %
We conclude that the condition in Theorem \ref{Dedekind3} is satisfied for these points. %
\end{rem}

\section{Generalizations of the Newton-Puiseux algorithm} %
\label{Algorithm}

In this section we generalize the Newton-Puiseux algorithm to obtain an algorithm that can calculate skeleta of curves over discrete valuation rings. %
We first give an algorithm 	for calculating power series expansions of tame polynomials over discrete valuation rings in Section \ref{NewtonDiscreteSection}. This is a generalization of the Newton-Puiseux algorithm for calculating power series expansions of polynomials over $\mathbf{C}[[t]]$ (or more precisely, the one over $\mathbf{Q}$). %
In Section \ref{MixedNP1} we specialize to semistable models and %
give an algorithm that simultaneously calculates $\eta_{e}$-adic and the $\eta_{v}$-adic power series expansions of a disjointly branched polynomial $f$ for a pair $(v,e)$ of a vertex and an adjacent edge in strongly semistable model $\mathcal{Y}$.  It first uses the generalized discrete algorithm to calculate the $\eta_{v}$-adic power series expansions, after which it calculates the $\eta_{e}$-adic power series expansions using another application of the discrete algorithm to the coefficients in the $\eta_{v}$-adic power series. By applying this algorithm to all pairs $(v,e)$ in $\Sigma(\mathcal{Y})$, we are then able to reconstruct the tropicalization $\Sigma(\mathcal{X})\to\Sigma(\mathcal{Y})$.  %

  \subsection{Preliminaries for the algorithms} \label{PreliminariesAlgorithms} %

  Let $B$ be a Noetherian unique factorization domain with field of fractions $K(B)$ and let $f\in{B[y]}$ be a completely split polynomial. 
 We write $\gamma_{i}$ for the roots of $f$ in $K(B)$ and %
  $\mathfrak{q}$ for a prime of codimension one. The zero set of $f$ is denoted by $Z(f)$. Since $B$ is a UFD, we can find a generator $\pi_{B}$ for $\mathfrak{q}$. %
For every root $\gamma_{i}$ of $f$ in $K(B)$, we can find a minimal $k_{i}\in\mathbf{Z}_{\geq{0}}$ such that $\pi_{B}^{k_{i}}\gamma_{i}\in{B}$. The set of roots with $k_{i}=0$ is denoted by $Z_{0}(f)$. %
We now write $\alpha_{i}:=\pi_{B}^{k_{i}}\gamma_{i}$ and consider the polynomial 
\begin{equation}\label{RootsRepresentation54}
g:=\prod(\pi_{B}^{k_{i}}y-\alpha_{i}).
\end{equation} 
Since $f$ and $g$ have the same roots, we find that %
\begin{equation}
f=u\cdot{g}
\end{equation}
for some $u\in{K(B)}$. If we furthermore demand that $f$ be primitive, %
then $u$ is a unit in $B$ since %
$g$ is also primitive. We will assume without loss of generality that $u=1$. %
We then have the following elementary lemma: %
\begin{lemma}\label{ApproximationMainLemma}
Let $f\in{B}[y]$ be a primitive completely split polynomial as above. %
Let $S=\{c_{i}\}$ be a set of pairwise distinct elements in $B$ with distinct reductions in $B/\mathfrak{q}$ %
such that $\overline{Z_{0}(f)}=\overline{S}$.  
Then for every $\gamma_{i}\in{Z_{0}(f)}$, %
there is a unique $c_{i}$ such that 
\begin{equation}
\gamma_{i}-c_{i}%
\in\mathfrak{q}. 
\end{equation}
Vice versa, for every $c_{i}\in{S}$ there is a root $\gamma_{i}\in{Z_{0}(f)}$ %
such that the above holds.
\end{lemma} 
  
  \begin{rem}
  Note that this reduced polynomial $\overline{f}$ can be a constant, for instance if $f=\pi_{B}y-1$. We will not encounter these in our algorithms. %
  \end{rem}
  
  \begin{mydef}{\bf{[Zeroth-order approximations]}}\label{ZerothOrderApproximation}
  Let $f$ be as in Lemma \ref{ApproximationMainLemma}. We refer to the $c_{i}\in{B}$ as zeroth-order approximations of the roots $\gamma_{i}$ of $f$ with respect to $\mathfrak{q}$. Given a set $S$ as in Lemma \ref{ApproximationMainLemma} and a specific root $\gamma_{i}$, we refer to the unique $c_{i}\in{S}$ such that $\gamma_{i}-c_{i}\in\mathfrak{q}$ as a zeroth-order approximation of $\gamma_{i}$. 
  \end{mydef}

 \subsubsection{Higher-order approximations}\label{HigherOrderApproximations}
We would now like to obtain higher-order approximation of the roots of $f$.  %
To that end, we first translate the polynomial over a specific zeroth order-approximation $c_{i,0}:=c_{i}$ so that the roots $\alpha_{i}$ with $\alpha_{i}-c_{i,0}\in\mathfrak{q}$ now have positive valuation. We then scale over a certain power of $\pi_{B}$ to reduce the valuations of a subset of the roots back to zero. After this, we find a zeroth-order approximation of the resulting polynomial, which gives us higher-order approximations of the original roots of $f$ by reversing the above operations. We refer to the first operation as a {\it{translation}} and the second as a {\it{scaling}}. %

We start with the translation. Let $f$ be as in Lemma \ref{ApproximationMainLemma} and assume that its reduction is non-constant. Let $c_{i,0}:=c_{i}$ be a zeroth-order approximation for the roots of $f$ with respect to $\mathfrak{q}$ as in Definition \ref{ZerothOrderApproximation}. There is then a natural ring homomorphism $\phi: B[y]\to{B}[y]$ which sends $y$ to $y+c_{i,0}$. We apply this homomorphism to $f$ and obtain $f_{1}:=\phi(f)=f(y+c_{i,0})$. %
  The roots of this polynomial are given by $\gamma'_{i}:=\gamma_{i}-c_{i,0}$. %

  We now define the scaling operation on $f_{1}$. Consider the roots $\gamma_{i}$ %
  with $k_{i}=0$ such that $c_{i,0}-\gamma_{i}%
  \in\mathfrak{q}$ as in %
  Lemma \ref{ApproximationMainLemma}. For any such root, we find %
 that $v_{\mathfrak{q}}(\gamma'_{i})>0$, so they are divisible by $\pi_{B}^{k}$ for some $k$. To find these valuations in practice, we use the classical Newton polygon theorem. %
 We now choose a root $\gamma'_{i}$ with $v_{\mathfrak{q}}(\gamma'_{i})>0$ and set $k:=%
 v_{\mathfrak{q}}(\gamma'_{i})$. %
 We then %
  consider the polynomial
  \begin{equation}
  f'_{2}:=f_{1}(\pi_{B}^{k}y).
  \end{equation}
In terms of ring homomorphisms, we have $f'_{2}=\psi(f_{1})$, where $\psi:B[y]\to{B}[y]$ is the ring homomorphism sending %
 $y$ to $\pi_{B}^{k}y$. 
 The roots of $f'_{2}$ are then given by
  \begin{equation}
  \gamma''_{i}:=\dfrac{\gamma'_{i}}{\pi_{B}^{k}}.
  \end{equation}
  We now apply the %
   two transformations $\phi$ and $\psi$ to the individual factors of $f$. %
   Define $h_{i}:=\pi_{B}^{k_{i}}y-\alpha_{i}$, so that $f=\prod{h_{i}}$. %
   We then have $\phi(h_{i})=\pi_{B}^{k_{i}}(y+c_{i})-\alpha_{i}=\pi_{B}^{k_{i}}y+\pi_{B}^{k_{i}}c_{i,0}-\alpha_{i}$.
   Similarly, if we apply $\psi(\cdot{})$ to $\phi(h_{i})$ then we obtain
   \begin{equation}
h_{i,1}:= \psi(\phi(h_{i}))=\pi_{B}^{k_{i}+k}y+\pi_{B}^{k_{i}}c_{i,0}-\alpha_{i}.
   \end{equation} 
 We now divide the resulting polynomial $f'_{2}:=\prod{\psi(\phi(h_{i}))}$ by a suitable power of $\pi_{B}$ to make it primitive. To that end, set 
 \begin{equation}
 s_{i}:=
 \begin{cases}
 \mathrm{min}\{k,v_{\mathfrak{q}}(\gamma'_{i})\} &\mbox{if } k_{i}=0, \\
 0 & \mbox{if } k_{i}\neq{0}.
 \end{cases}
 \end{equation} 
  This is the valuation of the content of the polynomial $h_{i,1}$. The polynomial %
 \begin{equation}
 h_{i,2}:=\dfrac{1}{\pi_{B}^{s_{i}}}\cdot{h_{i,1}},
 \end{equation}
 then satisfies $\mathrm{cont}(h_{i,2})=1$.  %
The polynomial $f_{2}:=\prod{h_{i,2}}$ %
then also has content $1$ since the content is multiplicative. We can find $f_{2}$ in practice by dividing $f'_{2}$ by its content.  %

  To return to Lemma \ref{ApproximationMainLemma}, we set
  \begin{align*}
  k''_{i}:=&k_{i}+k-s_{i},\\
  \alpha''_{i}:=&\dfrac{1}{\pi_{B}^{s_{i}}}(\alpha_{i}-c_{i,0}\pi_{B}^{k_{i}}),
  \end{align*}
  so that 
  \begin{equation}
  f_{2}=\prod{(\pi_{B}^{k''_{i}}y-\alpha''_{i})}
  \end{equation}
  as in Equation \ref{RootsRepresentation54}. 
 Let us %
 identify the roots with $k''_{i}=0$. If $k_{i}>0$, then $s_{i}=0$ and thus $k''_{i}>0$. If $k_{i}=0$, then $k''_{i}=0$ if and only if $s_{i}=k$, so these are the roots $\gamma'_{i}$ that satisfy $v_{\mathfrak{q}}(\gamma'_{i})\geq{k}$. %
 As before, we let $\{c_{i,k}\}\subset{B}$ be a set of pairwise distinct elements %
 that are zeroth order approximations for the roots of $f_{2}$ as in Definition \ref{ZerothOrderApproximation}. 
For every $\gamma''_{i}$ with $k''_{i}=0$, we can then find a unique $c_{i,k}$ such that 
   \begin{equation}
\gamma''_{i}-c_{i,k}\in\mathfrak{q}.
   \end{equation}
   Multiplying both sides by $\pi_{B}^{k}$, we then obtain 
   \begin{equation}
   \pi_{B}^{k}\gamma''_{i}-c_{i,k}\pi_{B}^{k}=\gamma_{i}-c_{i,0}-c_{i,k}\pi_{B}^{k}\in\mathfrak{q}^{k+1}.
   \end{equation}
   We rearrange this to obtain %
   \begin{equation}
   \gamma_{i}=c_{i,0}+c_{i,k}\pi_{B}^{k}+\mathcal{O}(\mathfrak{q}^{k+1}).
   \end{equation}
   In other words, $c_{i,0}+c_{i,k}\pi_{B}^{k}$ is a power series approximation for %
   $\gamma_{i}$ of height $k+1$. If we continue this process with $f_{2}$, then we of course obtain higher order approximations of the $\gamma_{i}$. This is the main idea that can be found in both of the upcoming %
   Newton-Puiseux algorithms. %

 \subsection{The Newton-Puiseux algorithm for discrete valuation rings}\label{NewtonDiscreteSection}
 
  In this section we give a Newton-Puiseux algorithm for tame extensions of discrete valuation rings. The algorithm is a direct generalization of the one presented in \cite{Duval} and \cite{Puiseux1}, which calculates the formal algebraic Puiseux expansions of a polynomial $f(x,y)$ over a field of characteristic zero.
  In our generalization, we consider a polynomial $f\in{R[y]}$, where $R$ is now a discrete valuation ring with uniformizer $\pi$. The algorithm calculates the $\pi$-adic expansions of the roots of $f$ up to any desired precision $n$. Throughout this section, we will assume that we have an auxiliary algorithm at our disposal that can factorize polynomials over the residue field $k$ of $K$.  %
  In this paper, the residue field will be of transcendence degree $0$ or $1$ over $\mathbf{F}_{p}$ or $\mathbf{Q}$. %
  For factoring polynomials over these fields we use the computer algebra package MAGMA. We will view our calculations as calculations over the Henselization $R^{{\mathrm{h}}}$ of the ring of integers $R$ of $K$. The results in Section \ref{GaloisSection} then allow us to %
 deduce information about the extensions of $\mathfrak{m}=(\pi)$. %
  This will be used to calculate the vertices of a skeleton of a curve in Section \ref{AlgorithmBerkovich}.

  Let $f\in{R}[x]$ be a polynomial over a discrete valuation ring $R$ with field of fractions $K$, maximal ideal $\mathfrak{m}$, uniformizer $\pi$ and residue field $k$. %
  We will assume for simplicity that the leading coefficient of $f$ is invertible. We denote the splitting field of $f$ by $\overline{L}_{f}$ and the integral closure of $R$ in $\overline{L}_{f}$ by $\overline{R}_{f}$. We assume throughout this section that $\overline{L}_{f}/K$ is separable. %
  \begin{mydef}
Let $f$ be a polynomial with a separable splitting field and let $\overline{\mathfrak{m}}_{f}$ be a prime in $\overline{R}_{f}$ lying over $\mathfrak{m}$. We say that $f$ is tame if the extension of discrete valuation rings corresponding to $\overline{\mathfrak{m}}_{f}\supset{\mathfrak{m}}$ is tame.
\end{mydef}  

Consider the field extension $M_{n}:=\overline{L}_{f}\cdot{K(\pi^{1/n})}\supset{K(\pi^{1/n})}$. For $n$ divisible by $d!$ for $d=\mathrm{deg}(f)$, we then find using Abhyankar's Lemma that $M_{n}$ is unramified over $K(\pi^{1/n})$. In the algorithm we will construct field extensions $K\subset{K(\pi^{1/n})}$ using the data we get from Newton polygons instead of $d!$. We will comment on this during the proof of the correctness of the algorithm. %
 At any rate, we now fix a $n$ with the above property and write $K'=K(\pi^{1/n})$ for the corresponding field. %
 We fix a prime $\mathfrak{m}^{\mathrm{s}}$ lying over $\mathfrak{m}'$ in the integral closure of $R$ in the separable closure $K^{{\mathrm{s}}}$, so that we obtain injections $K'^{{\mathrm{h}}}\to{K'^{{\mathrm{sh}}}}\to{K^{{\mathrm{s}}}}$. %
 The polynomial $f$ then splits completely over $K'^{{\mathrm{sh}}}$. We can now use the procedure given in Section \ref{PreliminariesAlgorithms} with  %
 $K(B)=K'^{{\mathrm{sh}}}$ to find power series expansions for the roots. The algorithm will give \'{e}tale algebras over $K'^{{\mathrm{h}}}$ and these can be embedded into $K'^{\mathrm{sh}}$ in various ways as in ordinary Galois theory. For computations however, we will only need the \'{e}tale algebras over $K'^{{\mathrm{h}}}$ so we don't need to know the full Galois group, see Remark \ref{ExtensionsMinimalPolynomialsAlgorithm}. 

\begin{algorithm}
  \caption{The discrete Newton-Puiseux algorithm.}\label{DiscreteNPAlgorithm}
    \vspace*{0.1cm}
\textbf{Input:} A tame polynomial $f\in{R[y]}$ over a discrete valuation ring, a positive integer $h$.  \\
\textbf{Output:} {The power series expansions of the roots of $f$ up to height $rh$ in $R'^{{\mathrm{sh}}}$, where $R'=R[\pi^{1/r}]$ is a finite Kummer extension of $R$. }%
  \begin{algorithmic}[1]

  \State Calculate the Newton polygon $\mathcal{N}(f)$ of $f$ with respect to the discrete valuation on $K$. 
  \vspace{0.1cm}
\State If there is no line segment of strictly positive slope, go to Step $5$ with $f_{e}:=f$. Otherwise, choose a line segment $e$ in $\mathcal{N}(f)$ with slope $n=\dfrac{a}{b}>{0}$ and $\mathrm{gcd}(a,b)=1$. If the slope is not in the value group, then extend %
$K$ to $K:=K(\pi^{1/b})$. %
Adjust the power series with respect to $\pi$ to power series with respect to $\pi^{1/b}$. Set $h:=b\cdot{h}$. %

\vspace{0.1cm}   %
\State %
Update the power series expansions of the roots corresponding to $e$ by adding zeros. %
\State If the precision of an approximation %
is greater than or equal to $h$, %
stop the algorithm. %
Otherwise, %
find the unique $m$ such that $f_{e}=1/\pi^{m}\cdot{}f(\pi^{n}y)$ is primitive.
\vspace{0.1cm}

  \State %
Calculate an irreducible factorization of the reduction $\overline{f}_{e}$: %
\begin{equation}
\overline{f}_{e}=\prod_{i}\overline{g}_{i}^{r_{i}}.
\end{equation}
\vspace{0.1cm}

\State Choose an irreducible factor $\overline{g}_{i}$ and %
construct an \'{e}tale extension $A_{i}$ %
of $R$ as in Remark \ref{ConstructionEtaleAlgebras} using a lift $g_{i}$ of $\overline{g}_{i}$.   %
\vspace{0.1cm} %
\State Update the power series expansions of the roots of $f$ using a root $c_{i}\in{A}_{i}$ of $g_{i}$. %
\vspace{0.1cm}
\State %
Calculate $f_{2}:=f(y+c_{i})$ and return to Step $1$ with %
$f:=f_{2}$ and $R:=A_{i}$. 
\vspace{0.1cm}

\end{algorithmic}
\end{algorithm}

\begin{rem}
If the slope of one of the segments in Step $1$ is infinite, then we stop the algorithm for this line segment since we have obtained an exact approximation of a root. %
\end{rem}
\begin{rem}\label{ChoiceEmbedding}
We explain Step $7$ in more detail. We write $K'$ for the discretely valued field $K(\pi^{1/d!})$, so that $f$ splits over $K'^{{\mathrm{sh}}}$. %
The \'{e}tale algebra created in Section \ref{HenselizationsSection} is isomorphic to $R'[y]/(g_{i})$, where $g_{i}$ is a lift of $\overline{g}_{i}$. %
Note however that this does not fix an embedding into $R'^{{\mathrm{sh}}}$, since we can send $y\bmod{(g_{i})}$ to any root of $g_{i}$ in $R'^{{\mathrm{sh}}}$. Consequently, we have different options for the first coefficients of the roots of $f$ reducing to roots of $\overline{g}_{i}$. We denote the different roots of $g_{i}$ in $R'^{{\mathrm{sh}}}$ by $c_{i,j}$. Then the power series of the roots of $f$ reducing to roots of $\overline{g}_{i}$ are the $c_{i,j}$. Abstractly, we can update our power series expansions using these embeddings. This is what we mean by updating the power series expansions in Step $8$. Note however that this does not add any additional computational complexity: for our calculations we only need the algebra $R'[y]/(g_{i})$ and not any embedding into $R'^{{\mathrm{sh}}}$. %
This problem is also discussed in Remark \ref{ExtensionsMinimalPolynomialsAlgorithm}.     %
\end{rem}

\begin{rem}\label{HeightStop}
Instead of working with a specific height $h$, we can also stop the algorithm whenever the reduction $\overline{f}$ has no double factors.  %
At this point, we can use Hensel's lemma and conclude that the roots lift to the Henselization or completion, so no more field extensions are needed after this point.   %
\end{rem}

We now prove that Algorithm \ref{DiscreteNPAlgorithm} is correct.
 \begin{theorem}\label{CorrectnessTheoremDiscrete}
 Let $f\in{R}[y]$ be a tame polynomial and consider the discrete Newton-Puiseux algorithm given in \ref{DiscreteNPAlgorithm}. This algorithm correctly computes the power series expansions of the roots of $f$ up to height $rh$. 
 \end{theorem}
 \begin{proof}

The splitting field of $f$ over $K^{{\mathrm{sh}}}$ is contained in %
$K^{{\mathrm{sh}}}(\pi^{1/d!})$ by construction. %
We now claim that the calculations in the algorithm are all contained in this field. The only non-trivial thing to check here is the part about the Newton-polygons. For any completely split polynomial over $K^{{\mathrm{sh}}}(\pi^{1/d!})$, we have that the slopes of the Newton-polygon are in the value group $\dfrac{1}{d!}\mathbf{Z}$. In Step $2$ on the first iteration, we thus find $\dfrac{a}{b}=\dfrac{a'}{d!}$ for some integer $a'$, meaning that $ba'=ad!$. Since $a$ and $b$ are coprime, we find that $b$ divides $d!$ and thus $K^{{\mathrm{sh}}}(\pi^{1/b})\subset{K^{{\mathrm{sh}}}(\pi^{1/d!})}$. Using induction on the iterations in the algorithm, we conclude that all Kummer extensions made in Step $2$ are in $K^{{\mathrm{sh}}}(\pi^{1/d!})$.  
In terms of Section \ref{PreliminariesAlgorithms}, we can now take $K(B)=K^{{\mathrm{sh}}}(\pi^{1/d!})$.  %
The validity of the intermediate steps in the algorithm was shown in Section \ref{PreliminariesAlgorithms}. We conclude that the algorithm correctly calculates the power series expansions of the roots of $f$. %
 \end{proof}

\begin{exa}
We give an example where the Newton-Puiseux algorithm does not terminate. %
Consider the polynomial %
$f=y^2-2y-2$ over the ring of $2$-adic numbers $\mathbf{Z}_{2}$. We first make the extension $\mathbf{Q}_{2}(2^{1/2})$ and consider the scaled version
 \begin{equation}
f_{1}:=1/2\cdot{} f(2^{1/2}y)=y^2-2^{1/2}y-1.
 \end{equation}
 The reduction of $f_{1}$ has the double root $1$ in $\mathbf{F}_{2}$, so we translate over the lift of this root to obtain the polynomial
 \begin{equation}
 f_{2}:=f_{1}(y+1)=y^2+(2-2^{1/2})y-2^{1/2}.
 \end{equation} 
 The Newton polygon has a line segment of slope $1/2$ and the algorithm tells us to again take a square root of $2^{1/2}$. In fact, at every step we obtain the same Newton polygon, as one %
 can easily see using induction. %
 Using this, we see that %
 this process of extracting roots goes on indefinitely. %
 Note that the extension of $\mathbf{Q}_{2}$ induced by $f$ is a wildly ramified extension of differential exponent $2$. The differential exponents of the extensions $\mathbf{Q}_{2}\subset\mathbf{Q}_{2}(2^{1/2^{k}})$ on the other hand are never divisible by $2$. %

 \end{exa}

  \begin{exa}\label{MainExampleDiscreteValuation}
 [{\bf{Main Example}}] 
 We apply Algorithm \ref{DiscreteNPAlgorithm} to Example \ref{MainExample}. The generic points of the special fiber of the semistable model $\mathcal{Y}$ of $\mathbf{P}^{1}_{K}$ considered in Example \ref{SeparatingMainExample} give rise to discrete valuation rings $\mathcal{O}_{\mathcal{Y},\eta_{v}}$. We calculate %
 the Newton-Puiseux expansions of the roots of $f$ over two of these discrete valuation rings. Here we view $f$ as a polynomial in $y$ over the function field $K(x)$, where $K=\overline{\mathbf{Q}}((t))$.  %

 \vspace{0.2cm}
 
{\underline{ {\bf{Case 1: $v(x)=0$}}}}

\vspace{0.2cm} 
 
We start with the natural valuation of $K(x)$ induced from the codimension-one prime ideal $\mathfrak{p}_{1}=(t)$ in $A=R[x]$ with residue field $\overline{\mathbf{Q}}(x)$. In terms of Section \ref{P1Models}, this valuation corresponds to the positively oriented closed disk $\mathbf{B}^{+}_{0}(0)$. We denote the Henselization of the valuation ring corresponding to $\mathfrak{p}_{1}$ by $A^{{\mathrm{h}}}$. %
 
 We first transform $f$ so that it becomes monic. That is, we set $g=t^2+1$ and calculate \begin{align}
 f_{1}:=g^3f(y/g).
 \end{align}
 Reducing $f_{1}$ modulo $(t)$ then gives 
 \begin{equation}
 \overline{f}_{1}=(y^2-x)^2.
 \end{equation}
  We therefore construct the \'{e}tale algebra $A_{1}=A^{{\mathrm{h}}}[w]/(w^2-x)$. We write $w:=w\bmod{(w^2-x)}$. %
and translate $f_{1}$ over $w$ to obtain $f_{2}:=f_{1}(y+w)$. We then find that the Newton polygon with respect to $t$ consists of a line segment of length $2$ and slope $0$, %
 and a line segment of length $2$ and slope $1$. We thus calculate $f_{2}:=1/t^2f_{1}(ty)$. This polynomial reduces modulo $(t)$ to 
 \begin{equation}
 \overline{f}_{2}=4w^2y^2+w^4.
 \end{equation}
This has the roots
 $\beta_{1}=iw/2$ and $\beta_{2}=-\beta_{1}$ (where $i$ is a square root of $-1$). We can now stop the algorithm as the roots of this polynomial are separable. Note that we can use Hensel's lemma to conclude that these roots lift to %
 $A_{1}$ (here we use that a finite local extension of a Henselian ring is Henselian). The two {{generic approximations}} in $A_{1}$ %
 of the roots of $f$ are given by %
 \begin{align*}
 r_{+}=w+\dfrac{iw}{2}t,\\
 r_{-}=w-\dfrac{iw}{2}t.
 \end{align*}
  To obtain the full set of approximations, we can consider the two embeddings of $A_{1}$ into the strict Henselization. %
We denote the images of $w$ under these two embeddings by $w_{i}$.  %
This then gives the following four approximations: %
\begin{align*}
r_{1}=&w_{1}+\dfrac{iw_{1}}{2}t,\\
r_{2}=&w_{1}-\dfrac{iw_{1}}{2}t,\\
r_{3}=&w_{2}+\dfrac{iw_{2}}{2}t,\\
r_{4}=&w_{2}-\dfrac{iw_{2}}{2}t. %
\end{align*}

Let $K^{{\mathrm{h}}}_{\mathfrak{p}_{1}}$ be the field of fractions of $A^{{\mathrm{h}}}$. Note that the Galois group of $K^{{\mathrm{h}}}_{\mathfrak{p}_{1}}(w)$ over $K^{{\mathrm{h}}}_{\mathfrak{p}_{1}}$ is $\mathbf{Z}/2\mathbf{Z}$. Using this action on the approximations, we find that $\{r_{1},r_{3}\}$ and $\{r_{2},r_{4}\}$ form an orbit. There are thus two primes lying above $\mathfrak{p}_{1}$ in the normalization of $A$ in $K(x)[y]/(f(y))$ by Proposition \ref{OrbitsFactors}. Furthermore, using Theorem \ref{Dedekind2a} we see that the Galois group of $f$ over $K(x)$ contains a product of two $2$-cycles. %

\vspace{0.2cm}

{\underline{{\bf{Case 2: $v(x)=4$}}}}

\vspace{0.2cm}

We now consider another discrete valuation, namely the one induced from the codimension one prime ideal $\mathfrak{p}_{2}=(t)$ in the algebra $A=R[x_{1},x]/(x-x_{1}t^4)$. In terms of Section \ref{P1Models}, this algebra corresponds to the positively oriented closed disk $\mathbf{B}^{+}_{4}(0)$. We denote the corresponding Henselization by $A^{{\mathrm{h}}}$. %
As before, we first make $f$ monic over this ring. %
We calculate the Newton polygon of $f$ with respect to prime ideal $(t)$ and find that it contains a single line segment of slope $2$. Consequently, we calculate $f_{1}:=1/t^8f(t^2y)$. The reduction of $f_{1}$ modulo $(t)$ is given by
\begin{equation}
(y^2-x_{1})^2.
\end{equation}
We thus again consider the extension $A^{{\mathrm{h}}}\subset{A^{{\mathrm{h}}}[w]/(w^2-x_{1})}=:A_{2,1}$. We write $w:=w\bmod{(w^2-x_{1})}$. 
Translating over $w$ gives us the polynomial $f_{2}:=f_{1}(y+w)$, whose Newton polygon has a single line segment of slope $1$. We calculate $f_{3}:=1/t^2f_{2}(ty)$, which reduces as 
\begin{equation}
\overline{f}_{3}=4w^2y^2+w^4+w^3.
\end{equation}
This polynomial is irreducible and thus squarefree over $\overline{\mathbf{Q}}(x_{1})(w)$, so we stop the algorithm, see Remark \ref{HeightStop}. Let us denote the corresponding algebra by $A_{2,2}=A_{2,1}[z]/(z^2+(w^2+w)/4)$. We again write $z$ for the image of $z$ in the quotient $A_{2,2}$. The field $K^{{\mathrm{h}}}_{\mathfrak{p}_{2}}(w,z)$ is of degree four over $K^{{\mathrm{h}}}_{\mathfrak{p}_{2}}$. %
Indeed, it is at most of degree four and the minimal polynomial of $z$ over $K^{{\mathrm{h}}}_{\mathfrak{p}_{2}}$ is the irreducible polynomial
\begin{equation}
h_{z}(y)=y^4-xy^2+x^2/4-x/4.
\end{equation}
Unlike the previous case, we now see that the field $K^{{\mathrm{h}}}_{\mathfrak{p}_{2}}(w,z)$ has four embeddings into $K_{\mathfrak{p}_{2}}^{\mathrm{sh}}$. We denote the choices for $w$ and $z$ by $w_{i}$ and $z_{i}$. We then have the approximations
\begin{align*}
r_{1}=&w_{1}+z_{1}t,\\
r_{2}=&w_{1}+z_{2}t,\\
r_{3}=&w_{2}+z_{3}t,\\
r_{4}=&w_{2}+z_{4}t %
\end{align*}
inside $K_{\mathfrak{p}_{2}}^{\mathrm{sh}}$. The Galois group of $K^{{\mathrm{h}}}_{\mathfrak{p}_{2}}\subset{K^{{\mathrm{h}}}_{\mathfrak{p}_{2}}(w,z)}$ acts transitively on the $z_{i}$ since they have the same minimal polynomial. %
Using Remark \ref{OrbitsFactors}, we then see that there is only one extension of $\mathfrak{p}_{2}$ to the normalization %
of $A$ inside $K(x)[y]/(f(y))$. Note that we didn't need the full Galois group of the extension induced by $z$ over $K^{{\mathrm{h}}}_{\mathfrak{p}_{2}}$. %
We will comment on this again in Remark \ref{ExtensionsMinimalPolynomialsAlgorithm}. We can also calculate the Galois group in this case, which gives $\mathrm{Gal}(K^{{\mathrm{h}}}_{\mathfrak{p}_{2}}(z)/K^{{\mathrm{h}}}_{\mathfrak{p}_{2}})=D_{4}$. Using a variant of Theorem \ref{Dedekind2a} for Henselizations, we conclude that $D_{4}$ injects into the Galois group of $f$ over $K(x)$. A similar calculation for the valuation corresponding to $\mathbf{B}^{+}_{8}(0)$ in fact shows that $S_{4}$ injects into the Galois group, so the Galois group is isomorphic to $S_{4}$.

\vspace{0.3cm}

\end{exa}

\begin{rem}\label{ExtensionsMinimalPolynomialsAlgorithm}{\bf{[From power series to extensions]}}
Let $f$ be a disjointly branched polynomial defined over $K(x)$, the function field of a strongly semistable model $\mathcal{Y}$ for $\mathbf{P}^{1}$.  %
In applying the discrete Newton-Puiseux algorithm at the %
generic points $\eta_{v}$ of the irreducible components in the special fiber of $\mathcal{Y}$, we find a finite extension $K'$ of $K$ such that $f$ splits completely over the strict Henselization of $K'$. We will work over this finite extension with $K=K'$, so that the normalization is \'{e}tale over $\eta_{v}$. %

   We will work in an affine patch of $\mathcal{Y}$ and write $\mathfrak{p}$ for the prime ideal corresponding to $\eta_{v}$. To calculate the extensions $\mathfrak{p}'\supset{\mathfrak{p}}$, we have to find the orbits of the roots of $f$ under the absolute decomposition group $D_{\mathfrak{p}}$. Write $A$ %
   for the Henselization of the coordinate ring %
   at %
   $\mathfrak{p}$  %
 and suppose that Algorithm \ref{DiscreteNPAlgorithm} finds a polynomial $h$ over an \'{e}tale $A$-algebra $A'$ such that $\overline{h}$ splits completely with only simple factors. We can then lift these roots to $A'$ since it is Henselian. %
Unraveling the construction that led us to $h$, we then also find roots $\alpha\in{A'}$ of the original polynomial $f$. 
For each of these, the algorithm now returns a finite approximation
\begin{equation}
\alpha=\sum_{i=0}^{r}c_{i}\pi^{i}+\mathcal{O}(\pi^{r+1}),
\end{equation}   
where $c_{i}\in{A'}$ are in the set of representatives of the residue field of $A'$. 
Consider the field extension $k\subset{k(\overline{c}_{0},...,\overline{c}_{r})}$. Under the equivalence in \cite[\href{https://stacks.math.columbia.edu/tag/04GK}{Lemma 04GK}]{stacks-project}, this corresponds to a finite \'{e}tale extension $B$ of $A$ with residue field $k(\overline{c}_{0},...,\overline{c}_{r})$ and we have $\alpha\in{B}$ by the earlier lifting argument. 
We claim that $K(B)$ is the smallest extension of $K(A)$ %
over which $\alpha$ is defined. Indeed, suppose that $\alpha$ is defined over a smaller extension $B'$ with field of fractions by $K(B')$.  This corresponds to a finite $k$-subextension $\ell$ of $k(\overline{c}_{0},...,\overline{c}_{r})$ by \cite[\href{https://stacks.math.columbia.edu/tag/04GK}{Lemma 04GK}]{stacks-project}. Let $j$ be the smallest integer such that $\overline{c}_{j}\notin\ell$. Note that $z:=1/\pi^{j}(\alpha-\sum_{i=0}^{j-1}c_{i}\pi^{i})=c_{j}+\mathcal{O}(\pi)\in{A'}\cap{K(B')}=B'$. But then $\overline{z}=\overline{c}_{j}$ and so $\overline{c}_{j}\in\ell$, a contradiction. %
We conclude that $K(A)(\alpha)=K(B)$. %
This also implies that the degree $n_{\alpha}$ of the minimal polynomial of $\alpha$ over $K$ is equal to the degree $n_{c}$ of %
$k(\overline{c}_{0},...,\overline{c}_{r})$ over $k$. %
We conclude that there are exactly $n_{\alpha}$ embeddings of $K(A)(\alpha)$ into $K(A)^{{\mathrm{sh}}}$. The images of $\alpha$ under these embeddings give a single $D_{\mathfrak{p}}$-orbit, which corresponds to an extension $\mathfrak{p}'\supset{\mathfrak{p}}$.

To decide whether two roots $\alpha$ and $\beta$ are in the same $D_{\mathfrak{p}}$-orbit, we can do the following. We start by taking a set of representatives $S$ of the residue field of $K(A)^{{\mathrm{sh}}}$ that is invariant under the action of $D_{\mathfrak{p}}/I_{\mathfrak{p}}=\mathrm{Gal}(k^{\mathrm{sep}}/k)$. This implies that the action of this group on a coefficient can be detected on the level of residue fields. The algorithm then gives the following power series expansions %
for $\alpha$ and $\beta$: %
\begin{align*}
\alpha=\sum_{i=0}^{r}c_{i}\pi^{i}+\mathcal{O}(\pi^{r+1}),\\
\beta=\sum_{i=0}^{r}d_{i}\pi^{i}+\mathcal{O}(\pi^{r+1}).
\end{align*}
Here we assume that the height is sufficiently high, in the sense that the roots of $f$ have distinct power series expansions. The two roots $\alpha$ and $\beta$ are then conjugate if and only if inductively for every $i\leq{r-1}$ we have that the minimal polynomial of $\overline{c}_{i+1}$ over $k(\overline{c}_{0},...,\overline{c}_{i})$ is the same as the minimal polynomial of $\overline{d}_{i+1}$ over $k(\overline{d}_{0},...,\overline{d}_{i})$ under the isomorphism $k(\overline{c}_{0},...,\overline{c}_{i})\simeq{k(\overline{d}_{0},...,\overline{d}_{i})}$ given by the previous step. Here the initial isomorphism is just the identity $k\to{k}$, so in the first step we should check whether the minimal polynomial of $\overline{c}_{0}$ over $k$ is the same as the minimal polynomial of $\overline{d}_{0}$ over $k$. If these are the same, then we have an isomorphism $k(\overline{c}_{0})\simeq{k(\overline{d}_{0})}$ sending $\overline{c}_{0}$ to $\overline{d}_{0}$ and we then check whether the minimal polynomial of $\overline{c}_{1}$ over $k(\overline{c}_{0})$ is the same as the minimal polynomial of $\overline{d}_{1}$ over $k(\overline{d}_{0})$ under the above isomorphism, and so on. %
The fact that this is sufficient to establish the conjugacy of $\alpha$ and $\beta$ follows as before from the equivalence in \cite[\href{https://stacks.math.columbia.edu/tag/04GK}{Lemma 04GK}]{stacks-project}. More precisely, the isomorphisms of residue fields give isomorphisms of algebras over the Henselization and the $c_{i}$ are sent to the $d_{i}$ under these since the representatives are in an invariant set. Using the given data we then directly deduce an isomorphism $A[\alpha]\to{A[\beta]}$ which can be extended to an automorphism of the Galois closure in the usual way.    %
All in all we now see that this criterion gives us a concrete way to check whether two roots are conjugate.

\end{rem}

\begin{rem}\label{GenusComponent}{\bf{[Calculating the genus]}}
Let $\mathfrak{p}'\supset{\mathfrak{p}}$ be an extension as in Remark \ref{ExtensionsMinimalPolynomialsAlgorithm}. %
The residue field of $\mathfrak{p}'$ is the function field of a unique smooth curve $\Gamma(\mathfrak{p}')$ and we are interested in its arithmetic genus. This is a birational invariant, so we can compute it using the residue field $k(\mathfrak{p}')$. %
Using the notation from Remark \ref{ExtensionsMinimalPolynomialsAlgorithm}, we find the equality
\begin{equation}
k(\mathfrak{p}')=k(\mathfrak{p})(\overline{c}_{0},...,\overline{c}_{r}).
\end{equation}
We can then for instance calculate the genus iteratively using $k(\mathfrak{p})(\overline{c}_{0},...,\overline{c}_{i})\subset{k(\mathfrak{p})(\overline{c}_{0},...,\overline{c}_{i+1})}$ and the Riemann-Hurwitz formula. Another way would be to write down a primitive element of this extension and then to calculate the genus of the corresponding plane curve.    
\end{rem}

\begin{rem}\label{RemarkGeometricallyIrreducible}
{\bf{[Geometrically irreducible curves]}} In the above examples, we had to calculate factorizations of polynomials over $\overline{\mathbf{Q}}(x)$. In practice, we start by calculating a factorization over $\mathbf{Q}(x)$ (or over $K(x)$ for a finite extension $K$ of $\mathbf{Q}$) and we then check whether the corresponding curve is {\it{geometrically irreducible}}. If it is geometrically reducible, then we extend our base field to obtain the correct factorization. %
\end{rem}

\subsection{The mixed Newton-Puiseux algorithm}\label{MixedNP1}

In this section, we describe a mixed Newton-Puiseux algorithm for tame semistable coverings. This algorithm starts with a disjointly branched polynomial $f$ defined over a strongly semistable model $\mathcal{Y}$ with intersection graph $\Sigma(\mathcal{Y})$. For any pair $(v,e)$ of a vertex and an edge in $\Sigma(\mathcal{Y})$, %
it then calculates %
two types of power series for the roots of $f$: $\mathfrak{p}_{v}$-adic power series $r_{\mathfrak{p}_{v}}$ and %
$\mathfrak{m}_{e}$-adic power series $r_{\mathfrak{m}_{e}}$. Here $\mathfrak{p}_{v}$ and $\mathfrak{m}_{e}$ are the prime ideals corresponding to $v$ and $e$ in some local chart of $\mathcal{Y}$ as given in Section \ref{P1Models}. %
These pairs of approximations $(r_{\mathfrak{p}_{v}}, r_{\mathfrak{m}_{e}})$ for the roots of $f$ furthermore satisfy a compatibility condition, which says that $r_{\mathfrak{p}_{v}}$ reduces to $r_{\mathfrak{m}_{e}}$ modulo $(u^{{h}_{\mathfrak{m}}},v^{{h}_{\mathfrak{p}}})$. %
This compatibility allows us to glue together the power series obtained for pairs $(v,e)$ of edges and vertices in $\Sigma(\mathcal{Y})$. %
To find the $r_{\mathfrak{p}_{v}}$ and $r_{\mathfrak{m}_{e}}$, we use the algorithm given in Section \ref{PreliminariesAlgorithms} on the strict Henselization of a local chart at an edge. %
By applying the discrete algorithm to the $\mathfrak{p}_{v}$-adic coefficients, we then obtain the $\mathfrak{m}_{e}$-adic coefficients. %

\subsubsection{Preliminaries}\label{PreliminariesMixedNP}

We start with an algebra of the form $A=R[u,v]/(uv-\pi^{n})$ corresponding to a closed annulus as in Section \ref{P1Models}. We write $\mathfrak{p}_{1}=(u,\pi)$ and $\mathfrak{p}_{2}=(v,\pi)$ for the primes corresponding to the irreducible components of the special fiber and $\mathfrak{m}=(u,v,\pi)$ for their intersection. Let $f$ be a disjointly branched monic polynomial over $A$. We will also assume that $f$ is already \'{e}tale over the special fiber. We would now like to calculate the extensions of the chains $\mathfrak{m}\supset{\mathfrak{p}_{i}}$ to the normalization of $A$ inside the field $K(A)[y]/(f)$. We will focus on $\mathfrak{p}:=\mathfrak{p}_{2}=(v,\pi)$ and $\mathfrak{m}$ for the remainder of these sections (the algorithm for the other pair is the same). %

We first modify our initial algebra to make it regular. As in Section \ref{KummerOrdinaryDoublePointsSection}, %
we have a set of compatible $n$-th roots of $u$ and $v$, which gives us the regularization $K(A_{\mathrm{reg}})=K(A)(u^{1/n},v^{1/n})$. We again write $K(A)$ for this field, so effectively we can now assume $n=1$ in our above algebras. The ring $A=R[u,v]/(uv-\pi)$ is then a unique factorization domain and we again write $A$ for the localization of this ring at $\mathfrak{m}$. %
This is a regular local ring with parameters $u$ and $v$. The strict of Henselization of this ring is again regular with the same parameters, see \cite[\href{https://stacks.math.columbia.edu/tag/06LN}{Lemma 06LN}]{stacks-project}. By Remark \ref{StrictHenselizationInclusion}, we now have a commutative diagram of fields 
\begin{equation}
\begin{tikzcd}
K(A)\arrow{r} \arrow{d} & K^{\mathrm{sh}}_{\mathfrak{m}}(A) \arrow{dl} \\
K^{\mathrm{sh}}_{\mathfrak{p}}(A) & %
\end{tikzcd}
\end{equation}
Since the extension induced by $f$ is assumed to be unramified over the special fiber, we find by that $f$ splits completely over $K^{\mathrm{sh}}_{\mathfrak{m}}(A)$ and thus also over $K^{\mathrm{sh}}_{\mathfrak{p}}(A)$. We now apply the algorithm in Section \ref{PreliminariesAlgorithms} on $f$ over $K^{\mathrm{sh}}_{\mathfrak{m}}(A)$ to obtain the $\mathfrak{p}$-adic expansions of the roots of $f$. Here we run into a problem: the \'{e}tale algebras obtained from irreducible polynomials over $k(\mathfrak{p})$ can be non-\'{e}tale over $\mathfrak{m}$. %
To remedy this, we introduce the  %
concept of \'{e}tale lifts.  %
If $\mathrm{char}(k)=0$ then we can find canonical lifts using a section of the residue map. If $\mathrm{char}(k)=p>0$, then we can use a modification of the discrete Newton-Puiseux algorithm together with Hensel's lemma to find the lifts. %

We now summarize the steps in the mixed Newton-Puiseux algorithm. %
We first run the $\mathfrak{p}$-adic algorithm, which gives us
\begin{equation}
\alpha=\sum_{i=0}^{r}c_{i}v^{i}+\mathcal{O}(v^{r+1})
\end{equation} 
where the $c_{i}$ are all contained in 
$K^{\mathrm{sh}}_{\mathfrak{m}}(A)$ by the construction of $\mathfrak{m}$-integral split lifts. In fact, they will be integral, so $c_{i}\in{A^{\mathrm{sh}}_{\mathfrak{m}}}$. We then write down $\mathfrak{m}$-adic power series expansions for the $c_{i}$. If $\mathrm{char}(k)=0$, this can be done by calculating the %
$k[[u]]$-expansions of the $\overline{c}_{i}$ and if $\mathrm{char}(k)=p$ then we can use Hensel's lemma. %
We then have two sets of approximations: the $\mathfrak{p}$-adic approximations and the $\mathfrak{m}$-adic approximations. %
 We use these to calculate the $D_{\mathfrak{m}}$ and $D_{\mathfrak{p}}$-orbit of a root. This in turn allows us to find extensions of the chain $\mathfrak{m}\supset{\mathfrak{p}}$.

\subsubsection{\'{E}tale lifts}

We first define the notion of an \'{e}tale lift. Let $K(A)$ be the quotient field of $A=R[u,v]/(uv-\pi)$ and let $\mathfrak{m}=(u,v,\pi)\supset{(v,\pi)}=\mathfrak{p}$.
 
\begin{mydef}\label{Lifts}
{\bf{[\'{E}tale lifts]}}
Let $\overline{g}\in{k(\mathfrak{p})}[y]=k(u)[y]$ be an irreducible polynomial. %
An \'{e}tale lift of $\overline{g}$ is a polynomial $g\in{A_{\mathfrak{p}}[y]}$ with the following properties:
\begin{enumerate}
\item The reduction of $g$ modulo $\mathfrak{p}$ is $\overline{g}$.
\item $\mathrm{deg}(g)=\mathrm{deg}(\overline{g})$.
\item The normalization $A'$ of $A$ inside the field extension $K(A)\subset{K(A)[y]/(g)}$ induced by $g$ is \'{e}tale at a prime $\mathfrak{m}'$ lying over $\mathfrak{m}$. 
\end{enumerate}  
It is said to be $\mathfrak{m}$-integral if $y\bmod{(g)}$ is in $A'_{\mathfrak{m}'}$. We similarly define $\mathfrak{m}'$-integral \'{e}tale lifts over extensions $A'\supset{A}$ that are \'{e}tale at a prime $\mathfrak{m}'$ lying over $\mathfrak{m}$. %
\end{mydef}

\begin{rem}
The algebra $A'$ injects into the (strict) Henselization of $A$ at $\mathfrak{m}$ by the third condition.  The integrality condition then says that $x\bmod{(g)}$ lies in $A^{\mathrm{sh}}_{\mathfrak{m}}$.   
\end{rem}

\begin{rem}
\label{SplitCompletelyResidue}
We will assume throughout the upcoming sections that the polynomial $\overline{g}$ splits completely over $k[[u]]$. The polynomials in the algorithm will always satisfy this condition, since they split completely over $\hat{A}_{\mathfrak{m}}=R[[u,v]]/(uv-\pi)$ and thus over the reduction $\hat{A}_{\mathfrak{m}}/\hat{\mathfrak{p}}\simeq{}k[[u]]$.%
\end{rem}
\begin{exa}
{\bf{[Non-example]}} We give an example here to show that not every lift works. 
Consider $R=\mathbf{Z}_{2}$ with $A=R[u]$. The polynomial $\overline{g}=y^3+u^2y+u^3\in\mathbf{F}_{2}(u)[y]$ is then irreducible and we can consider the lift $g=y^3+u^2y+u^3-2$. The normalization of $A$ inside $K(A)[y]/(g)$ is given by $A'=A[y]/(g)$ since the prime ideal $\mathfrak{m}'=(u,y,2)=(u,y)$ is regular (one also has to check the other primes for this; we leave this to the reader). This is furthermore the only prime lying over $\mathfrak{m}=(2,u)$ and the extension $A\subset{A'}$ is not \'{e}tale over this prime %
since the discriminant vanishes. We thus have an example of a non-\'{e}tale lift. A similar example over $R=\mathbf{Q}[[t]]$ can be given by replacing $2$ with $t$.   
\end{exa}

\subsubsection{Lifts in residue characteristic zero}\label{LiftResidueZero}

Suppose that $\mathrm{char}(k)=0$. The reader can think of %
$\overline{\mathbf{Q}}[[t]]$ here, but the uniformizer will be denoted by %
$\pi$. %
We will also work with $R[u]$ instead of $R[u,v]/(uv-\pi)$ for simplicity\footnote{The more general lifts can be obtained from this using the natural map $R[u]\to{R[u,v]/(uv-\pi)}$.}. %
Since $R$ is a complete discrete valuation ring of residue characteristic zero, we can find an isomorphism $R\simeq{k[[w]]}$, see \cite[Chapter 2, Theorem 2]{Ser1}. %
This then gives a %
ring-theoretic section $s:k\to{R}$ of the residual map $R\to{k}$. %
From this, we deduce a commutative diagram 
\begin{equation}
\begin{tikzcd}
k[u] \arrow{d} \arrow{r} & R[u] \arrow{d} \\
k[[u]] \arrow{r} & R[[u]]
\end{tikzcd}.
\end{equation}
There is a similar diagram for polynomial rings over these rings. We write $s(\cdot{})$ for any of these horizontal lifting maps. Suppose now that we are given an irreducible polynomial $\overline{g}\in{k[u][y]}$ with a root $\overline{c}$ over $k[[u]]$. The element $c:=s(\overline{c})$ is then a root of $s(\overline{g})$ by the commutativity of the above diagrams. By Theorem \ref{KummerDedekind}, we conclude that there is a prime $\mathfrak{m}'$ over $\mathfrak{m}=(u,\pi)$ at which the extension corresponding to $g$ is \'{e}tale. Since the image of $y\bmod{(g)}$ in the completion at $\mathfrak{m}'$ is integral (i.e., in $R[[u]]$), we then also easily find that $y\bmod{(g)}$ is in $A'_{\mathfrak{m}'}$. In other words, we have succeeded in finding an $\mathfrak{m}$-integral \'{e}tale lift of $\overline{g}$. 
\begin{rem}
Our construction of an $\mathfrak{m}$-integral \'{e}tale lift of $\overline{g}$ depends on the choice of a root of $\overline{g}$ over $k[[u]]$. In the upcoming algorithm the polynomial $\overline{g}$ will be completely split and one has to check each of the roots of $\overline{g}$ to obtain the $\mathfrak{m}$-adic power series for all the roots.  %
\end{rem}
\begin{rem}
To calculate $c$ up to a finite height, we can use the discrete Newton-Puiseux algorithm over $k[[u]]$. We then takes this approximation and lift it to $R[[u]]$ using our lifting map $s(\cdot{})$. 
\end{rem}

We now consider the algebra $R[u][y]/(g)\subset{A'}$, where $A'$ is the normalization of $A$ in $K(A)[y]/(g)$. Using the completion maps for the prime $\mathfrak{m}'$ %
and its pullback $\overline{\mathfrak{m}}'$ in $A'/(\pi)$, we then as before obtain a commutative diagram %
\begin{equation}
\begin{tikzcd}
k[u][y]/(\overline{g}) \arrow{d} \arrow{r} & R[u][y]/(g)\arrow{d} \\
k[[u]] \arrow{r} & R[[u]]
\end{tikzcd}.
\end{equation} 
If we now have an irreducible polynomial $\overline{g}_{1}$ over $k(\mathfrak{p}')=k(u)[y]/(g)$ with a root $\overline{c}_{1}$ in $k[[u]]$, then we can apply the same procedure to obtain an irreducible polynomial $g_{1}$, an extension $A''\supset{A'}$ and a maximal ideal $\mathfrak{m}''$ lying over $\mathfrak{m}'$ such that the extension is \'{e}tale at this maximal ideal. In other words, we obtain an $\mathfrak{m}'$-integral \'{e}tale lift of $\overline{g}_{1}$. By iterating this procedure, we obtain the lifts needed in the mixed Newton-Puiseux algorithm.

\begin{rem}
Instead of working with completions as we did in this section, we could also have worked with Henselizations. %
The corresponding commutative diagrams %
follow from \cite[\href{https://stacks.math.columbia.edu/tag/04GS}{Lemma 04GS}]{stacks-project} by localizing $k[u]\to{R[u]}$ at $\overline{\mathfrak{m}}=(u)$ and $\mathfrak{m}=(u,\pi)$. The other statements follow similarly.  %
\end{rem}

\subsubsection{Lifts in residue characteristic $p$}

\label{LiftCharacteristicp}

We now show how to obtain adjusted lifts if $R$ is of residue characteristic $p>0$. The problem we have to deal with now is that there is no canonical lift of the $c_{i}$ as in the previous section. %
We can however do the following. Let $\overline{g}$ be an irreducible polynomial over $k(\mathfrak{p})=k(u)$, where $\mathfrak{p}=(v,\pi)$ and $A=R[u,v]/(uv-\pi)$. We assume as in Remark \ref{SplitCompletelyResidue} that $\overline{g}$ splits completely over $k[[u]]$. %
We now apply the discrete Newton-Puiseux algorithm to $\overline{g}$ over $k[[u]]$ up to a separating height for some fixed root of $\overline{g}$ over $k[[u]]$. This gives a polynomial $\overline{g}_{a}$ that induces the same extension over $k(\mathfrak{p})$ as $\overline{g}$, but with a simple root modulo $(u)$. 
\begin{exa}
Consider the polynomial $\overline{g}=y^{8}-u^{4}(u+1)y^{4}+u^{8}(u+1)$ over $\mathbf{F}_{p}(u)$ with $p\neq{2},3$. We have that $\overline{g}\in{\mathbf{F}_{p}[u][y]}$ is irreducible since it is Eisenstein at $(u+1)$. It does not have a simple root modulo $(u)$ however. We scale $\overline{g}$ by setting $\overline{g}_{1}:=1/u^{8}\cdot{}\overline{g}(uy)=y^{8}-(u+1)y^{4}+(u+1)$. The discriminant of this polynomial modulo $(u)$ is nonzero for $p\neq{2,3}$, so it has no double roots. %
\end{exa}
We now take {\it{any}} lift $g_{a}$ of $\overline{g}_{a}$. By definition, $g_{a}$ has a simple root modulo $\mathfrak{m}$. The polynomial $g_{a}$ will then have a root over the Henselization or completion of $A$ at $\mathfrak{m}$, which corresponds to a prime $\mathfrak{m}'$ lying over $\mathfrak{m}$ in the normalization of $A$ in $K(A)[y]/(g_{a})$. We write $c':=y\bmod{(g_{a})}$. We now recall that $\overline{g}_{a}$ is obtained from $\overline{g}$ by a series of translations and scalings. This implies that we can find a polynomial $\overline{r}\in{k[u]}$ such that $\overline{r}\cdot{\overline{c'}}$ is a root of $\overline{g}$. We then take a lift $r$ of $\overline{r}$ and continue the algorithm with $c:=r\cdot{c'}$. The lifts over further extensions of $k(u)$ are obtained as in the previous section. %

\begin{rem}
To calculate the $\mathfrak{m}$-adic power series of $c'$ (and thus of $c$), we use Hensel's lemma as in \cite[Chapter IV, Lemma 1.2]{Silv1} with $I=\mathfrak{m}$. %
\end{rem}

\subsubsection{The mixed Newton-Puiseux algorithm}

In this section we give the mixed Newton-Puiseux algorithm. We will only write down a version for residue characteristic zero here using the technique in Section \ref{LiftResidueZero}. %
The algorithm using the lifts given in Section \ref{LiftCharacteristicp} is similar. %

As in Section \ref{PreliminariesAlgorithms}, we assume that $A$ is the localization of $R[u,v]/(uv-\pi)$ at $\mathfrak{m}$ and that $f\in{A[x]}$ splits completely over the strict Henselization of $A$ at $\mathfrak{m}$. We will only give the algorithm in this case, the general case follows from this as in the discrete algorithm. We will also assume throughout this section that the leading coefficient of $f$ is invertible in $A^{\mathrm{sh}}_{\mathfrak{m}}$. %
We now first define the output of the upcoming algorithm. 
\begin{mydef}\label{IntegralAdicApproximations}{\bf{[$(\mathfrak{m},\mathfrak{p})$-adic approximations]}} %
Let $A^{\mathrm{sh}}_{\mathfrak{m}}$ be the strict Henselization of $A$ at $\mathfrak{m}$ and let $\alpha$ be a root of $f$. A pair $(r_{\mathfrak{m}},r_{\mathfrak{p}})$ of elements in $A^{\mathrm{sh}}_{\mathfrak{m}}$ is said to be an $(\mathfrak{m},\mathfrak{p})$-adic approximation of height $(h_{\mathfrak{m}},h_{\mathfrak{p}})$ of $\alpha$ if $\alpha-r_{\mathfrak{p}}\in{(v^{{h_{\mathfrak{p}}}})}$ and $\alpha-r_{\mathfrak{m}}\in(u^{{h_{\mathfrak{m}}}},v^{h_{\mathfrak{p}}})$.
\end{mydef}

To find these approximations, we apply the algorithm in Section \ref{PreliminariesAlgorithms} with uniformizer $v$. %
We first calculate an irreducible factorization of $f$ modulo $(v)$:
\begin{equation}
\overline{f}=\prod_{i=1}^{d}\overline{g}^{d_{i}}_{i}.
\end{equation} 
The polynomials $\overline{g}_{i}$ split completely over $k[[u]]$ since $\overline{f}$ splits completely over $k[[u]]$. We can thus use the procedure in Section \ref{LiftResidueZero} to find $\mathfrak{m}$-integral \'{e}tale lifts of the $\overline{g}_{i}$. If we translate over a root of one of these lifts, then the translated polynomial $f_{1}$ together with its roots will still be defined over $A^{\mathrm{sh}}_{\mathfrak{m}}$ and similarly for the scaling operation (here we use that $A^{\mathrm{sh}}_{\mathfrak{m}}$ is a unique factorization domain and that $v$ is an irreducible element in $A^{\mathrm{sh}}_{\mathfrak{m}}$). %
We can then continue this process to obtain the desired $\mathfrak{p}$-adic power series expansions. To obtain the $\mathfrak{m}$-adic power series expansions, we write 
\begin{equation}
\alpha=\sum_{i=0}^{r}c_{i}v^{i}+\mathcal{O}(v^{r+1})
\end{equation}        
and calculate the $(u)$-adic power series expansions of the $\overline{c}_{i}$ up to a finite height using the discrete Newton-Puiseux algorithm. This also gives the $\mathfrak{m}$-adic power series expansions by the lifting maps in Section \ref{LiftResidueZero}. Combining the above techniques, we then arrive at the following algorithm.  %

\begin{algorithm}[H]
  \caption{The mixed Newton-Puiseux algorithm. }
    \label{MixedNPAlgorithm3a}
  \vspace*{0.1cm}
\textbf{Input:} A polynomial $f\in{A}[y]$ that splits completely over the strict Henselization $A^{\mathrm{sh}}_{\mathfrak{m}}$, 
 a pair of target heights $(h_{\mathfrak{m}},h_{\mathfrak{p}})$. \\
 \textbf{Output:} {The $(\mathfrak{m},\mathfrak{p})$-adic approximations of the roots of $f$ up to height $(h_{\mathfrak{m}},h_{\mathfrak{p}})$. %
  }
  \begin{algorithmic}[1]
 \State Calculate the Newton polygon $\mathcal{N}(f)$ of $f$ with respect to ${\mathfrak{p}}$. %
 
 \State   
 If there is no line segment of slope $k>0$, go to Step $4$ with $f_{1}=f$. Otherwise, choose a line segment $e$ in $\mathcal{N}(f)$ of slope $k>0$. %
 \State 
 Update the power series of the approximations $r_{\mathfrak{p}}$ and $r_{\mathfrak{m}}$ by adding zeros. 
 \State If the precision of an approximation is greater than or equal to $h_{\mathfrak{p}}$, then stop the algorithm.    %
Otherwise calculate $f_{1}=1/v^{m}f(v^{k}y)$, where $m$ is such that $f_{1}$ is primitive with respect to $\mathfrak{p}$.  
 \State Calculate an irreducible factorization  
   \begin{equation}
  \overline{f}_{1}=\prod_{i=1}^{l}\overline{g}_{i}^{r_{i}}
  \end{equation} of $\overline{f}_{1}$ over the residue field of $\mathfrak{p}$.
    \State Choose an irreducible factor $\overline{g}_{i}$ and lift it to a polynomial $g_{i}$ using the methods in Section \ref{LiftResidueZero}. Adjoin a root $\gamma_{i}$ of $g_{i}$. Update 
 $r_{\mathfrak{p}}$ using the root $\gamma_{i}$. 
  \State %
 Calculate a power series expansion %
 $r':=\sum_{i=0}^{h_{\mathfrak{m}}-1}d_{i}u_{2}^{i}$ of a root of $\overline{g}_{i}$ up to height $h_{\mathfrak{m}}$ using the discrete Newton-Puiseux algorithm. Update $r_{\mathfrak{m}}$ using the canonical lift of $r'$. %
  \State Calculate $f_{2}=f(y+\gamma_{i})$ and return to Step $1$ with %
  $f:=f_{2}$.   
     
\end{algorithmic}
\end{algorithm}

\begin{theorem}\label{MixedNPValidity}
 Let $f\in{A}[y]$ be a polynomial that splits completely over the strict Henselization $A^{\mathrm{sh}}_{\mathfrak{m}}$ and consider the mixed Newton-Puiseux algorithm. %
 This algorithm correctly computes the $(\mathfrak{m},\mathfrak{p})$-adic approximations of the roots of $f$ up to height $(h_{\mathfrak{m}},h_{\mathfrak{p}})$. %
 \end{theorem}

\begin{proof}
The correctness of the algorithm follows directly from the considerations in the previous sections.
\end{proof}
\begin{rem}
 If $f$ splits completely over the strict Henselization of $A$ at $\mathfrak{m}$ after a finite Kummer extension of $A$, then we can still run the algorithm, but we have to be add $m$-th roots of $u$ and $v$ as in Algorithm \ref{DiscreteNPAlgorithm}. If we do this conservatively as in Step $2$ of Algorithm \ref{DiscreteNPAlgorithm}, then the argument in the proof of Theorem \ref{CorrectnessTheoremDiscrete} shows that the algorithm doesn't go outside a given Kummer extension over which $f$ splits.  
We conclude that the algorithm also computes the right power series in this case.  
 \end{rem}

We now illustrate Algorithm \ref{MixedNPAlgorithm3a} by applying it to our Main Example. %

\begin{exa}\label{TransferDiscreteSemistable3}

{\bf{[Main Example]}} We will study the covering Example \ref{MainExample} for the closed edge defined by $0\leqslant{v(x)}\leqslant{4}$. %
We use Algorithm \ref{MixedNPAlgorithm3a} to calculate the $\mathfrak{m}_{e}$-adic and $\mathfrak{p}_{v}$-adic power series of the roots of the defining polynomial $f$. Here $e$ is the edge $0<v(x)<4$ %
and $v$ is one of the adjacent vertices of $e$. %

\vspace{0.2cm}
{\underline{{\bf{Case I: $0\leqslant{v(x)}<4$}}}}
\vspace{0.2cm}

We start with the algebra $A=R[x_{0},x_{1}]/(x_{0}x_{1}-t^{4})$, which is embedded into $K(x)$ using 
\begin{align}
x_{0}&\mapsto{x},\\
x_{1}&\mapsto{t^4/x}.
\end{align}
The corresponding regular Kummer extension $A\subset{}A_{\mathrm{reg}}=R[u,v]/(uv-t)$ is then given by $u^4=x_{0}$ and $v^4=x_{1}$. We set $A=A_{\mathrm{reg}}$ and write $\mathfrak{p}=(v)$ and $\mathfrak{m}=(u,v,t)$. We first make the polynomial $f$ monic by setting $f:=g^3f(y/g)$, where $g=u^2v^2+1$. The mixed Newton-Puiseux algorithm is not very different from the original discrete application since the $\mathfrak{p}$-adic coefficients are polynomials in $u$ for the first two iterations. The $\mathfrak{m}$-adic (and $\mathfrak{p}$-adic) approximations are %
\begin{align*}
r_{\mathfrak{m},1,1}&=u^2+iu^3v/2, \\ %
r_{\mathfrak{m},2,1}&=%
u^2-iu^3v/2,\\ %
r_{\mathfrak{m},3,1}&=-u^2+iu^3v/2,\\
r_{\mathfrak{m},4,1}&=-u^2-iu^3v/2.%
\end{align*}

\vspace{0.2cm}
{\underline{{\bf{Case II: $0<{v(x)}\leqslant{}4$}}}}
\vspace{0.2cm}

We use the same algebra as before, but we now apply the algorithm with $\mathfrak{p}=(u)$. We take the monic variant of $f$ as above and find that its $(u)$-Newton polygon contains a line segment of slope $2$. We thus calculate $f_{1}=1/u^8f(u^2y)$, which reduces to 
\begin{equation}
\bar{f}_{1}=(y^2-1)^2.
\end{equation}
We first translate over $1$ and calculate $f_{2}:=f_{1}(y+1)$, whose Newton polygon has a line segment of slope $1$. We calculate $f_{3}=1/u^2f_{2}(uy)$ and its reduction 
\begin{equation}
\bar{f}_{3}=4y^2+v^4+v^2.
\end{equation}
This is an irreducible polynomial over $k(\mathfrak{p})=\overline{\mathbf{Q}}(v)$. We take the natural lift $g_{3,0}$ of this polynomial to $A[x]\subset{A^{\mathrm{h}}_{\mathfrak{m}}[x]}$ and add a root to obtain the algebra $A_{0}=A^{\mathrm{h}}_{\mathfrak{m}}[x]/(g_{3,0})$. %
The roots of the polynomial %
$g_{3,0}$ over $A^{\mathrm{sh}}_{\mathfrak{m}}$ can be %
 $\mathfrak{m}$-adically approximated by %
\begin{align}
\beta_{1}&=1/2iv + 1/4iv^3 - 1/16iv^5+\mathcal{O}(v^6),\\
\beta_{2}&=-1/2iv-1/4iv^3+1/16iv^5+\mathcal{O}(v^6).
\end{align}  
Here we used the discrete Newton-Puiseux algorithm over $\overline{\mathbf{Q}}[[v]]$ to calculate the approximations.

We now translate over $-1$ with $f_{2,1}:=f_{1}(y-1)$. Repeating the same process as before then gives a polynomial $f_{3,1}$ whose reduction is  
\begin{equation}
\bar{f}_{3,1}=4y^2-v^4+v^2.
\end{equation}
We write $A_{1}=A^{\mathrm{h}}_{\mathfrak{m}}[x]/(g_{3,1})$ for the corresponding \'{e}tale extension of $A^{\mathrm{h}}_{\mathfrak{m}}$. %
The roots of $g_{3,1}$ over $A^{\mathrm{sh}}_{\mathfrak{m}}$ can be $\mathfrak{m}$-adically approximated by %
\begin{align}
\beta_{3}&=1/2iv - 1/4iv^3 - 1/16iv^5 +\mathcal{O}(v^6),\\
\beta_{4}&=-1/2iv + 1/4iv^3 + 1/16iv^5+\mathcal{O}(v^6).
\end{align}
We now have the following $\mathfrak{p}$-adic approximations of the roots of $f$:
\begin{align}
r_{\mathfrak{p},1}&=u^2+\beta_{1}u^3,\\
r_{\mathfrak{p},2}&=u^2+\beta_{2}u^3,\\
r_{\mathfrak{p},3}&=-u^2+\beta_{3}u^3,\\
r_{\mathfrak{p},4}&=-u^2+\beta_{4}u^3.
\end{align}  
Here the $\beta_{i}$ are considered as abstract elements of the strict Henselization. More specifically, they are the images of $x\bmod{(g_{3,0})}$ and $x\bmod{(g_{3,1})}$ under embeddings of $A_{0}$ and $A_{1}$ in the strict Henselization.  %
Using the earlier approximations of the $\beta_{i}$, we then obtain the $\mathfrak{m}$-adic approximations %
\begin{align}
r_{\mathfrak{m},1,2}&=u^2+(1/2iv + 1/4iv^3 - 1/16iv^5)u^3,\\
r_{\mathfrak{m},2,2}&=u^2+(-1/2iv-1/4iv^3+1/16iv^5)u^3,\\
r_{\mathfrak{m},3,2}&=-u^2+(1/2iv - 1/4iv^3 - 1/16iv^5)u^3,\\
r_{\mathfrak{m},4,2}&=-u^2+(-1/2iv + 1/4iv^3 + 1/16iv^5)u^3.
\end{align}

We have $r_{\mathfrak{p},i}-\alpha_{i}\in{(u^{4})}$ and $r_{\mathfrak{m},i,2}-\alpha_{i}\in(u^4,v^6)$. We invite the reader to compare the $r_{\mathfrak{m},i,2}$ to the $r_{\mathfrak{m},i,1}$ from Case I. %

\end{exa}

Suppose that the mixed Newton-Puiseux algorithm has calculated approximations $(r_{\mathfrak{m},i},r_{\mathfrak{p},i})$ up to height $(h_{\mathfrak{m}},h_{\mathfrak{p}})$.  %
We now assume that $h_{\mathfrak{m}}$ and $h_{\mathfrak{p}}$ are large enough so that $I:=(u^{h_{\mathfrak{m}}},v^{h_{\mathfrak{p}}})$ is separating for the approximations $r_{\mathfrak{m},i}$, see Definition %
\ref{SeparatingMonomialsRegular}. The approximations $r_{\mathfrak{p},i}$ are then also $\mathfrak{p}^{h_{\mathfrak{p}}}$-adically separated since %
$\mathfrak{p}^{h_{\mathfrak{p}}}\subset{I}$. %
For every root $\alpha_{i}\in{A^{\mathrm{sh}}_{\mathfrak{m}}}$, there now is a unique pair %
$(r_{\mathfrak{m},i},r_{\mathfrak{p},i})$ with $r_{\mathfrak{m},i}, r_{\mathfrak{p},i}\in{A^{\mathrm{sh}}_{\mathfrak{m}}}$ %
satisfying $\alpha_{i}-r_{\mathfrak{m},i}\in(u^{h_{\mathfrak{m}}},v^{h_{\mathfrak{p}}})$ and $\alpha_{i}-r_{\mathfrak{p},i}\in(v^{h_{\mathfrak{p}}})$. If we can now calculate the $D_{\mathfrak{p}}$-orbit of $r_{{\mathfrak{p},i}}$ and the $D_{\mathfrak{m}}$-orbit of $r_{\mathfrak{m},i}$, then by Theorem \ref{Dedekind3} this uniquely determines the chain lying over $\mathfrak{m}\supset{\mathfrak{p}}$. %
For the $D_{\mathfrak{p}}$-orbit we use %
Remark \ref{ExtensionsMinimalPolynomialsAlgorithm}. To calculate the $D_{\mathfrak{m}}$-orbit, %
we use the following. %
\begin{rem}\label{ActionEdges}
{\bf{[Calculating $D_{\mathfrak{m}}$-orbits]}} Suppose that we have found a finite Kummer extension %
$A'$ of $A=R[[u,v]]/(uv-\pi^{n})$ such that $f$ splits completely over %
$A'$. %
By taking a large enough extension of $R$, we can assume that $A\subset{}A'$ is \'{e}tale over the special fiber. By Lemma \ref{LemmaKummer33}, this implies that $K(A')$ is a tame subextension of the regularization 
$K(A_{\mathrm{reg}})$ of $A$. In particular, it is Galois. We write $u_{1}$ and $v_{1}$ for the elements such that %
$u_{1}^{n}=u$ and $v_{1}^{n}=v$. For simplicity, we now assume that %
$K(A)$ contains a primitive $n$-th root of unity $\zeta_{n}$ and that $K(A_{\mathrm{reg}})/K(A)$ is separable, so that it is Galois with Galois group $G=\mathbf{Z}/n\mathbf{Z}$. %
We can then find a generator $\sigma$ of $G$ that acts %
on $u_{1}$ and $v_{1}$ by %
\begin{align*}
u_{1}&\mapsto{\zeta_{n}u_{1}},\\
v_{1}&\mapsto{\zeta^{-1}_{n}v_{1}}.
\end{align*}
To find the orbit of the roots $\alpha_{i}$ under this action, we calculate the action of $\sigma$ on the approximations $r_{\mathfrak{m},i}$ (which are polynomials in $u_{1}$ and $v_{1}$).
More precisely, we calculate the action modulo $I$, where %
$I=(u^{{h}_{\mathfrak{m}}},v^{{h}_{\mathfrak{p}}})$ is the separating invariant ideal for the approximations as before. %
This then completely determines the action of $D_{\mathfrak{m}}$ on the $\alpha_{i}$. %

\end{rem}

\begin{rem}\label{LengthEdge1}
{\bf{[Lengths of the edges]}} We can use the $D_{\mathfrak{m}}$-orbit calculated in Remark \ref{ActionEdges} to calculate the length of an edge. Namely, the length of the edge is the length of the original edge $\ell(e)$ divided by the degree of the field extension corresponding to an orbit, see Lemma \ref{LemmaKummer33}. This degree in turn is equal to the number of roots in an orbit, so this gives a way to calculate the lengths of the edges lying above an edge. %
\end{rem}

\begin{rem}
{\bf{[Calculating $D_{\mathfrak{p}}$-orbits]}} In order to determine the $D_{\mathfrak{p}}$-orbit as in Remark \ref{ExtensionsMinimalPolynomialsAlgorithm}, we have to transform the power series in $v$ to a power series in $\pi$ (since the action on $v$ is non-trivial). Effectively, this means that we have to divide our power series by powers of $u$. %
\end{rem}

We now use the aforementioned techniques to calculate the skeleton of the covering of curves in Example \ref{MainExample}. %

\begin{exa}\label{MainExampleFullSkeleton}
{\bf{[Main Example]}} We calculate the action of the decomposition groups on the power series given in Example \ref{TransferDiscreteSemistable3} and glue the edges and vertices using this information.

\vspace{0.2cm}
{\underline{{\bf{Case I: $0\leqslant{v(x)}<{}4$}}}}
\vspace{0.2cm}

We first calculate the $D_{\mathfrak{p}}$-orbit using Remark \ref{ExtensionsMinimalPolynomialsAlgorithm}. In terms of the uniformizer $t=uv$, we then find that the roots are all minimally defined over $\overline{\mathbf{Q}}(x_{0})(u^2)$, which is of degree $2$ over $\overline{\mathbf{Q}}(x_{0})$. There are thus two vertices lying over $v(x)=0$. %
We now calculate the action of $D_{\mathfrak{m}}$ on the approximations. %
Note that this action sends $u$ to $iu$ and $v$ to $i^3v$. We then see that $\{r_{\mathfrak{m},1,1},r_{\mathfrak{m},4,1}\}$ and $\{r_{\mathfrak{m},2,1},r_{\mathfrak{m},3,1}\}$ form the two orbits under $D_{\mathfrak{m}}$. We thus have two edges over $e:0<v(x)<4$. %

\vspace{0.2cm}
{\underline{{\bf{Case II: $0<{v(x)}\leqslant{}4$}}}}
\vspace{0.2cm}

In this case, the $\mathfrak{p}$-adic power series are not written in terms of $t$, but we can do this by using the relation $uv=t$, so that $u^3=\dfrac{t^3}{v^3}$. We consider the coefficients $\{1/v^2,\beta_{1}/v^3\}$ of the first approximation $r_{\mathfrak{p},1}$. We have that the extension $\overline{\mathbf{Q}}(x_{1})\subset\overline{\mathbf{Q}}(x_{1})(1/v^2,\beta_{1}/v^3)$ is of degree $4$, so there are four roots in the $D_{\mathfrak{p}}$-orbit of $r_{\mathfrak{p},1}$ and thus in the $D_{\mathfrak{p}}$-orbit of $\alpha_{1}$. In other words, there is only one vertex lying over $v(x)=4$. The curve corresponding to the field $\overline{\mathbf{Q}}(x_{1})(1/v^2,\beta_{1}/v^3)$ has genus $0$, so the genus of this vertex is $0$.  %
Note also that the $r_{\mathfrak{m},i,1}$ correspond exactly to the $r_{\mathfrak{m},i,2}$, so we can link the corresponding chains. %

By continuing in this way with the other annuli given in Example \ref{SeparatingMainExample} we obtain the skeleton of $X$, see Figure \ref{Eindresultaat} for the final result. Here the vertices lying over $v(x)=8$ and $v(x)=-8$ have arithmetic genus $1$. The first Betti number of the graph is $1$ and the length of the non-trivial cycle is $8$.

\end{exa}

\subsubsection{Switching between local presentations}\label{LocalPresentations1}

In the last part of this paper, we explain how to transfer various Newton-Puiseux data between two local presentations as given in Section \ref{P1Models}. We give an elaborate example using a genus three curve that illustrates this transfer. %

Consider an annulus $\mathbf{S}_{a,b}(x_{0})$ with coordinate ring $A_{1}:=R[u_{1},v_{1}]/(u_{1}v_{1}-\pi^{n})$ %
and embedding $i:A_{1}\to{K(x)}=K(\mathbf{P}^{1})$, see Section \ref{P1Models}.  %
We start the mixed Newton-Puiseux method by passing to the regularization of $A_{1}$ given by %
$A'_{1}=R[u,v]/(uv-\pi)$. %
This gives a finite Kummer extension of degree $n$ on the residue field of $\mathfrak{p}_{2}$. Note however that if we have another ring $A_{2}$ whose spectrum is glued to that of $A_{1}$ along some open affine containing $\mathfrak{p}_{2}$, then the induced Kummer extensions can be different. %
\begin{exa}\label{PresentationSwitchExample}
Consider the algebras $\mathbf{A}_{\mathbf{S}_{0,2}(0)}$ and $\mathbf{A}_{\mathbf{S}_{2,5}(\pi^{2})}$. %
These algebras are obtained by embedding $A_{1}=%
R[u_{1},v_{1}]/(u_{1}v_{1}-\pi^{2})$ and $A_{2}=%
R[u_{2},v_{2}]/(u_{2}v_{2}-\pi^{3})$ into $K(x)$ as follows: %
\begin{align*}
u_{1}&\mapsto{x}\\
v_{1}&\mapsto{\pi^{2}/x}\\
u_{2}&\mapsto{\dfrac{x-\pi^{2}}{\pi^{2}}}\\
v_{2}&\mapsto{\dfrac{\pi^{5}}{x-\pi^{2}}}
\end{align*}
These algebras are glued over the opens $D(u_{2})\subset{\mathrm{Spec}(A_{2})}$ and $D(1-v_{1})\subset{\mathrm{Spec}(A_{1})}$ by the equations $u_{2}=\dfrac{1-v_{1}}{v_{1}}$ and $v_{2}=\pi^{3}/u_{2}=\pi^{3}v_{1}/(1-v_{1})$. %
Note that the codimension one prime  $\mathfrak{p}_{2}=(v_{2},t)$ is identified with $\mathfrak{p}_{1}=(u_{1},t)$ under this isomorphism. To obtain regular algebras, we now take the regularizations of these algebras. %
For $A_{1}$, this means that we take a square root of $u_{1}$ and $v_{1}$, and for $A_{2}$ this means that we take a cube root of $u_{2}$ and $v_{2}$. These induce different extensions of $K(x)$. %
\end{exa}

In our computations, this means the following for us. We have a codimension-$1$ point $\eta_{v}\in\mathcal{Y}$ corresponding to an irreducible component of the special fiber. %
Our computations with respect to $\eta_{v}$ are then contained in $\mathcal{O}^{\mathrm{sh}}_{\mathcal{Y},\eta_{v}}$, but we use two different presentations for this ring. That is, the uniformizer $\pi$ is the same, but the %
 parameter of the residue field is different. These parameters differ by a standard M\"{o}bius transformation on $k(\eta_{v})=k(\mathbf{P}^{1})$, see Example \ref{PresentationSwitchExample} for instance. To check which power series correspond to which under this transformation, we use Remark \ref{ExtensionsMinimalPolynomialsAlgorithm}. In practice, this means that we apply the M\"{o}bius transformation to the minimal polynomials of the coefficients. %
To illustrate the above, we now use the algorithms on an example provided by Prof. Hannah Markwig. 

\begin{exa}\label{HannahExample1}
Consider the curve $X/\overline{\mathbf{Q}}((t))$ given by
\begin{equation}\label{PolynomialHannah}
f:=t_{2}^{24}x^4-x^2y^2+t_{2}^8xy^3+t_{2}^{18}y^4-2x^2y+t_{2}^3xy^2+t_{2}^{12}y^3-(t_{2}^4-1)x^2+xy+t_{2}^8y^2+t_{2}^6x+t_{2}^{11}y+t_{2}^{18}=0,
\end{equation}
where $t_{2}=t^6$. 
\begin{figure}[h]
\begin{center}
\includegraphics[width=10cm, height=4.0cm]{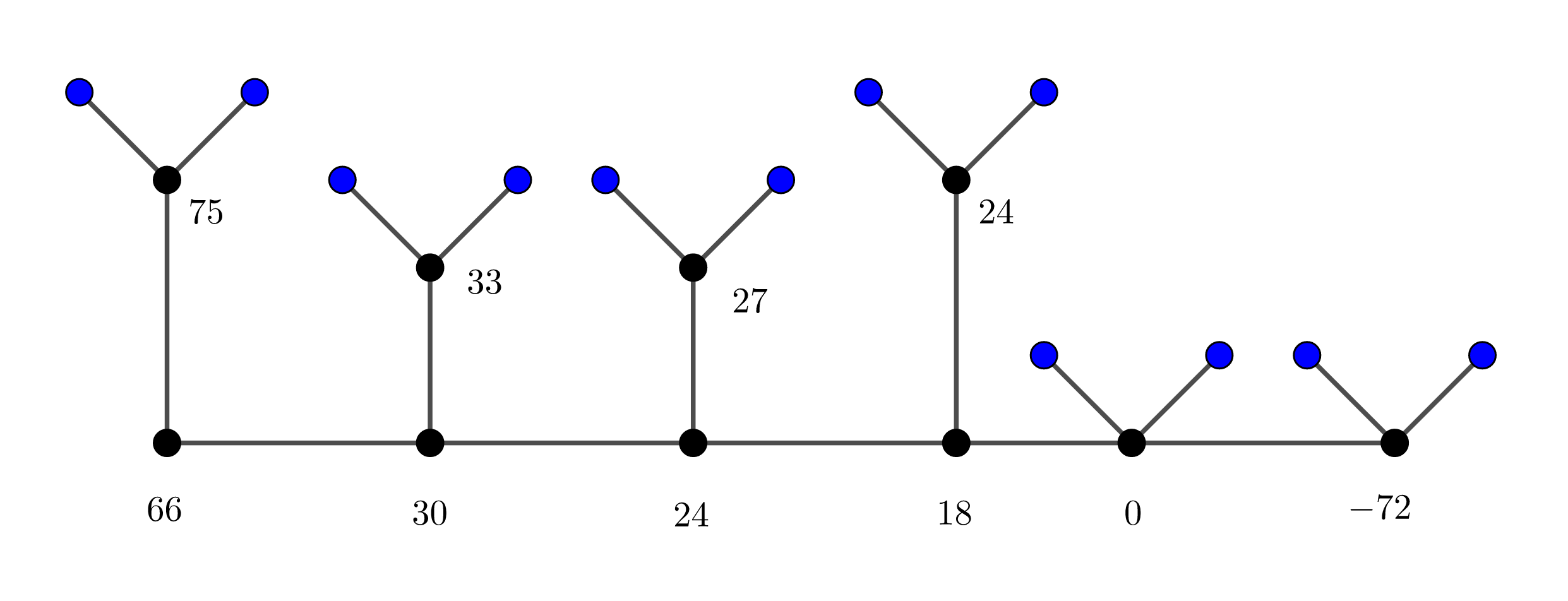}
\caption{\label{SeparatingTreeHannah1} The separating tree for the covering in Example \ref{HannahExample1}. The integers next to the vertices correspond to the boundaries of the algebras given in Equations \ref{SeparatingTreeAlgebras2a} and \ref{SeparatingTreeAlgebras3a}.}%
\end{center}
\end{figure}

As in our Main Example, this defines a smooth quartic in $\mathbf{P}^{2}$, so its genus is $3$. We consider the degree $4$ covering given by $(x,y)\mapsto{x}$, or in other words we consider $f$ as a polynomial in $K(x)[y]$, where $K=\overline{\mathbf{Q}}((t))$. The separating tree for this covering is given in Figure \ref{SeparatingTreeHannah1}. The corresponding algebras in terms of Section \ref{P1Models} are 
\begin{align}
&\mathbf{A}_{\mathbf{S}_{66,75}(-t^{66})},\,\,\mathbf{A}_{\mathbf{S}_{30,66}(0)},\,\,\mathbf{A}_{\mathbf{S}_{30,33}(-t^{30})},\,\,\mathbf{A}_{\mathbf{S}_{24,27}(0)},\,\,\mathbf{A}_{\mathbf{S}_{24,27}(-t^{24})},\label{SeparatingTreeAlgebras2a}\\
&\mathbf{A}_{\mathbf{S}_{18,24}(0)},\,\,\mathbf{A}_{\mathbf{S}_{18,24}(t^{18})},\,\,\mathbf{A}_{\mathbf{S}_{0,18}(0)},\,\,\mathbf{A}_{\mathbf{S}_{-72,0}(0)}.\label{SeparatingTreeAlgebras3a}
\end{align} Using the algorithms in this paper, we find that over each of the non-leaf edges, the covering is completely split, in the sense that there are $4$ edges lying above every edge. Over the vertices adjacent to leaves, we find that there are $3$ vertices, with two trivial coverings and one of degree $2$. Thus the local picture is as in Figure \ref{LocalCoveringHannah1}. We are now interested in connecting these local coverings to obtain the global covering.
\begin{figure}[h]
\begin{center}
\includegraphics[height=6cm]{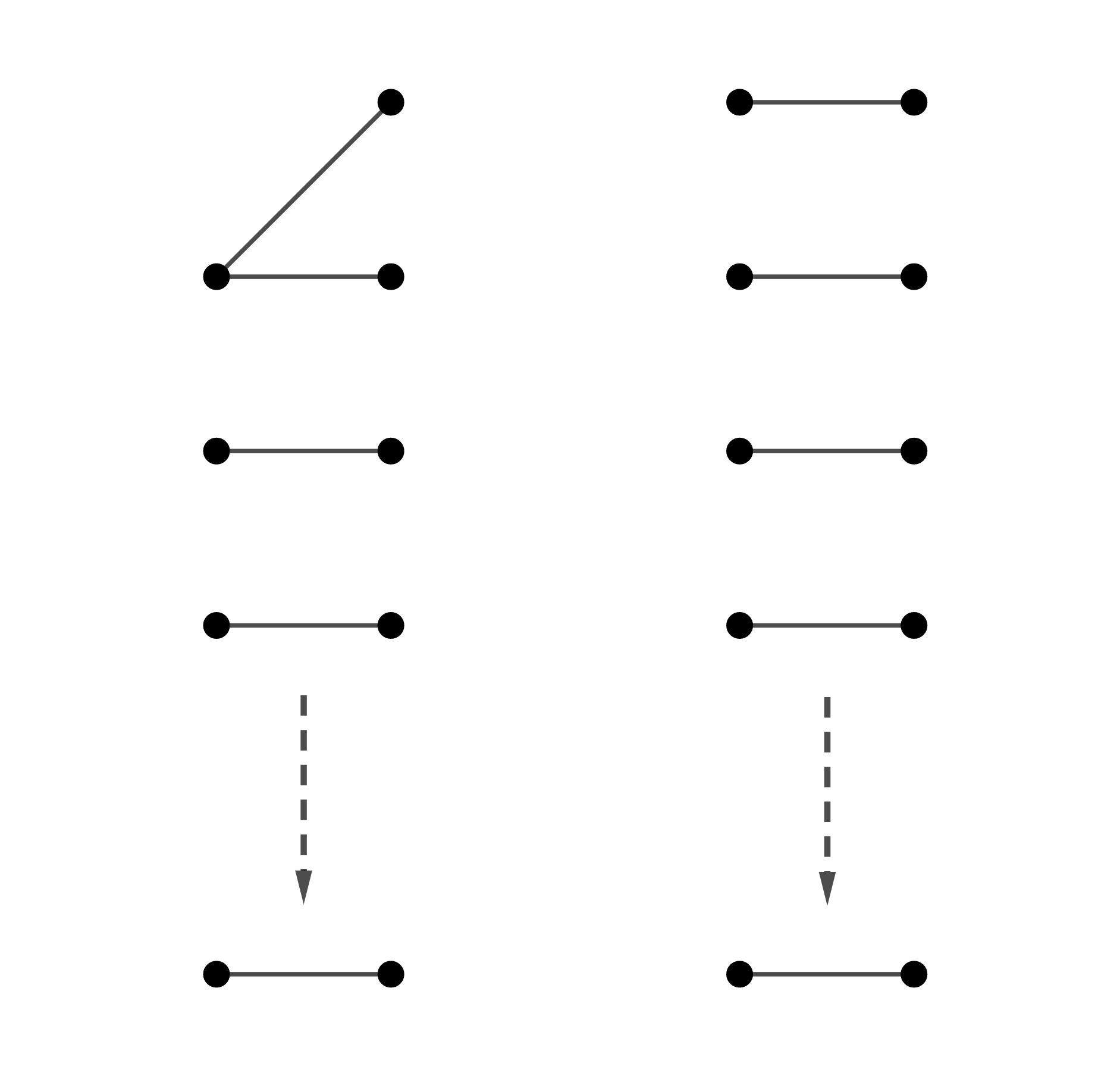}
\caption{\label{LocalCoveringHannah1} The local coverings in Example \ref{HannahExample1}.}
\end{center}
\end{figure}

We calculate the $\mathfrak{p}_{v}$-adic and $\mathfrak{m}_{e}$-adic power series approximations of $f$ for the vertices $v$ and edges $e$ in the half-closed annuli given by %
$66\leq{v(x+t^{66})}<{75}$ and $30<{v(x)}\leq{66}$. %
The algebras corresponding to these two closed annuli are $\mathbf{A}_{\mathbf{S}_{66,75}(-t^{66})}$ and $\mathbf{A}_{\mathbf{S}_{30,66}(0)}$, see Section \ref{P1Models}. The vertex $v_{2}=\{x:v(x)=66\}$ gives a codimension one prime in both of these algebras and we obtain two different expressions of the $\mathfrak{p}_{v_{2}}$-adic power series of the roots of $f$ using the two local presentations. 

\vspace{0.2cm}
{\underline{{\bf{Case I: $66\,\leq{}v(x+t^{66})<{}75$}}}}
\vspace{0.2cm}

Writing $A_{1}=R[u_{1},v_{1}]/(u_{1}v_{1}-t^{9})$ and $A_{2}=R[u_{2},v_{2}]/(u_{2}v_{2}-t^{36})$, we have that $\mathbf{A}_{\mathbf{S}_{66,75}(-t^{66})}$ and $\mathbf{A}_{\mathbf{S}_{30,66}(0)}$ are the images of the embeddings defined by  %
\begin{align*}
u_{1}&\mapsto{\dfrac{x+t^{66}}{t^{66}}},\\
v_{1}&\mapsto{\dfrac{t^{75}}{x+t^{66}}},\\
u_{2}&\mapsto{\dfrac{x}{t^{30}}},\\
v_{2}&\mapsto{\dfrac{t^{66}}{x}}.
\end{align*}

These algebras are glued over $D(u_{1}-1)$ and $D(v_{2})$ using the relations $v_{2}=1/(u_{1}-1)$ and $u_{2}=t^{33}(u_{1}-1)$. We first pass to the regularization of $A_{1}$ given by $R[u,v]/(uv-t)$. %
Here $u^9=u_{1}$ and $v^9=v_{1}$. We then transform the polynomial $f$ as given in Equation \ref{PolynomialHannah} so that it becomes monic over $\mathbf{A}_{\mathbf{S}_{66,75}(-t^{66})}$. Specifically, we calculate $f_{1}:=t_{2}^{36}f(y/t_{2}^{36})$. We now calculate approximations of the roots of $f_{1}$ using Algorithm \ref{MixedNPAlgorithm3a} with the pair $\mathfrak{p}=(v,t)\subset{(u,v,t)}=\mathfrak{m}$ with uniformizer $v$. This gives the $\mathfrak{m}$-adic and $\mathfrak{p}$-adic approximations 
\begin{align*}
r_{1,2}&=-1+\mathcal{O}(v),\\
r_{2,2}&=-u^{12}v^{12}+\mathcal{O}(v^{13}),\\
r_{3,2}&=-u^{63}v^{54}+\mathcal{O}(v^{55}),\\
r_{4,2}&=(-u^{72}+u^{63})v^{72}+\mathcal{O}(v^{73}).%
\end{align*}
We thus find that there are four vertices lying over $v_{2}$. 

\vspace{0.2cm}
{\underline{{\bf{Case II: $30<v(x){\leq}\,66$}}}}
\vspace{0.2cm}

We first pass to the regularization $R[n,m]/(nm-t)$ of the algebra $A_{2}=R[u_{2},v_{2}]/(u_{2}v_{2}-t^{36})$ using $n^{36}=u_{2}$ and $m^{36}=v_{2}$. The vertex $v(x)=66$ is given by the prime ideal $(n,t)$. We use Algorithm \ref{MixedNPAlgorithm3a} for the inclusion $(n,t)\subset{(m,n,t)}$ and obtain the approximations         

\begin{align*}
r_{1,3}&=-1+\mathcal{O}(n),\\
r_{2,3}&=-m^{12}n^{12}+\mathcal{O}(n^{13}),\\
r_{3,3}&=-(m^{18}+m^{54})n^{54}+\mathcal{O}(n^{55}),\\
r_{4,3}&=\beta_{3,1}m^{72}n^{72}+\mathcal{O}(n^{73}), %
\end{align*}
where $\beta_{3,1}=-1/(1+m^{36})=-1/(1+v_{2})$.%
We can transfer these approximations to the algebra used in Case I by comparing the valuations of the roots. To illustrate the general procedure, we also show that the coefficient for $r_{4,2}$ corresponds to the one in $r_{4,3}$. We write 
\begin{equation}
r_{4,2}=(\dfrac{1-u_{1}}{u_{1}})t^{72}+\mathcal{O}(t^{73}).
\end{equation}
Using the M\"{o}bius relation $u_{1}=\dfrac{1+v_{2}}{v_{2}}$, we then find that $\dfrac{1-u_{1}}{u_{1}}=-1+\dfrac{1}{u_{1}}=\dfrac{-1}{1+v_{2}}=\beta_{3,1}$, so $r_{4,2}$ indeed corresponds to $r_{4,3}$.

\begin{figure}[h]
\begin{center}
\includegraphics[height=5.5cm]{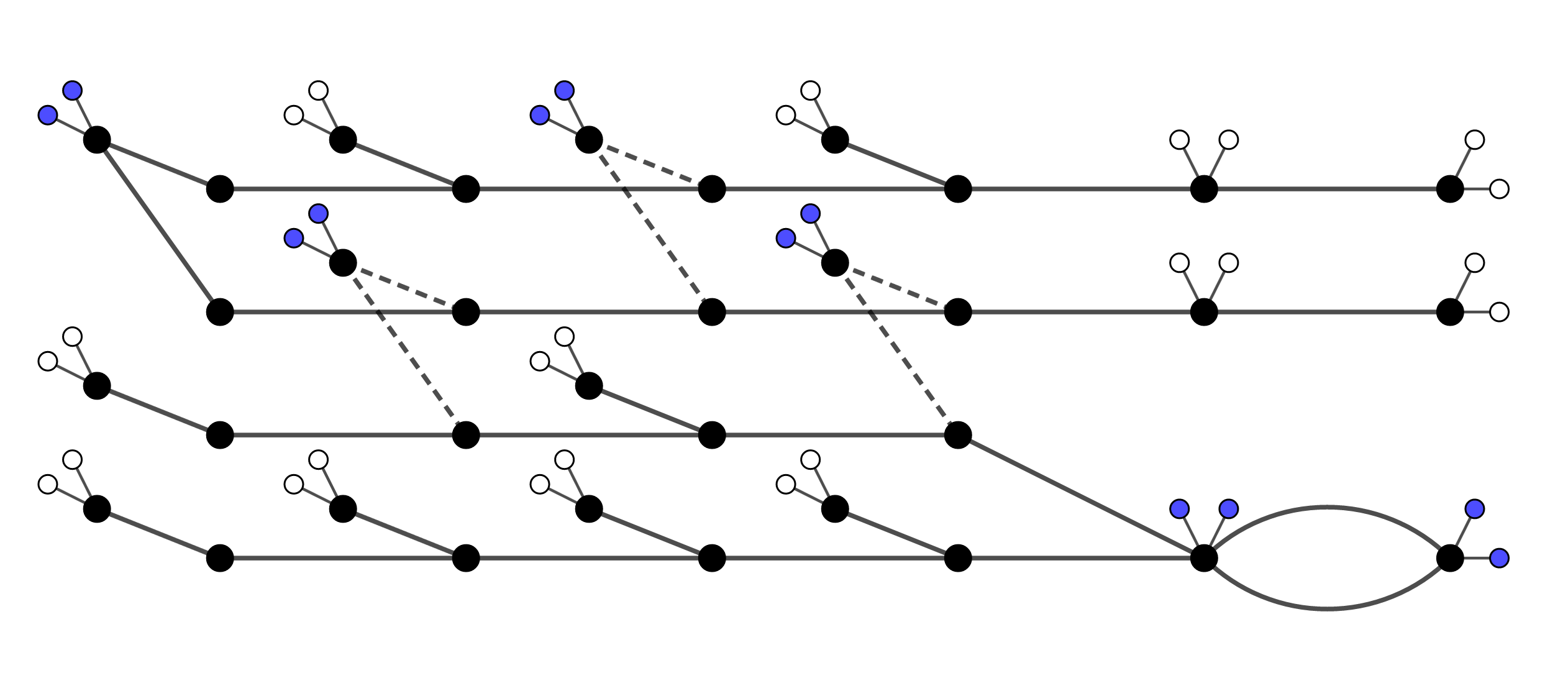}
\caption{\label{EndResultHannah1} The final graph in Example \ref{HannahExample1}. The corresponding metrized complex has a degree four covering to the tree in Figure \ref{SeparatingTreeHannah1}.}
\end{center}
\end{figure}

Continuing in this way for the other edges leads to the global skeleton as in Figure \ref{EndResultHannah1}. The skeleton has Betti number three, so we find that the curve $X/K$ is a Mumford curve. All the non-leaf edges in this picture have the same length as the corresponding edge in the separating tree, since the extension is \'{e}tale there. We note that the local covering data does not fix the global structure, as we can move around some of the loops without harming the local structure.

\end{exa}

\begin{rem}
If the induced covering is \'{e}tale over the special fiber, then we can work with algebras that are smaller than the regularization. That is, we can use subextensions $K(B)\subset{K(A_{\mathrm{reg}})}$ with $[K(B):K(A)]<n$ such that %
$f$ splits completely over the completion or Henselization of the prime lying over $\mathfrak{m}=(u,v,\pi)$. %
\end{rem}

\subsubsection{Calculating skeleta}\label{AlgorithmBerkovich}

We now combine the techniques in the previous sections and give an algorithm to calculate the skeleton of a curve $X/K$. %
For simplicity, %
we assume that $X$ is a plane curve given by $f(x,y)=0$ and that the projection map $(x,y)\mapsto{x}$ induces a finite separable covering $\phi: X\to\mathbf{P}^{1}$. The covering corresponds to the extension of function fields   %
$K(x)\subset{}K(x)[y]/(f)$. %
We now calculate a separating tree for the branch locus of $\phi$ as explained in Section \ref{P1Models} or in \cite[Chapter 6]{tropabelian}. To do this, it suffices to find a separating tree for $S=\{\infty\}\cup{Z(\Delta_{y}(f))}$, where $\Delta_{y}(f)$ is the $y$-discriminant of $f$ and $Z(\Delta_{y}(f))$ is its zero set over an algebraic closure. This gives a strongly semistable model $\mathcal{Y}$ of $\mathbf{P}^{1}$. We now assume that the normalization $\mathcal{X}$ of $\mathcal{Y}$ in $K(x)[y]/(f)$ is tame over $\mathcal{Y}$. We can then find a finite Kummer extension $K\subset{K'}$ such that %
the normalized base change of $\mathcal{X}$ has reduced special fiber. %
We again write $K=K'$ for this extension. %
We then find using Theorem \ref{MainThm1} that the normalization $\mathcal{X}$ is a strongly semistable model for $X$. %
Using Algorithm \ref{MixedNPAlgorithm3a}, we now calculate the $\eta_{v}$-adic and $\eta_{e}$-adic power series expansions of $f$ for every pair $(v,e)$ consisting of a vertex $v\in\Sigma(\mathcal{Y})$ and an adjacent edge $e$. %
We then determine the graph-theoretical structure of %
$\Sigma(\mathcal{X})$ by comparing the $\eta_{e}$-adic power series for overlapping edges to link the local coverings. %
The genera of the vertices are given by Remark \ref{GenusComponent} and the lengths of the edges are given by Remark \ref{LengthEdge1}. We summarize these steps in the following algorithm.   %

\begin{algorithm}
  \caption{An algorithm for calculating the skeleton of a covering of curves.}\label{FinalAlgorithm}
  \vspace*{0.1cm}
\textbf{Input:} A disjointly branched polynomial $f\in{K(x)}[y]$. \\
\textbf{Output:} {A tropicalization $\Sigma(\mathcal{X})\to\Sigma(\mathcal{Y})$ of the covering corresponding to $f$, together with the genera of the vertices and the lengths of the edges.}%
  \begin{algorithmic}[1]
  \State Calculate a separating tree $\Sigma(\mathcal{Y})$ of the covering using Section \ref{P1Models}. We denote the corresponding model for $\mathbf{P}^{1}_{K}$ by $\mathcal{Y}$. 
  \State For every pair $(v,e)$, where $v$ is a vertex in $\Sigma(\mathcal{Y})$ and $e$ is an adjacent edge, calculate the $\eta_{e}$-integral $\eta_{v}$-adic power series of the roots of $f$ using Algorithm \ref{MixedNPAlgorithm3a} up to a separating height. Here $\eta_{v}\subset{\eta_{e}}$ are the points corresponding to $v$ and $e$ in $\mathcal{Y}$. %
 \State  Calculate the pairs $(v',e')$ of vertices and adjacent edges lying above $(v,e)$ using   \ref{Dedekind3}, \ref{ExtensionsMinimalPolynomialsAlgorithm} and \ref{ActionEdges}.  %
\State For an edge $e$ adjacent to two vertices $v_{1}$ and $v_{2}$, identify the $\eta_{e}$-adic power series coming from the calculations for $(v_{i},e)$.
 \State Calculate the genus of a vertex $v'$ using Remark \ref{GenusComponent} and the length of an edge $e'$ using Remark \ref{LengthEdge1}. %
\State Use the data from Steps $3$, $4$ and $5$ to obtain the weighted metric graph $\Sigma(\mathcal{X})$.
  \end{algorithmic}
  \end{algorithm}
  
 The most time-consuming parts of the algorithm are Steps $2$ and $3$, where $\eta_{v}$-adic factorizations of the roots of $f$ are calculated. This problem seems unavoidable when calculating with coverings of curves. %

 \vspace*{0.1cm}

\center

\bibliographystyle{alpha}
\bibliography{bibfiles}{}
\end{document}